\documentclass[a4paper,11pt]{article}
\usepackage[english]{babel}
\usepackage{amsthm}
\usepackage{geometry}
\usepackage{graphicx}
\usepackage{mathtools}
\usepackage{amsmath}
\usepackage{amssymb}
\usepackage{tikz}
\usepackage{dsfont}
\usepackage{faktor}
\usepackage{mathrsfs}
\usepackage{xcolor}
\usepackage{stmaryrd}
\usepackage[utf8]{inputenc}
\usepackage{mathabx}

\DeclareMathOperator{\Bin}{Bin}

\DeclareMathOperator{\BP}{BP}
\DeclareMathOperator{\capa}{Cap}
\newcommand{\TV}{\mathrm{TV}}

\renewcommand{\geq}{\geqslant}
\renewcommand{\leq}{\leqslant}

\def\RR{\mathbb{R}}
\def\ZZ{\mathbb{Z}}
\def\NN{\mathbb{N}}
\def\PP{\mathbb{P}}
\def\EE{\mathbb{E}}

\def\ZZ{\mathbb{Z}}

\renewcommand{\P}{\mathrm{P}}
\newcommand{\E}{\mathrm{E}}
\newcommand{\bT}{\mathbf{T}}
\newcommand{\bt}{\mathbf{t}}
\newcommand{\bP}{\mathbf{P}}
\newcommand{\bE}{\mathbf{E}}
\newcommand{\bc}{\mathrm{a}}
\newcommand{\ind}{\mathds{1}}
\newcommand{\ex}{\mathrm{ex}}
\newcommand{\CM}{\mathrm{CM}}

\newcommand{\gb}{\beta}
\newcommand{\gep}{\varepsilon}       % \geq already exists...

\newcommand{\sumtwo}[2]{\sum_{\substack{#1 \\ #2}}} % sum with 2 lines
\theoremstyle{plain}

\newtheorem{theorem}{Theorem}[section]
\newtheorem{lemma}[theorem]{Lemma}
\newtheorem{proposition}[theorem]{Proposition}

\newtheorem{assumption}{Assumption}
\newtheorem{conjecture}{Conjecture}

\theoremstyle{definition}
\newtheorem{definition}[theorem]{Definition}

\newtheorem{example}[theorem]{Example}

\theoremstyle{remark}
\newtheorem{remark}[theorem]{Remark}

\definecolor{Orange}{HTML}{B26536}

\numberwithin{equation}{section}

\usepackage{hyperref}
\hypersetup{					% setup the hyperref-package options
	plainpages=false,			% 	- 
	colorlinks=true,			% 	- colorize links?
	pdfborder={0 0 0},			% 	-
	breaklinks=true,			% 	- allow line break inside links
	bookmarksnumbered=true,		%
	bookmarksopen=true			%
}

\usepackage{enumitem}
\setlist{itemsep=1pt,topsep=2pt,parsep=1pt, leftmargin=1.5em}

\let\OLDthebibliography\thebibliography
\renewcommand\thebibliography[1]{
  \small
  \OLDthebibliography{#1}
  \setlength{\parskip}{0pt}
  \setlength{\itemsep}{3pt plus 0.3ex}
}

\usepackage{authblk}

\begin{document}

\title{\vspace{-2cm} Ising model on a Galton--Watson tree\\ with a sparse Bernoulli external field}

\author[*]{Irene Ayuso Ventura}
\author[$\dagger$,$\ddagger$]{Quentin Berger}

\affil[*]{\footnotesize Laboratoire d'Analyse et de Math\'ematiques Appliqu\'ees, Universit\'e Paris-Est Cr\'eteil.}
\affil[$\dagger$]{\footnotesize Universit\'e Sorbonne Paris Nord, Laboratoire d'Analyse Géométrie et Applications.}
\affil[$\ddagger$]{\footnotesize Institut Universitaire de France.}

\date{}

\maketitle

\vspace{-1cm}
\begin{abstract}
We consider the Ising model on a supercritical Galton--Watson tree $\mathbf{T}_n$ of depth $n$ with a sparse random external field, given by a collection of i.i.d. Bernoulli random variables with vanishing parameter $p_n$. 
This may me viewed as a toy model for the Ising model on a configuration model with a few interfering external vertices carrying a plus spin: the question is to know how many (or how few) interfering vertices are enough to influence the whole graph.
Our main result consists in providing a necessary and sufficient condition on the parameters $(p_n)_{n\geq 0}$ for the root of $\mathbf{T}_n$ to remain magnetized in the large $n$ limit.
Our model is closely related to the Ising model on a (random) \textit{pruned} sub-tree $\mathbf{T}_n^*$ with plus boundary condition; one key result is that this pruned tree turns out to be an inhomogeneous, $n$-dependent, Branching Process. 
We then use standard tools such as tree recursions and non-linear capacities to study the Ising model on this sequence of Galton--Watson trees; one difficulty is that the offspring distributions of $\mathbf{T}_n^*$, in addition to vary along the generations $0\leq k \leq n-1$, also depend on~$n$.\\[3pt]
\textit{MSC2020 AMS classification: 60K35; 82B20; 82B44; 82B26}
\end{abstract}
\vspace{-0.4cm}

\begin{center}
	\begin{figure}[htbp]
	\centering
	\includegraphics[scale=0.25]{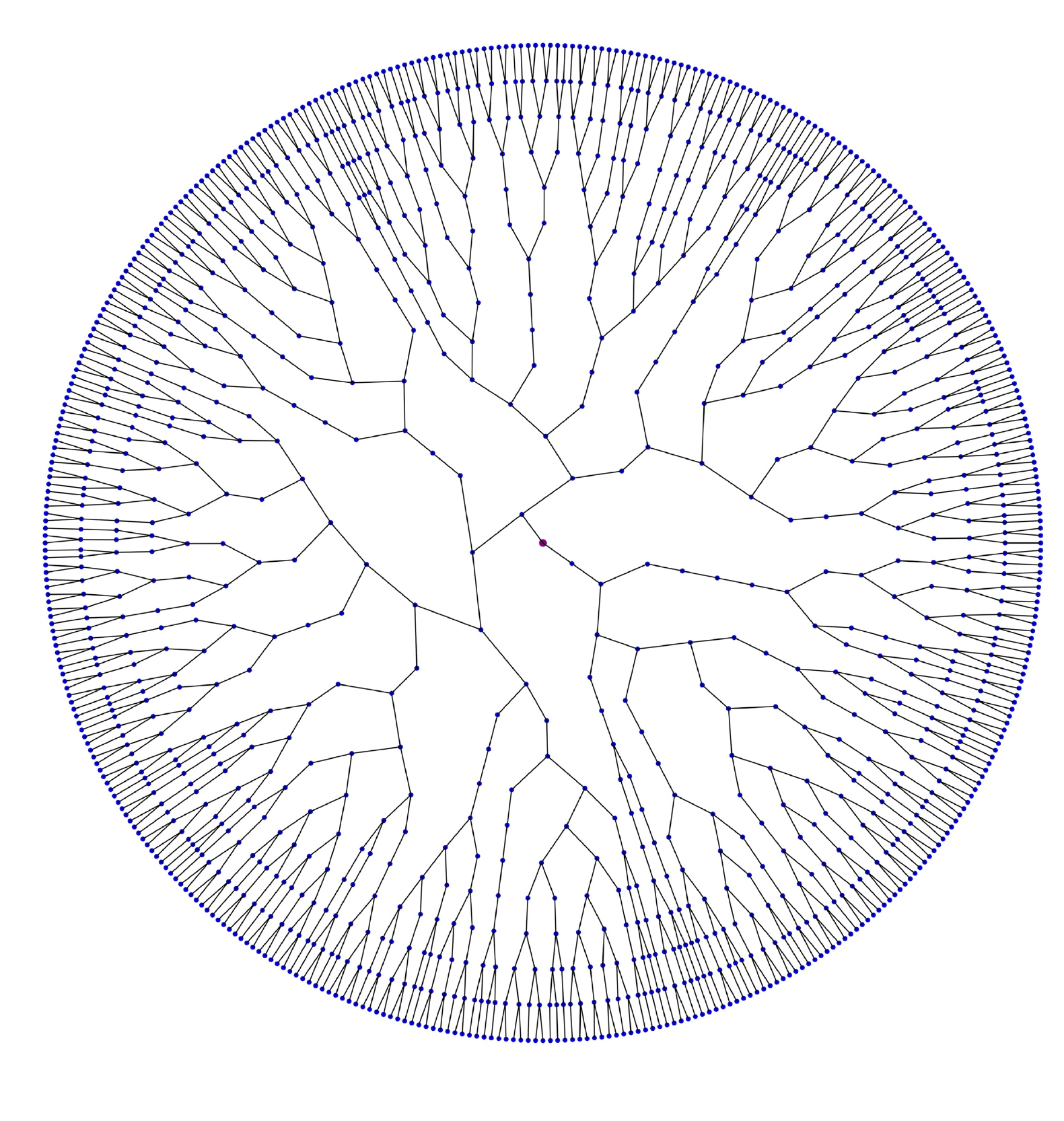}
		\qquad
		\includegraphics[scale=0.25]{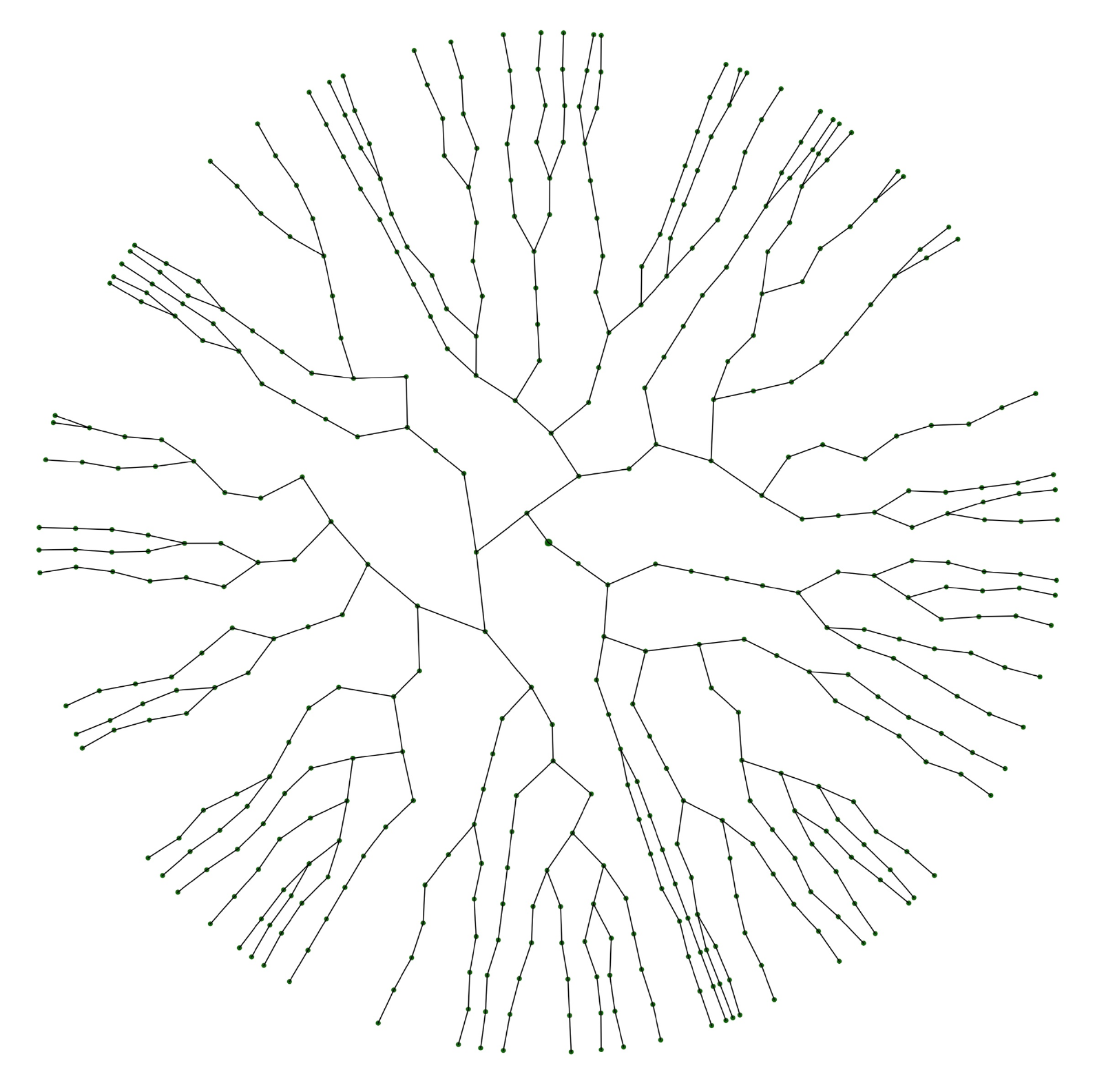}
		\caption{\footnotesize A tree of depth $n=15$ and its \textit{pruned} version by a Bernoulli process on its leaves.}
		\label{fig:firstpage}
	\end{figure}
	\end{center}

% commentaire

\setcounter{tocdepth}{1}
\tableofcontents

% \begin{figure}[htbp]
% \centering
% \includegraphics[scale=0.24]{treeStart.pdf}
% 	\qquad
% 	\includegraphics[scale=0.24]{treeEnd.pdf}
% 	\caption{\small A tree of depth $n=15$ and its \textit{pruned} version by a Bernoulli process on its leaves.}
% 	\label{fig:firstpage}
% \end{figure}

\section{Introduction of the model and main results}

The Ising model is a celebrated model, studied in depths for over 100 years. It was first introduced by Wilhelm Lenz and Ernst Ising as a model for magnetism, and was originally defined on regular lattices; we refer for instance to the books \cite{Bovier,FriedliVelenik} for a general introduction. 
Later, the Ising model was also studied on different types of graphs, starting with regular structures such as the Bethe lattice or Cayley tree, see~\cite{Rozikov} for an extensive summary on the subject. 
The seminal article~\cite{Lyons89}, in which Lyons identifies the critical temperature of the model on an arbitrary infinite tree, opened the way to the study of the Ising model and other statistical mechanics models on tree-like graphs, and, more generally, on random graphs (see for instance~\cite{Montanari2013}).

Motivated by the interest of the model to describe complex networks (see~\cite{Dor08} for a review), the literature on the Ising model on random graphs has grown considerably in recent years. 
Let us mention a few relevant results on the subject.
First, the thermodynamic limit of the Ising model on quenched random graphs has been studied in \cite{DemboMontGibbs10,DemboMontanari10,DomGiaRvdH} as well as its critical behavior in \cite{DomGiaRvdH12,GGvdHP2015}. 
Thereafter, the annealed Ising model has also gotten some attention, see for instance \cite{can2017critical,can2022annealed} references therein. 
More recently, the local weak limit of the Ising model on locally tree-like random graphs was considered in \cite{DemboBasak17,MontanariMosselSly}.

However, most of the literature on the Ising model on random graphs considers free or plus boundary conditions or a \emph{homogeneous} external field, but there does not seem to be many results when the boundary conditions or the external field are random (and depend on the size of the graph). 
In the present paper, we consider the Ising model on the simplest random graph possible, a Galton--Watson tree of depth $n$, but with a sparse random external field (which may be restricted to the boundary, which is close to being a boundary condition) whose distribution depends on $n$. 
We see our results as a first step towards the study of the Ising model on a random graph with a few external interfering vertices.

\subsection{General setting of the paper}

For a finite graph $G=(V,E)$, we consider the following Gibbs (ferromagnetic) Ising measure on spins $\sigma \in \{-1,+1\}^{|V|}$, with inverse temperature $\gb \in \mathbb{R}_+$ and external field $h= (h_v)_{v\in V} \in \mathbb{R}^{|V|}$:
\begin{equation}
\label{def:Ising}
\P_{G,\gb}^{h} (\sigma) := \frac{1}{Z_{G,\gb}^h} \exp\bigg( \gb \Big( \sumtwo{u,v \in V}{u\sim v} \sigma_u \sigma_v  + \sum_{v\in V} h_v \sigma_v  \Big)\bigg) \,,
\end{equation}
where we denoted $u\sim v$ if $\{u,v\} \in E$.
To simplify the statements and without loss of generality, we have assumed here that the coupling parameter is $J=1$.

In many cases, a natural  boundary of $V$, denoted $\partial V \subset V$, can be identified\footnote{For instance, if $G$ is a subgraph of a graph $\hat G = (\hat V,\hat E)$, the natural boundary is $\partial V = \{v \in V, \exists u \in \hat V \setminus V \text{ s.t. } u\, \hat{\sim}\, v\}$.
If $G$ is a finite tree, one usually takes $\partial V$ as the set of leaves; if $V = \{-n,\ldots, n\}^d$, one usually takes $\partial V = \{(v_1, \ldots, v_d) \in V \,, \exists i\ |v_i|=n\}$.}.
Then, we can consider the Ising model on the graph $G$  with boundary condition $\xi \in \{+1,-1\}^{\partial V}$ by considering the Gibbs measure
\begin{equation}
\label{def:Isingbc}
\P_{G,\gb}^{\xi} (\sigma) := \frac{1}{Z_{G,\gb}^h} \exp\bigg( \gb \Big(\sumtwo{u,v \in V\setminus \partial V}{u\sim v} \sigma_u \sigma_v  + \sumtwo{u \in V\setminus \partial V, v\in \partial V }{u\sim v} \xi_v \sigma_v \Big) \bigg) \,.
\end{equation}
For the Ising model with external field~\eqref{def:Ising}, in the case where $h_v=0$ for all $v \in V\setminus \partial V$, we will say that the Ising model has \emph{boundary external field}.

\begin{remark}[Exterior boundary]
\label{rem:boundary}
It might also be natural to consider an exterior boundary of $G$, denoted $\partial^{\ex} G = (\partial^{\ex} V, \partial^{\ex} E)$, where $\partial^{\ex} V$ is a set of external vertices (disjoint from $V$) and $\partial^{\ex} E$ is a set of boundary edges $\{x,y\}$ with $x\in V$, $y\in \partial^{\ex}V$.
We can then consider the Ising model on $G$ with \textit{exterior} boundary condition $\xi\in \{-1,+1\}^{\partial^{\ex} V}$ by considering the Gibbs measure~\eqref{def:Isingbc} on the graph $\bar G = ( \bar V, \bar E)$ with $\bar V =V \cup \partial^{\ex} V$, $\bar E =E\cup \partial^{\ex} E$ and with boundary condition $\xi$ on $\partial \bar V = \partial^{\ex} V$.

Notice that, in the definition~\eqref{def:Ising}, if the external field has value $h_v\in \ZZ$, we may interpret the external field as some exterior boundary condition, where the set $\{v, h_v\neq 0\}$ can be interpreted as the boundary.
Indeed, it corresponds to adding $|h_v|$ extra edges to~$v$, all leading to vertices with assigned value $\mathrm{sign}(h_v)$.
%For example, if $G =(V,E)$ is a subgraph of $\ZZ^d$, then setting $h_v= |\{u\in \ZZ^d\setminus V,  u\sim v\}|$ for all $v\in V$ (in particular $h_v=0$ is $v$ is not on the internal boundary), the definition~\ref{def:Ising} corresponds to considering the Ising model inside $G$ with plus boundary condition on its external boundary.
\end{remark}

\paragraph*{}
In the following, we focus on the Ising model~\eqref{def:Ising} with external field with $h_v\in \{0,1\}$; in fact, we will consider a random external field with either $(h_v)_{v\in V}$ or $(h_v)_{v\in \partial V}$ given by i.i.d.\ Bernoulli random variables of parameter $p \in (0,1)$.
With an abuse of terminology, this corresponds to  adding a plus (exterior) boundary condition to the vertices $v\in V$ with $h_v=+1$; it indeed corresponds to adding exactly one extra vertex $v'\sim v$ with boundary condition $\sigma_{v'}=+1$.
We refer to Figure~\ref{fig:graphboundary} for an illustration.

\begin{figure}[tbp]
\centering
\includegraphics[scale=0.9]{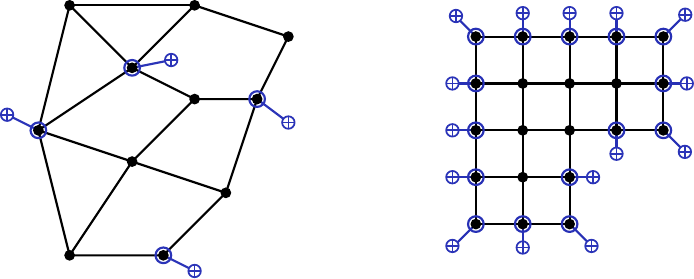}
\caption{\small Representation of graphs with external fields $h_v\in \{0,1\}$.
On each graph, we have circled the vertices where $h_v =+1$.
The graph on the left has no natural boundary: we can think of the vertices with $h_v=+1$ as the boundary, each vertex being connected to a `$+$' through an extra edge.
On the right, the graph is a subset of $\mathbb{Z}^2$ and $h_v=+1$ for all $v \in \partial V$: this is our \emph{boundary external field} and in the present case it does not exactly correspond to the Ising model with (exterior) plus boundary condition, because the corner vertices are connected to a `$+$' through only one extra edge (to obtain plus boundary condition one should take $h_v =+2$ for the corner vertices).}
\label{fig:graphboundary}
\end{figure}

One key physical quantity that we are going to study is the magnetization of a vertex~$v\in V$:
\[
\mathfrak{m}_{G,\gb}^h(v) := \E_{G,\gb}^{h}[\sigma_v] \,.
\]
A closely related quantity is the probability that a given spin is in the plus state, namely $\P_{G,\gb}^{h}(\sigma_v=+1) =\frac12 (1+\mathfrak{m}_{G,\gb}^h(v))$ and the following log-likelihood ratio
\[
r_{G,\gb}^h(v) = \log \bigg( \frac{\P_{G,\gb}^{h}(\sigma_v=+1)}{\P_{G,\gb}^{h}(\sigma_v=-1)}\bigg) 
= \log \bigg( \frac{1+\mathfrak{m}_{G,\gb}^h(v) }{ 1-\mathfrak{m}_{G,\gb}^h(v)}\bigg) \,.
\]
Considering a sequence of growing graphs $(G_n)_{n\geq 1}$ with associated non-negative external fields (that may depend on $n$), we say that there is spontaneous magnetization at inverse temperature $\beta$ if, choosing for each $n$  a vertex $v_n$ uniformly at random in the graph,the magnetization $\mathfrak{m}_{G_n,\gb}^h(v_n)$ (or equivalently the log-likelihood ratio $r_{G_n,\gb}^h(v_n))$ remains bounded away from $0$ as $n\to\infty$ (either almost surely or in probability).

\subsection{Ising model on the Configuration Model with interfering vertices}
\label{sec:ConfigModel}

Let us now introduce one of our main motivation for considering a random sparse external field on a Galton--Watson tree: the Ising model on a random graph given by the configuration model, with a small proportion of additional interfering individuals.

The Configuration Model is a random graph in which edges are places randomly between vertices whose degrees are fixed beforehand.
We refer to \cite[Ch.~7]{vdHRandomGraphs} for a complete introduction but let us briefly present the construction. 
Let $N$ be the number of vertices of the graph, and let $\mathbf{d} = (d_i)_{i \in \llbracket N \rrbracket}$ be a sequence of degrees, verifying $d_i\geq 1$; we also use the notation $\llbracket N \rrbracket = \{1,\ldots, N\}$.
Then, the configuration model, noted $\CM_N (\mathbf{d})$ is an undirected random (multi)graph such that each vertex $i \in \llbracket N \rrbracket$ has degree $d_i$ (self-loops and multiple edges between pairs of vertices are allowed).
It is constructed inductively as follows.
As a preliminary to the construction, attach $d_i$ half-edges to each vertex vertex $i \in \llbracket N \rrbracket$, so that there is a collection $\mathcal{C}_0$ of $\ell_N := \sum_{i=1}^N d_i$ available half-edges.
Then, construct the first edge of the graph by choosing two half-edges uniformly at random from $\mathcal{C}_0$ and by pairing them; afterwards, remove these two half-edges from the set $\mathcal{C}_0$.
After this first step, the new set $\mathcal{C}_1$ of available half-edges contains $\ell_N-2$ elements. 
This procedure is iterated $\frac12\ell_N$ times, until there are no more half-edges available; notice that $\ell_N$ must be even.

%\begin{figure}[h] %pour qu'il soit exactement ou on l'a mis
%\centering
%\includegraphics[scale=0.1]{ConstructionCM.JPG}
%\caption{Construction of the configuration model $CM_4(\mathbf{d})$, where $\mathbf{d} = (4, 3, 3, 2)$.}
%\label{figure1}
%\end{figure} 

Let us stress that the Configuration Model has no natural boundary, but one may think of having a few additional external vertices that are ``interfering'' with the graph.
To model this, add $M_N$ vertices to the initial $N$ vertices of the model (we think of having $M_N\ll N$), all with degree $1$, and call these extra vertices \textit{interfering}.
One can then proceed to construct the graph as described above, \textit{i.e.}\ a configuration model with both original and interfering vertices\footnote{The fact that interfering vertices have degree $1$ ensures that these vertices cannot interfere with each other; but one could naturally consider a degree sequence $(\tilde d_i)_{i\in \llbracket M_N \rrbracket}$ for these vertices.}. 
Notice that, even if interfering vertices have degree one, an original vertex might have more than one interfering vertex attached to it.
The $M_N$ interfering vertices may also be interpreted as some (external) boundary of the graph: in the context of the Ising model, one may consider the model where interfering vertices all have a plus spin, and try to determine a condition whether there is spontaneous magnetization on a sequence of configuration models, depending on $M_N,N$ and the degree sequence $\mathbf{d}$.

In the case where $M_N \leq N$, another natural (and closely related) way of adding interfering external vertices is to consider the graph $\CM_N(\mathbf{d})$ and, to each vertex $i \in \llbracket N \rrbracket$, attach an extra (interfering) vertex of degree one, with probability $p_N := M_N / N$. 
One obvious difference from the previous construction is that each vertex has at most one interfering vertex attached; also, in the first construction, interfering vertices change the distribution of the original configuration model. 
Indeed, an interfering vertex is taking over from an original one in the first construction, instead of just being added afterwards, as in the second construction.
%(an interfering vertex is taking over from an original one instead of just being added to the model).
In the context of the Ising model, this version corresponds to having a (random) external field given by the spin of interfering vertices, say  i.i.d.\ Bernoulli.
Here again, one may ask whether there is spontaneous magnetization on a sequence of configuration models, depending on $p_N,N$ and the degree sequence $\mathbf{d}$.

To summarize, the general question is to determine how many (or how few) interfering vertices are enough to have some influence on a random individual.
Since the configuration model rooted at some randomly chosen vertex $o$ locally behaves like a branching process see~\cite[Sec.~4.2]{vdHRandomGraphs2} (and also Section~\ref{sec:backtoCM} below), we  consider in the present paper, as a toy model, the Ising model on a Galton--Watson tree with randomly attached interfering vertices, \textit{i.e.}\ a sparse external Bernoulli field.

%When working on the Ising model, we will set the spin of the $M_N$ ``boundary vertices" to be $+1$, and we will consider $p_N$ to converge to zero: we will denote this as random sparse ``+" boundary conditions. 

%We want to understand whether a vertex chosen at random in the graph is magnetized. Put otherwise, we want to know how many ``boundary vertices'' one needs to add so that the graph feel their influence.

\subsection{Ising model on a Galton--Watson tree: main results}
\label{sec:results}

Let $\mu$ be a distribution on $\NN$ and consider $\bT_n$ a random tree of depth $n \in \NN$ generated by a Branching Process with offspring distribution~$\mu$, stopped at generation~$n$; we will write $\bT=\bT_n$ if there is no confusion possible.
We will denote by $\rho$ its root and by $\partial \bT$ the set of its leaves. Also, we denote by $\bP$ the law of $\bT$, and we make the following assumption, which ensures in particular that the tree is super-critical.

\begin{assumption}
\label{hyp:branching}
The offspring distribution $\mu$ satisfy $\mu(0)=0$ and $\mu(1)<1$. In particular, $\nu := \sum_{d=1}^{\infty} d\mu(d) >1$.
\end{assumption}

The condition $\mu(0)=0$ ensures that a.s.\ there is no extinction, so the tree $\bT_n$ a.s.\ reaches depth~$n$, and it also ensure that it has no leaves except at generation~$n$. We could weaken this assumption and work conditionally on having no extinction, but we work with Assumption~\ref{hyp:branching} for technical simplicity; we refer to Section~\ref{sec:mu0} below for more comments.

We consider the Ising model~\eqref{def:Ising} on a tree $\bT$ with different possibilities for the external field or boundary condition:
\begin{itemize}
\item[\textbf{(a)}] With `$+$' boundary condition, meaning that $\sigma_v =+1$ (corresponding to an external field $h_v=+\infty$) if $v\in \partial \bT$; the results are well-known since the work of Lyons~\cite{Lyons89}, see Theorem~\ref{thm:Lyons} below.
We denote the Ising Gibbs measure $\P_{n,\gb}^+$ in this case.
\item[\textbf{(b)}] With (sparse) Bernoulli external field, in which $(h_v)_{v\in \bT}$ are i.i.d.\ Bernoulli random with parameter $p_n$, independent of $\bT$, whose law we denote by $\PP$.
%; one could think of $(h_v)_{v \in \bT}$ as part of a countable collection $(h_v)_{v\in \mathcal{V}}$, whose law 
We denote by~$\P_{n,\gb}^{(p_n)}$ the Ising Gibbs measure in this case.
\item[\textbf{(c)}] With (sparse) Bernoulli boundary external field, in which $(h_v)_{v\in \bT}$ are i.i.d.\ Bernoulli variables with parameter $\hat p_n := p_n \ind_{\{v\in \partial \bT\}}$; again, we denote their law by $\PP$, by a slight abuse of notation.
We denote by~$\P_{n,\gb}^{(\hat p_n)}$ the Ising Gibbs measure in this case.
\end{itemize}

\begin{figure}[tbp]
\centering
\includegraphics[scale=0.9]{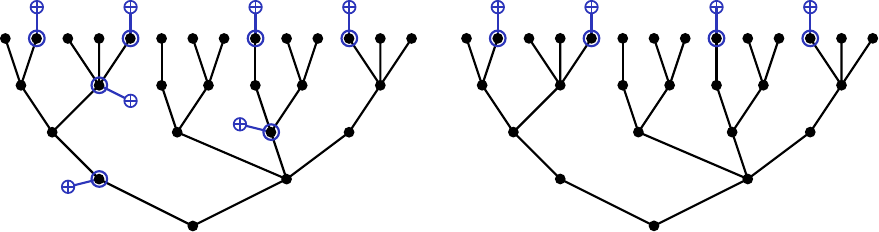}
\caption{\small Two Ising models on a Galton--Watson tree with Bernoulli external field $h_v\in \{0,1\}$: the vertices with $h_v=+1$ are circle and a `$+$' has been added to them.
On the left, the external field $h_v \in \{0,1\}$ lives inside the whole tree (model $\P_{n,\gb}^{(p_n)}$); on the right, the external field $h_v\in\{0,1\}$ lives only on the leaves (model $\P_{n,\gb}^{(\hat p_n)}$).}
\end{figure}

\begin{remark}
The second model \textbf{(b)} mimics the Ising model on a configuration model with sparse interfering vertices\footnote{Notice that this is not exactly how the local limit of the Ising model on the configuration model would look like. For instance, the configuration model converges locally to an uni-modular branching process, that is, the root has a different offspring distribution from the rest of the vertices.}; the third model \textbf{(c)} is interesting in itself and will serve as a point of comparison.
We stress that the parameter $p_n$ depends on the depth~$n$ of the tree, but is constant among vertices in the tree (no matter their generation).
\end{remark}

We then consider the \emph{root magnetization} with external field (or boundary condition) $\bc\in \{+,(p_n),(\hat p_n)\}$:
\begin{equation}
\mathfrak{m}_{n,\gb}^{\bc} := \mathfrak{m}_{n,\gb}^{\bc} (\rho) := \E_{n,\gb}^{\bc}[\sigma_\rho] \,,
\end{equation}
which is a random variable that depends on the realization of the tree $\bT$ and of the field $(h_v)_{v\in \bT}$ (if the latter is random).
We say that the \textit{root is asymptotically (positively) magnetized} for the model $\bc \in \{+, (p_n),(\hat p_n)\}$ at inverse temperature $\beta$ if 
\begin{equation}
\label{eq:rootmagnetization}
\lim_{\gep\downarrow 0}\liminf_{n\to\infty} \bP\otimes \PP\big( \mathfrak{m}_{n,\gb}^{\bc} >\gep \big)  =1 \,,
\end{equation}
where $\bP\otimes \PP$ denotes the joint law of $\bT$ and $(h_v)_{v\in \bT}$.
Conversely, the root is \textit{asymptotically not magnetized} if $\mathfrak{m}_{n,\gb}^{\bc}$ goes to $0$ in $\bP\otimes \PP$-probability.
Finally, note that these statements are equivalent if we replace the root magnetization by the log-likelihood ratio $r_{n,\gb}^{\bc} := r_{n,\gb}^{\bc}(\rho)$ of the root.

Obviously, one can compare the three models \textbf{(a)}, \textbf{(b)}, \textbf{(c)} above, \textit{i.e.}\ external field or boundary condition $\bc\in\{+,(p_n),(\hat p_n)\}$: indeed, by the Griffiths--Kelly--Sherman (GKS) inequality~\cite{griffiths67,KellySherman68}, we have that $\mathfrak{m}_{n,\gb}^{+} \geq \mathfrak{m}_{n,\gb}^{(\hat p_n)}$ and $ \mathfrak{m}_{n,\gb}^{(p_n)}\geq \mathfrak{m}_{n,\gb}^{(\hat p_n)}$.
We now state our main results.

\paragraph*{(a) With a plus boundary condition on the leaves.}
First of all, we recall the seminal result from Lyons~\cite{Lyons89} about the phase transition of the Ising model on a tree; we state it here only in our simpler context of a Galton--Watson tree.
 Note that in this context of plus boundary condition, the magnetisation of the root $\mathfrak{m}_{n,\gb}^+$ is non-increasing in $n$, so the asymptotic magnetization $\mathfrak{m}^+_{\gb} := \lim_{n\to\infty} \mathfrak{m}_{n,\gb}^+$ is well-defined.

\begin{theorem}[\cite{Lyons89}]
\label{thm:Lyons}
Consider the Ising model on a Galton--Watson tree $\bT$ with plus boundary condition. Then we have root asymptotic magnetization~\eqref{eq:rootmagnetization} (with $\bc=+$) at inverse temperature $\beta$ \emph{if and only if} $\nu \tanh(\gb) >1$, where we recall that $\nu$ is the mean offspring number.
\end{theorem}

The general result in~\cite{Lyons89} holds for a generic infinite tree $\bT$: there is root magnetization if and only if $\mathrm{br}(\bT) \tanh(\gb) >1$, where $\mathrm{br}(\bT)$ is the so-called branching number of the tree $\mathbf{T}$ which encodes some of its growth properties; we do not give a formal definition here but we refer to~\cite[\S3.2]{LyonsPeres}, and we stress that we have $\mathrm{br}(\bT) = \nu$ for branching processes.
In other words, Theorem~\ref{thm:Lyons} identifies the critical temperature for the Ising model on a tree, $\gb_c := \inf\{\beta \mid \mathfrak{m}^+_{\gb} >0 \} = \tanh^{-1}(\frac{1}{\mathrm{br}(\bT)})$; it also shows that there is no root magnetization at $\gb=\gb_c$.

\paragraph*{(b) With a sparse Bernoulli external field inside the tree.}
%Let us now turn to the case of a sparse external field ($\bc=(p_n)$).
Let $(p_n)_{n\geq 0}$ be a sequence of parameters in $[0,1]$. For each $n$, we let $\bT_n$ be a GW tree up to generation $n$ and we let $(h_v)_{v\in \bT_n}$ be i.i.d.\ Bernoulli random variables of parameter~$p_n$, independent of~$\bT_n$.
Then, we have the following result, which is the main goal of this article.

\begin{theorem}
\label{thm:main}
Suppose that Assumption~\ref{hyp:branching} holds and that $\mu$ has a finite moment of order $q$ for some $q>1$.
Consider the Ising model on $\bT_n$ with sparse Bernoulli external field, with $\lim_{n\to\infty} p_n = 0$. Then we have asymptotic root magnetization~\eqref{eq:rootmagnetization} (with $\bc=(p_n)$) at inverse temperature~$\beta$ \emph{if and only if} 
\[
\nu \tanh(\gb) >1 \qquad \text{ and }\qquad  \liminf_{n\to\infty} (\nu \tanh(\gb))^n p_n >0 \,,
\]
where $\nu$ is the mean offspring number.
In the case where $\liminf_{n\to\infty} p_n >0$, the root is asymptotically magnetized.
\end{theorem}

In other words, this theorem gives the exact speed at which $(p_n)_{n\geq 0}$ should decrease in order not to have root magnetization.
The first condition in the theorem comes from Lyons' Theorem~\ref{thm:Lyons}; the second condition shows that the sparsity of the Bernoulli field may somehow shift the critical point.
For instance, if $p_n = \alpha^n$ for some $\alpha \in (0,1)$, then one has root magnetization if and only if $\alpha \nu \tanh(\gb) \geq 1$, so the new critical inverse temperature is $\gb_c(\alpha)  \tanh^{-1}(\frac{1}{\alpha\nu})$. 
Note that in that case, the root is also magnetized at the critical temperature, contrary to what happens in Theorem~\ref{thm:Lyons}; this is due to the fact that the model can be compared to an Ising model on a modified tree and the reason behind the first-order transition is the structure of this modified tree, and more precisely, how its $p$-capacity behaves (see Remark~\ref{rem:order} for a heuristic computation).

\paragraph*{(c) With a sparse Bernoulli boundary external field on the leaves.}
We have a similar result when we put the sparse Bernoulli field only on the leaves.
As above, for each $n$, we let $\bT_n$ be a GW tree up to generation $n$ and we let $(h_v)_{v\in \bT_n}$ be i.i.d.\ Bernoulli random variables of parameter~$\hat p_n = p_n \ind_{\{v\in \bT_n\}}$, independent of $\bT_n$.

\begin{theorem}
\label{thm:second}
Suppose that Assumption~\ref{hyp:branching} holds and that $\mu$ has a finite moment of order $q$ for some $q>1$.
Consider the Ising model on $\bT_n$ with sparse Bernoulli boundary external field , with $\lim_{n\to\infty} p_n =0$.
Then we have asymptotic root magnetization~\eqref{eq:rootmagnetization} (with $\bc=(\hat p_n)$) at inverse temperature~$\beta$ \emph{if and only if} 
\[
\nu \tanh(\gb) >1 \qquad \text{ and }\qquad  \liminf_{n\to\infty} (\nu \tanh(\gb))^n p_n >0 \,,
\]
where $\nu$ is the mean offspring number.
In the case where $\liminf_{n\to\infty} p_n >0$, the root is asymptotically magnetized if and only if $\nu \tanh(\gb) >1$.
\end{theorem}

\begin{remark}
We could obviously consider a more general sparse i.i.d.\ random external field. For instance if $(h_v)_{v\in \bT}$ are i.i.d.\ with distribution $(1-p_n)\delta_0 + p_n \mu_Y$ for some positive random variable $Y$, one can easily compare this model with a Bernoulli external field and obtain identical results (provided that $\mathbb{E}[Y]<+\infty$). We have chosen to focus on Bernoulli external fields for the simplicity of exposition.
\end{remark}

\subsection{Outline of the proof and organisation of the paper}

Let us outline our strategy of proof, which relies on standard tools for the Ising model on trees, namely Lyons' iteration~\cite{Lyons89} for the log-likelihood ratios (we recall it in Section~\ref{sec:Lyons}, see~\eqref{eq:Lyonsiteration}), and Pemantle--Peres~\cite{PemPer10} relation between the log-likelihood ratio and the non-linear $\frac32$-capacity of the tree, see Section~\ref{sec:PemPer} for an overview.

After a few preliminaries in Section~\ref{sec:prelim}, we prove in Section~\ref{sec:upper} the upper bound on the magnetization of the root.
More precisely, we use Lyons' iteration to derive an upper bound on the log-likelihood ratios for the Ising model with Bernoulli external field inside the whole tree (which dominates the case where the external field is only on the leaves).

Then, the starting point of the lower bound is presented in Section~\ref{sec:lower}.
First, we show in Section~\ref{sec:comparePruned} that the log-likelihood ratio of the root for the Ising model with Bernoulli boundary external field is equivalent to that of the Ising model with plus boundary external field on a modified tree, that we call \emph{pruned tree} since it corresponds to removing all branches that do not lead to some $h_v=+1$.
In other words, the Ising model with random (Bernoulli) boundary external field on $\bT_n$ corresponds to an Ising model with plus boundary condition on a random subtree $\bT_n^*$ of $\bT_n$, that we call the \emph{pruned tree}.
Then, we explain in Section~\ref{sec:PemPer} how one can relate the magnetization of the root of a tree for the Ising model with plus boundary condition to some $p$-capacity of the tree equipped with specific resistances, see Theorem~\ref{thm:PemPer}.
The rest of the paper then focuses on estimating the $p$-capacity of the pruned tree~$\bT^*_n$: this is summarized in Proposition~\ref{prop:capacitypruned}.
Section~\ref{sec:pruning} consists in studying fine properties of the pruned tree, 
which are crucially used in Section~\ref{sec:capacitypruned} to derive the $p$-capacity estimates.

More specifically, in Section~\ref{sec:pruning}, we show that under $\mathbf{P}\otimes \mathbb{P}$ the pruned tree is actually an \textit{inhomogeneous} Branching Process, whose offspring distributions $(\mu_k^*)_{0\leq k \leq n-1}$ are explicit (see~\eqref{defmuk}) and depend on the generation $k$ but also on the depth~$n$ of the tree --- in other words, we have a triangular array of offspring distributions.
Then, we show in Sections~\ref{sec:parameters}-\ref{sec:shape} that the pruned tree somehow exhibits a sharp phase transition: there exists some $k^*:= \log(p_n \nu^n)/\log \nu$ (which depends on~$n$ and go to $+\infty$ if $\liminf_{n\to\infty} p_n (\nu\tanh \gb)^n>0$), such that:
\begin{itemize}
\item if $k^*-k$ is large, then $\mu_k^*$ is very close to the original offspring distribution $\mu$;

\item if $k-k^*$ is large, then $\mu_k^*$ is very close to being a Dirac mass at $1$.
\end{itemize}
This statement is made precise (and quantitative) in Proposition~\ref{prop:d_TV}. 
Section~\ref{sec:shape} contains several additional technical results quantifying this phase transition for various quantities of the pruned tree (for instance the mean and variances of $\mu_k^*$), which are important for the last part of the proof.

\smallskip
Section~\ref{sec:capacitypruned} concludes the proof of the lower bound on the log-likelihood ratio in Theorem~\ref{thm:second} (which implies the lower bound in Theorem~\ref{thm:main}) by establishing Proposition~\ref{prop:capacitypruned}.
The proof relies on Thomson's principle, which, after a few technical estimates and some moment condition, allows us to obtain a lower bound for $\capa_{p}(\bT^*_n)$, thereby completing the proof.
For completeness, we also provide an upper bound on $\capa_{p}(\bT^*_n)$ in Proposition~\ref{prop:capacitypruned} (without requiring moment conditions); throughout the proof, we aim to present a general framework for estimating the $p$-capacity of any inhomogeneous Galton--Watson tree.
%As a side result of independent interest, we find for instance that for the Ising model on a Galton--Watson tree (with offspring distribution of finite variance) with plus boundary condition, then at the critical temperature $\gb_c= \tanh^{-1}(\frac1\nu)$ the magnetization of the root lies between $n^{-1}$ and~$n^{-1/2}$, see Remark~\ref{rem:critical}, which  seems to be a new result (we believe that the correct order is $n^{-1/2}$, at least if the offspring distribution admits enough moments, but we leave this to another work since it is not the main focus of the paper). 

\subsection{Some further comments}

Let us now conclude this section with several comments on our results, suggesting for instances possible directions for further investigations.

\subsubsection{Comparison with a homogeneous but vanishing external field}

A first natural question is to know whether the same results as in Theorems~\ref{thm:main}-\ref{thm:second} would hold if one replaced the random Bernoulli external field $(h_v)_{v\in \bT}$ by its mean $(\mathbb{E}[h_v])_{v\in \bT}$.
This corresponds to considering the Ising model on a tree with a \emph{homogeneous} but vanishing external field (either inside the whole tree or only on the boundary).

In particular, we want to look at the two following (sequences of) Ising models on the tree $\bT_n$ of depth~$n$:
(i)~homogeneous external field $h_v\equiv p_n$ for all $v\in \bT_n$.
(ii)~homogeneous boundary external field $h_v\equiv  p_n$ for all $v\in \partial \bT$.

Using Lyons iteration as in Section~\ref{sec:upper} would yield that there is no root magnetization in the first model (hence in the second model) whenever $p_n\to 0$ and $p_n (\nu\tanh(\gb))^n \to 0$; we leave this as an exercise to the reader. 
Using our method, one could actually prove that if $p_n\to 0$ and $\liminf_{n\to\infty} p_n (\nu\tanh(\gb))^n >0$, then the root is asymptotically magnetized in the second model (hence in the first one). 

Let us explain how one should adapt our proof in Section~\ref{sec:capacitypruned} to derive such a result.
The main observation is that the second model can be compared with an Ising model with plus boundary condition on some modified \emph{elongated} tree $\tilde{\bT}_n$, where~$\tilde{\bT}_n$ is obtained from $\bT_n$ by adding ``straight branches'' of length $\ell = \log p_n/\log \tanh(\gb)$ to each leaf of $\bT_n$, \textit{i.e.}\ $\tilde{\bT}_n$ is an inhomogeneous Branching Process of depth $n+\ell$ with offspring distribution $\mu_k=\mu$ for generations $k\leq n-1$ and $\mu_k =\delta_{1}$ for $n\leq  k \leq n+\ell-1$.
Indeed, by Lyons' iteration, the log-likelihood ratio of vertices at generation $n$ of this elongated tree is proportional to $\tanh(\gb)^{\ell} = p_n$ (by our choice of $\ell$); this corresponds to an Ising model on~$\bT_n$ with boundary external field proportional to $p_n$.

Then, thanks to Pemantle--Peres~\cite{PemPer10} comparison theorem (see Theorem~\ref{thm:PemPer}), the problem is reduced to estimating the $\frac32$-capacity of $\tilde{\bT}_n$, which presents the same type of ``phase transition'' as the pruned tree.
For this, one may use similar calculations as in Section~\ref{sec:capacitypruned}. In particular, the same estimates obtained for the $p$-capacity of our pruned tree $\bT_n^*$ (cf.~Proposition~\ref{prop:capacitypruned}) hold for the elongated tree~$\tilde{\bT}_n$. 
One would therefore end up with the following result:
\begin{quote}
Consider the Ising model on the Galton--Watson tree $\bT_n$ with either (i)~homogeneous external field ${h_v\equiv p_n}$ for $v\in \bT_n$ or
(ii)~homogeneous boundary external field $h_v\equiv  p_n$ for $v\in \partial \bT_n$, with $\lim_{n\to\infty} p_n = 0$. Then, under the assumptions of Theorems~\ref{thm:main}-\ref{thm:second}, we have asymptotic root magnetization at inverse temperature~$\beta$ \emph{if and only if} 
\[
\nu \tanh(\gb) >1 \qquad \text{ and }\qquad  \liminf_{n\to\infty} (\nu \tanh(\gb))^n p_n >0 \,.
\]
\end{quote}

We do not develop further on this issue since it is not the main purpose of our paper, but this is a simpler model and hides no extra difficulty (in fact, $\tilde{\bT}_n$ is much easier to study than the pruned tree $\bT_n^{*}$), so we leave it as an exercise.

\subsubsection{About Assumption~\ref{hyp:branching} on the offspring distribution}
\label{sec:mu0}

Obviously, there are a few limitations on the generality of the offspring distribution~$\mu$ that we consider, the most important one being that $\mu(0) =0$, so that there is no ``internal leaf'' in $\bT$.
We have chosen to work with this condition, in analogy with the works of Lyons~\cite{Lyons89} and Pemantle--Peres~\cite{PemPer10}; this also simplifies the presentation.

One could in fact remove the assumption~$\mu(0)=0$ by considering a super-critical Galton--Watson tree $ \bT$, \textit{i.e.}\ assuming $\nu>1$, conditioned on survival: let us denote~$\widecheck{\bT}$ such a tree, and $\widecheck{\bT}_n$  its restriction to the first $n$ generations.
The tree $\widecheck{\bT}$ is infinite, its $n$-th generation growing like $\nu^n$, and there is an explicit decomposition of its law that we now  briefly recall; we refer to \cite[\S5.5]{LyonsPeres} for details.
If $G$ is the generating function of $\mu$, then the extinction probability $\varrho$ is the smallest fixed point of $G$, and  one can construct two reproduction laws~$\hat \mu$, $\tilde \mu$ with respective generating functions $\hat G(s) = \frac{1}{1-\varrho} [G(\varrho + (1-\varrho)s)-\varrho]$ and $\tilde G(s) = \frac{1}{1-\varrho} G(\varrho s)$.
Note that~$\hat \mu$ verifies Assumption~\ref{hyp:branching} with mean $\hat G'(1) = G'(1)=\nu$, and that~$\tilde \mu$ is a subcritical reproduction law.
Then, \cite[Prop.~5.23]{LyonsPeres} shows that the tree $\widecheck{\bT}$ can be constructed as follows:
\begin{itemize}
\item[(i)] consider a Galton--Watson tree $\hat \bT$ with reproduction law $\hat \mu$ (\textit{i.e.}\ super-critical, with no internal leaf);
\item[(ii)] independently, to each leaf $v\in \hat \bT$, attach a (random) number $U_v$ subcritical Galton--Watson trees $\tilde \bT$ with reproduction law $\tilde \mu$, where $U_v$ has an explicit distribution (that depends on the number of children of~$v$ in~$\hat \bT$).
\end{itemize}

Going back to our problem, Theorems~\eqref{thm:main} and \ref{thm:second} also hold with $\widecheck{\bT}$ in place of $\bT$.
Indeed, one can check that the upper bound that we present in Section~\ref{sec:upper}, which simply relies on a tree iteration, is still valid.
For the lower bound, thanks to the GKS inequality, one can simply consider the Ising model with sparse Bernoulli external field on $\hat \bT_n$ instead of $\widecheck{\bT}_n$.
This way, one falls back to a model for which Assumption~\ref{hyp:branching} is verified, with the mean of the offspring distribution also equal to $\nu$: we can then simply apply Theorems~\ref{thm:main} and~\ref{thm:second}.

\subsubsection{Starting with an inhomogeneous, $n$-dependent, Galton--Watson tree}

Keeping in mind our application to the Configuration Model (see Section~\ref{sec:backtoCM} below), let us mention that our method appear to be robust to the case where the initial Galton--Watson tree $\bT_n$ of depth $n$ is homogeneous but with an offspring distibution $\mu^{(n)}$ that depends on $n$.
In particular, we believe that our results hold simply by replacing $\nu$ with~$\nu^{(n)}$, the mean of the offspring distibution~$\mu^{(n)}$, provided that the means $(\nu^{(n)})_{n\geq 0}$, resp.\  the variances $(\sigma_{\mu^{(n)}}^2)_{n\geq 0}$, remain bounded away from $1$, resp.\ $0$, and $+\infty$.

The case where one starts with a Galton--Watson tree $\bT_n$ of depth $n$ which is inhomogeneous with offspring distribution $(\mu_k^{(n)})_{0\leq k \leq n-1}$ should be analogous. One would need to replace the quantity $\nu^n$ with $\prod_{k=0}^{n-1}\nu_k^{(n)}$, where $\nu_k^{(n)}$ is the mean of the offspring distibution~$\mu^{(n)}_k$; again, one should for instance assume that the means $(\nu^{(n)}_k)_{0\leq k \leq n-1, n\geq 1}$, resp.\ the $q$-moments $(m_{q,k}^{(n)})_{0\leq k \leq n-1, n\geq 1}$, remain bounded away from $1$ and $+\infty$.

\subsubsection{Back to the configuration model}
\label{sec:backtoCM}

Coming back to the configuration model, we may try to use our toy model to make some predictions.
Consider the configuration model recalled in Section~\ref{sec:ConfigModel}, with $N$ vertices and degree sequence $\mathbf{d} = (d_i)_{i \in \llbracket N\rrbracket}$. 
Denote $N_k := |\{i\in \llbracket N \rrbracket , d_i=k\}|$ the number of vertices of degree $k$ and let $D_N$ be a random variable whose distribution is given by $\P(D_N=k) = N_k/N$, which corresponds to the degree of a vertex choosen uniformly at random.
A standard and natural assumption on the model, see~\cite[\S~1.3.3]{vdHRandomGraphs2}, is that, as $N\to\infty$, $D_N$ converges in distribution to some random variable $D$ and that we also have convergence of the first two moments of $D_N$ to those of $D$ (assuming that $\E[D^2]<+\infty$).

Then, it is known from \cite[Thm.~4.1]{vdHRandomGraphs2} that, choosing a vertex $o$ uniformly at random and rooting $\CM(\mathbf{d})$ at $o$, the rooted $\CM(\mathbf{d})$ locally converges (\textit{i.e.}\ can be locally coupled with) to a branching process with offspring distribution $\mu$ (except the root which has offspring distribution $D$), where $\mu$ is the distribution of $D^*-1$, with $D^*$ is the size-biased version of $D$; more precisely $\mu(k) := \frac{k+1}{\E[D]} \P(D=k+1)$.

Therefore, Assumption~\ref{hyp:branching} corresponds to having $\P(D=1)=0$, $\P(D=2)<1$ and also $\nu := \sum_{k=1}^{\infty} k\mu(k) = \frac{1}{\E[D]} \E[D(D-1)] >1$ (note that $\nu<+\infty$ if $\E[D^2]<+\infty$).
The assumption that $\mu$ admits a finite second moment translates into the requirement that $\E[D^3]<+\infty$.
Then, we can try to apply our results simply by analogy, \textit{i.e.}\ identifying the configuration model rooted at a random vertex $v$ with a Galton--Watson tree $\bT_n$ with offspring distribution $\mu$ (except at its root) and depth $n = \frac{1}{\log \nu} \log N$; the choice of $n$ is such that the number of vertices in $\bT_n $ is roughly $\nu^n =N$.
Our results then translate into the fact that the vertex $o$ is asymptotically magnetized if and only if $\liminf_{n\to\infty} p_n \nu^n \tanh(\gb)^n >0$, or since $p_n = M_N/N$, if and only if $\liminf_{N\to\infty} M_N N^{\alpha} >0$, where $\alpha :=  \frac{1}{\log \nu} \log \tanh(\gb)$.

However, the approximation of the configuration model rooted at a random vertex~$o$ with a Galton--Watson tree only works up to depth $\tilde n=c \log N$ for some constant $c>0$, provided that the maximal degree $d_{\mathrm{max}} = \max_{v\in \llbracket N \rrbracket} d_v$ verifies $d_{\rm max} = O(N^{a})$ for some $a<1$.
This is a reformulation of~\cite[Lem.~3.3 and Rem.~3.4]{vdH21} in our context: a coupling can be made between the configuration model $\CM(\mathbf{d})$ rooted at a random vertex $o$ and a ($N$-dependent) Branching Process up to $m_N = o(\sqrt{N/d_{\rm max}})$ vertices; if $d_{\rm max} = N^{o(1)}$, this corresponds to a depth $\tilde n = \frac{1}{2\log\nu} \log N$ for the Galton--Watson tree.
The fact that loops start to appear in the graph at some point breaks Lyons iteration's argument and new ideas are needed.
However, one should be able to obtain at least some bounds on the magnetization of~$o$; in particular, the fact that our result is robust to the case of a $n$-dependent Galton--Watson tree could prove useful when working with the coupling mentioned above.

A natural (weak form of the) conjecture is the following.
\begin{conjecture}
For the Ising model on the configuration model $\CM(\mathbf{d})$ with $M_N$ interfering external `$+$' vertices, there exists some $\tilde \alpha>0$, depending only on the inverse temperature $\beta$ and the mean $\nu$, such that a randomly chosen vertex~$o$ in $\CM(\mathbf{d})$ is:
\begin{itemize}
\item asymptotically \textit{magnetized} if $M_N \geq N^{-\tilde\alpha+\gep}$ for some $\gep>0$;
\item asymptotically \textit{not magnetized} if  $M_N \leq N^{-\tilde\alpha-\gep}$ for some $\gep>0$.
\end{itemize}
\end{conjecture}

\noindent
In other words, the threshold for having asymptotic root magnetization should be at a polynomial number of ``interfering'' vertices $M_N = N^{-\tilde\alpha +o(1)}$.
It is natural to guess from our upper bound on the magnetization that  $\tilde \alpha \leq \alpha$ with $\alpha := \frac{1}{\log \nu} \log \tanh(\gb)$. 
On the other hand, if $d_{\rm max} = N^{o(1)}$, applying the exploration process up to depth $\tilde n \simeq \frac{1}{2\log\nu} \log N$ and using our lower bound on the magnetization from Theorem~\ref{thm:reduced} would yield that $\tilde \alpha \geq \frac12 \alpha$.
Overall, it is not clear whether $\tilde \alpha=\alpha$ or not (we personally believe so); this is an interesting question.

\subsubsection{About extremal Gibbs states and free boundary conditions}

A question that has been extensively studied both in the physics and mathematics literature, is that of the extremal Gibbs measures (also called pure states) for the nearest-neighbour ferromagnetic Ising model on an infinite tree (or more general hyperbolic graphs) with zero external field. 
One important question is to know whether the free measure~$\P_{\gb}^f$ (\textit{i.e.}\ with free boundary condition) is extremal. 
If this is not the case, then one can naturally ask whether~$\P_{\gb}^f=\frac12 (\P_{\gb}^++\P_{\gb}^-)$ with $\P_{\gb}^+,\P_{\gb}^-$ the Gibbs measures with `$+$' and `$-$' boundary conditions respectively, or whether~$\P_{\gb}^f$ has a more complex extremal decomposition.
In the present paper, we focus on a non-negative (boundary) external field, which breaks the $+/-$ symmetry of the model.

Remind that on regular trees, in contrast to what happens on $\mathbb{Z}^d$---where the only invariant extremal Gibbs states are $\P_{\gb}^+,\P_{\gb}^-$ (see~\cite{Bodineau2006})---the Ising model exhibits the following key features:
(i) The free measure~$\P_{\gb}^f$ is extremal in a non-trivial subcritical temperature regime, see~\cite{BRZ95,Ioffe96};
(ii) At sufficiently low temperature, there are uncountably many extremal Gibbs states, see~\cite{Higuchi1977} and~\cite{BlekGani1991,GandRuizShlos2012} for rigorous results, see also~\cite{CoqKulLeNy2023} for other finite-spin models.
Regarding the extremal decomposition of the free measure~$\P_{\gb}^f$, recent results~\cite{CKLN23,GMR20} show that, at sufficiently low temperature, the extremal decomposition is supported on uncountably many extremal states.
In contrast, \cite{Wu2000} provides examples of hyperbolic graphs where, at sufficiently low temperatures, $\P_{\gb}^f$ is given by $\frac12 (\P_{\gb}^+ + \P_{\gb}^-)$.

It is then reasonable to ask whether the results on regular trees remain valid on random trees, or on tree-like graphs.
An example of such study is done in~\cite{MontanariMosselSly}, where the authors consider the free Ising measure on a growing sequence of graphs $(G_n)_{n\geq 1}$ that locally converge to a regular tree: their main result is that the Ising measure locally weakly converges to the mixture $\frac12 (\P_{\gb}^+ + \P_{\gb}^-)$, which contrasts with the behaviour that has been described in the previous paragraph.
Our present work raises the natural question to determine whether, and to what extent, this result continues to hold when a \emph{signed} sparse (boundary) random external field is considered, \textit{i.e.}\ if $(h_v)_{v\in \bT}$ (or $(h_v)_{v\in \partial\bT}$) are i.i.d.\ with law $\PP(h_v = \pm 1) =\frac12 p_n$ and $\PP(h_v = 0) =1-p_n$, whit $\lim_{n\to\infty} p_n= 0$.
This is related to the question of the effect of random boundary conditions on the Ising model, which is relevant to the theory of spin glasses. 
In fact, on trees, the Edwards Anderson model is related to an Ising model with random (signed) boundary condition (see~\cite{CCST86}).
In particular, an interesting question is the effect of a random boundary conditions on the coexistence of pure states, in the spirit of \cite{vanEnterNetocnySchaap2006}; see also~\cite{EndoVanEnterLeNy2021} for a recent overview.
Notably, the results of~\cite{CKLN23} on the extremal decomposition of $\P_{\gb}^f$ over uncountably many pure states are stable under the addition of a small external field, whether random or not.
We believe that such questions are natural and promising directions of research, in continuity of the present paper.

\section{A few preliminaries}
\label{sec:prelim}

\subsection{Some notation for trees}

Let $\mathbf{t} = (V(\mathbf{t}), E(\mathbf{t}), \rho)$ be a tree\footnote{Recall that a tree is a connected, acyclic and undirected graph, or equivalently, a graph where every two vertices are connected by exactly one path.} with root $\rho$;
with an abuse of notation, we also write $v\in \bt$ as a shorthand for $v\in V(\bt)$.
The distance between two vertices in the tree is the number of edges of the unique path connecting them.
Given two vertices $v,w \in \mathbf{t}$, we say that $w$ is a \emph{descendant} of $v$, and we denote $v \leqslant w$, if the vertex $v$ is on the shortest path from the root $\rho$ to the vertex $w$.
For a vertex $v \in \mathbf{t}$, we let $\vert v \vert$ denote the distance from $v$ to the root $\rho$; note that if we have $v\leq w$, we have $\vert v \vert \leqslant \vert  w\vert$.

We consider a tree of depth $n$, that is such that $\max_{v\in \bt}|v| = n$.

\begin{itemize} 
\item We denote $t_k = \{ v \in \mathbf{t} \;:\;\vert v \vert = k\}$ the $k$-th generation of the tree $\mathbf{t}$, for $k \leqslant n$; in general, we use bold font to denote whole trees and standard font to denote generations.
\item Given two vertices $v, w \in \mathbf{t}$ we say $w$ is a (direct) successor of $v$, and we note $v \rightarrow w$, if $v < w$ and $\vert  v \vert = \vert w \vert - 1$.  
\item For a vertex $v \in \mathbf{t}$, we denote $S(v) = \{ w \in \mathbf{t} \;:\; v \rightarrow w \}$ the set of (direct) successors of $v$, and $d(v) = \vert S(v)\vert$ its cardinal, \textit{i.e.}\ the number of descendants of $v$. 

\item We say that a vertex $v \in \mathbf{t}$ is a leaf if it has no successor, namely if $d(v)=0$, and we denote $\partial \bt = \{ v\in \bt, d(v)=0\}$ the set of leaves of the tree. 
\item For a vertex $v \in \bt$, we let $\bt(v)$ be the subtree of $\bt$ consisting of $v$ (as a root) and all vertices $w$ such that $v \leqslant w$; if $v\in t_k$ is in generation $k$, we may also use the notation $\bt_k(v)$ to make the dependence on $k$ explicit.
\end{itemize}

\noindent
In the following, thanks to Assumption~\ref{hyp:branching}, we only consider trees that have no leaves except at generation $n$ ($\partial \bt =t_n$), \textit{i.e.} such that all vertices $|v|<n$ have at least one successor, $d(v) \geqslant 1$.

\subsection{Log-likelihood ratio and Lyons iteration}
\label{sec:Lyons}

One fundamental element of the proof of Theorem~\ref{thm:Lyons} is the fact that the log-likelihood ratio of the root magnetization can be expressed recursively on the tree.
% that is, it can be expressed in terms of the log-likelihood ratio of its descendants. 
In this section, we introduce a few notation and state this recursive formula, that we call Lyons iteration in reference to~\cite[Eq.~(2.1)]{Lyons89}; for the sake of completeness and because we adapt the iteration in the following, we recall how to obtain it.

\paragraph*{Lyons iteration with plus boundary condition.}
Let us start by giving the definition of the log-likelihood ratio $r_{\bt, \gb}^{+}(\rho)$  of the root on a tree $\bt$ of depth $n$, with (classical) plus boundary conditions at temperature $\beta$:
\begin{equation*}
r_{n, \gb}^{+}(\rho) := \log \bigg( \dfrac{\P_{n,\gb}^{+}(\sigma_{\rho}=+1)}{\P_{n,\gb}^{+}(\sigma_{\rho}=-1)} \bigg)
\end{equation*}
%With plus boundary condition, we have $r_{n, \gb}^{+}(\rho) > 0$ for all $n$, meaning that the root is positively magnetized, but the question is whether $r_{n, \gb}^{+}(\rho)$ remains bounded from below in the limit $n\to\infty$; this characterizes the phase transition of the model. 

We also introduce the log-likelihood ratio $r_{n, \gb}^{+}(u)$ of a vertex $u \in \bt$.
We define the partition function of the Ising model on the sub-tree $\bt(u)$ with plus boundary conditions at temperature $\beta$, conditioned on the vertex $u$ (the root of $\bt(u)$) having spin $a \in \{-1,1\}$: 
\begin{equation}
\label{eq:PartitionFctVertex}
Z_{\beta}^{+,a}(u) =  \sumtwo{\sigma \in \{-1,1 \}^{\vert \bt(u) \vert} }{\sigma_{u}=a} \exp\bigg( \gb \Big(\sumtwo{v,w \in \bt(u)}{v\sim w} \sigma_v \sigma_w  + a \sum_{v \in \partial\bt(u) } \sigma_v \Big) \bigg) \,.
\end{equation} 
Then, we let 
\begin{equation}
\label{def:loglike}
r_{n,\gb}^{+}(u) = \log \bigg(\frac{Z_{\beta}^{+,+1}(u)}{Z_{\beta}^{+,-1}(u)} \bigg) \,,
\end{equation}
which corresponds to the log-likelihood ratio of the root of the Ising model in the sub-tree $\bt(u)$, of depth $n-|u|$,  with plus boundary condition (we keep the subscript $n$ to remember that $u\in \bt$ with $\bt$ of depth $n$).
Let us notice that, for any vertex $u \in \bt$ different from the root, $r_{n, \gb}^{+}(u)$ does not correspond to $\log \big( \P_{n,\gb}^{+}(\sigma_{u}=+1) / \P_{n,\gb}^{+}(\sigma_{u}=-1) \big)$.

After a straightforward computation, for $u\in \bt \setminus \partial \bt$, we can write $(Z_{n, \beta}^{+,a}(u))_{a \in \{-1,+1\}}$ in terms of the partition functions $(Z_{n, \beta}^{+,a}(v))$ with $v$ successors of $u$: we have
\begin{equation}
\label{eq:PartitionFctComputation}
Z_{\beta}^{+,a}(u)
 = \prod_{v \colon u \rightarrow v} \Big( e^{+\beta a} Z_{\beta}^{+, +1}(v) + e^{-\beta a} Z_{\beta}^{+, -1}(v) \Big) \,.
\end{equation}
We can therefore express the log-likelihood ratio $r_{n,\gb}^{+}(u)$ in terms of partition functions:
\begin{equation}
\label{eq:loglikelihood}
r_{n, \gb}^{+}(u)  = 
\log \bigg(\frac{Z_{\beta}^{+,+1}(u)}{Z_{\beta}^{+,-1}(u)} \bigg) = \sum_{v \colon u \to v } \log \bigg(  
\dfrac{e^{+\beta} Z_{n,\beta}^{+, +1}(v) + e^{-\beta} Z_{\beta}^{+, -1}(v)}{e^{-\beta} Z_{\beta}^{+, +1}(v) + e^{+\beta} Z_{\beta}^{+, -1}(v)}
\bigg) \,.
\end{equation}

\noindent
Defining $g_{\gb} (x) := \log \Big(\frac{e^{2\gb}e^{x} +1}{e^{2\gb} + e^x} \Big)$, we therefore end up with the following crucial recursion:
\begin{equation}
\label{eq:Lyonsiteration}
r_{n, \gb}^{+}(u) =
\begin{cases}
\displaystyle  \sum_{v \colon u\rightarrow v} g_{\gb}(r_{n, \gb}^{+}(v)) & \qquad  \text{if } u \notin \partial\bt \,, \\
 +\infty  & \qquad \text{if } u \in \partial\bt \,.
\end{cases}
\end{equation}

\paragraph*{Lyons iteration with an external field.}
We can also define, analogously to~\eqref{def:loglike}, the log-likelihood ratio of a vertex $u\in \bt$ for the Ising model with external field $h=(h_v)_{v\in \bt}$. 

For $u\in \bt$, let $Z_{\beta}^{h,a} (u)$ denote the partition function of the Ising model on $\bt(u)$ at temperature $\beta$ with external field $(h_v)_{v\in \bt(u)}$, conditioned on the root having spin $a \in \{-1,+1\}$. In the same way as in \eqref{eq:PartitionFctComputation}, we have for $u\in \bt \setminus \partial \bt$
\begin{equation*}
Z_{\beta}^{h,a}(u) 
= e^{\gb h_u a}\prod_{v \colon u\rightarrow v} \Big( e^{+\beta a} Z_{\beta}^{h, +1}(v) + e^{-\beta a} Z_{\beta}^{h, -1}(v) \Big) \,.
\end{equation*}
Then, as in~\eqref{eq:loglikelihood}, the log-likelihood ratio  $r_{n, \beta}^{h}(u)$ can be expressed in terms of the conditional partition functions, and we end up with the following recursion, analogous to~\eqref{eq:Lyonsiteration}:
\begin{equation}
\label{eq:recLyons}
r_{n, \beta}^{h}(u) = 
\begin{cases}
\displaystyle
2 \beta h_{u} + \sum_{v \colon u \to v} g_{\beta} (r_{n, \beta}^{h} (v)) \quad & \text{if } u \notin \partial\bt \,, \\
\displaystyle
2 \beta h_{u} & \text{if } u \in \partial\bt \,.
\end{cases}
\end{equation}
Here, we have also used that if $u\in \partial \bt$ is a leaf, the log-likelihood ratio $r_{n,\gb}^h(u)$ is easily seen to be~$2\beta h_u$.

\begin{remark}
Note that the Ising model on $\bt$ with a plus boundary external field $h_+=(h_v)_{v\in \partial \bt}$, \textit{i.e.}\ $h_v=+1$ for all $v\in \partial \bt$, corresponds to the Ising model with plus boundary condition on some extended tree, obtained by adding an extra generation to $\bt$ with exactly one descendant to all $v\in \partial \bt$, recall Figure~\ref{fig:graphboundary}.
In this setting, the log-likelihood ratios $r_{n,\gb}^{h_+} (u)$ verify the same recursion as in~\eqref{eq:Lyonsiteration} but with a different initial condition on the leaves:
\begin{equation}
\label{eq:initialcond}
r_{n, \gb}^{h_+}(v) = 
2\gb \quad  \text{if } v \in \partial\bt \,.
\end{equation}
%Indeed, each vertex $u\in \partial \bt$ has exactly one descendant (a leaf) in the extended tree, so that $r_{n,\gb}^{h_+}(u) = g_{\gb}(+\infty) =2\gb$.
\end{remark}

\subsection{The case of a non-vanishing sequence $(p_n)_{n\geq 0}$}

Before we move to the proofs, let us evacuate the case of a non vanishing sequence $(p_n)_{n\geq 0}$ in the different models of Section~\ref{sec:results}.

\begin{lemma}
\label{lem:pnconstant}
Let $\bT$ be a branching process satisfying Assumption~\ref{hyp:branching} and consider the Ising model with sparse Bernoulli external field inside the tree, \textit{i.e.}\ model \textbf{(b)} of Section~\ref{sec:results}, with parameters $(p_n)_{n\geq 0}$ in $[0,1]$.
If $\liminf_{n\to\infty} p_n >0$, then the root is asymptotically magnetized.
\end{lemma}

\begin{proof}
To prove this, it is enough to consider the event where $h_\rho=+1$.
On the event that $h_{\rho}=1$ and since we have $r_n^{(p_n)}(v)\geq 0$ for all $v\leftarrow \rho$, we get that $r_{n}(\rho) \geq 2 \gb $.
Hence, for $\gep< 2\gb$ we have $\bP\otimes \PP(r_{n}(\rho) >\gep) \geq \PP(h_{\rho}=1) =p_n$.
For such $\gep$, we get that
\[
\liminf_{n\to\infty} \bP\otimes \PP(r_{n}(\rho) >\gep) >0\,,
\]
which concludes the proof.
\end{proof}

In the case of the Ising model with non-vanishing Bernoulli boundary external field, \textit{i.e.}\ model \textbf{(c)} of Section~\ref{sec:results} with $\liminf_{n\to\infty} p_n >0$, one obtains a similar result as with plus boundary conditions, that is Theorem~\ref{thm:Lyons}:
the root is asymptotically magnetized if and only if $\nu\tanh(\gb)>1$.
This is actually a corollary of Theorem~\ref{thm:PemPer} (from~\cite{PemPer10}) and our bounds on the $3$-capacity of the pruned tree, see Proposition~\ref{prop:capacitypruned} which is also valid for a non-vanishing sequence $(p_n)_{n\geq 0}$.

\section{Upper bound: Lyons argument}
\label{sec:upper}

In the section, we consider the Ising model with sparse external field inside the tree, \textit{i.e.}\ model \textbf{(b)} of Section~\ref{sec:results}.
We let $(p_n)_{n\geq 0}$ be a sequence in $[0,1]$ and we give an upper bound on the log-likelihood ratio.

We show here the following proposition, using Lyons' iteration~\eqref{eq:recLyons}, which gives a sufficient condition for having no root magnetization.

\begin{proposition}
\label{prop:upperLyons}
Suppose that Assumption~\ref{hyp:branching} holds.
Consider the Ising model  on $\bT$ with sparse Bernoulli external field, with $\lim_{n\to\infty} p_n=0$.

Then if $\lim_{n\to\infty} (\nu \tanh(\gb))^n p_n =0$, there is no asymptotic root magnetization.
More precisely, there is a constant $C=C_{\gb,\nu}$ such that
\[
\bE\otimes \EE \big[ r_{n, \beta}^{p_n}(\rho) \big]  
\leq C
\begin{cases}
p_n & \quad \text{ if } \tanh(\gb)\nu <1 \,,\\
\sqrt{p_n} & \quad \text{ if } \tanh(\gb)\nu = 1\,,\\
p_n (\nu \tanh(\gb))^n & \quad \text{ if } \tanh(\gb)\nu > 1 \,.
\end{cases}
\]
\end{proposition}

\begin{proof}
This is a simple application of the recursion formula~\eqref{eq:recLyons}.
Let us first give a simple bound, valid for a generic tree $\bt$ and external field $(h_v)_{v\in \bt}$.

Using the easy bound $g_{\beta}(x) \leqslant \tanh(\beta) x$ for all $x \geqslant 1$ (note that $\tanh(\gb)$ is the derivative of $g_{\gb}$ at $0$), we obtain from~\eqref{eq:recLyons} that for any $u\in \bt \setminus \partial \bt$,
\[
r_{n, \beta}^{h}(u) \leqslant  2\gb h_u +\tanh(\beta) \sum_{v \colon u \to v} r_{n,\gb}^h(v) \,.
\]
Applying this inequality recursively we obtain the following upper bound for the log-likelihood ratio of the root~$\rho$:
\begin{equation*}
r_{n, \beta}^{h}(\rho) \leqslant \sum_{k=0}^{n} \tanh(\beta)^k \sum_{u \in t_k} 2 \beta h_u \,.
\end{equation*}
%Using in particular that $\tilde{r}_{n, \beta}^{p_n}(v)=0$ on the leaves $v \in \partial\bT$, we obtain that
%\begin{equation}
%r_{n, \beta}^{p_n}(\rho) = \tilde{r}_{n, \beta}^{p_n}(\rho) - 2\gb h_{\rho} d_{\rho} \leqslant \sum_{k=1}^{n-1} \tanh(\beta)^k \sum_{u \in \bT_k} 2 \beta h_u d_u \,.
%\end{equation}

In our specific setting where $(h_v)_{v\in \bT}$ is an i.i.d.\ Bernoulli field with parameter $p_n$, we need to show that the upper bound goes to $0$ in $\PP\otimes\bP$ probability.
We simply take its expectation with respect to $\bP\otimes \mathbb{P}$: we obtain
\begin{equation}
\label{eq:uppermean}
\bE\otimes \EE \big[ r_{n, \beta}^{(p_n)}(\rho) \big] \leq  
2 \beta p_n \sum_{k=0}^{n} (\tanh(\beta) \nu )^k  \,.
\end{equation}

\smallskip
\textbullet\ In the case $\tanh(\beta) \nu <1$, the  last sum is bounded by a constant and we have that $\bE\otimes \EE \big[ r_{n, \beta}^{p_n}(\rho) \big] \leq c p_n$, which goes to $0$.
 
 \smallskip
\textbullet\ In the case $\tanh(\beta) \nu >1$, we have 
\begin{equation*}
\bE\otimes \EE \big[ r_{n, \beta}^{(p_n)}(\rho) \big] \leqslant \frac{2 \beta}{\tanh(\gb)-1} p_n   ((\tanh(\gb)\nu)^n-1) \leq c_{\gb,\nu} p_n   (\tanh(\gb)\nu)^n \,,
\end{equation*}
which goes to $0$ under the assumptions of the proposition.

\smallskip
\textbullet\ The case $\tanh(\gb)\nu=1$ is a bit more subtle: taking the expectation in~\eqref{eq:uppermean} gives the upper bound $\bE\otimes \EE [ r_{n, \beta}^{(p_n)}(\rho) ] \leq 2\gb n p_n$, which goes to $0$ only if $p_n=o(\frac1n)$.
In the general case, we need some extra work.
Going back to Lyons' iteration~\eqref{eq:recLyons} and using that $x\mapsto g_{\gb}(x)$ is concave, we get by Jensen's inequality (recalling that $\nu$ is the mean offspring number) that
\[
\bE\otimes \EE \big[ r_{n, \beta}^{(p_n)}(u) \big] \leq 2 \gb p_n + \nu g_{\gb} \Big(\bE\otimes \EE \big[ r_{n, \beta}^{(p_n)}(v) \big]  \Big)  \qquad \text{ for } u\notin \partial \bT \,,
\]
where $u,v$ are generic vertices at successive generations; note that the expectation $\bE\otimes \EE \big[ r_{n, \beta}^{(p_n)}(u) \big]$ depends only on (the distribution of) the subtree $\bt(u)$ and $(h_v)_{v\in \bt(u)}$, so in fact only on $|u|$.
Hence, setting $y_k := \bE\otimes \EE [ r_{n, \beta}^{(p_n)}(u) ]$ with $|u|=k$,  we end up with the following iteration
\begin{equation}
\label{eq:recLyons2}
y_k \leq  
\begin{cases}
f_{n,\gb}(y_{k+1})  \quad & \text{if }  0\leq k<n\,, \\
2 \beta p_n & \text{if } k=n\,.
\end{cases}
\end{equation}
with $f_{n,\gb}(x) = 2\gb p_n + \nu g_{\gb}(x)$, which is a concave increasing function.
Let us denote $x_n$ the unique solution of $f_{n,\gb}(x_n)=x_n$. 
Then, it is clear that either the initial condition verifies $y_n \leq x_n$
and then $y_k\leq x_n$ for all $k\leq n$, or it verifies $y_n>x_n$ and then $y_k\leq y_n =2\gb p_n$ for all $k\leq n$. 
We have therefore proven that 
\[
y_0 = \bE\otimes \EE [ r_{n, \beta}^{p_n}(\rho) ] \leq \max \{2\gb p_n,x_n \} \,,
\]
and it remains to show that the fixed point $x_n$ goes to $0$ as $n\to\infty$.
But this should be clear, since $x\mapsto \nu g_{\gb}(x)$ is a strictly concave function with slope $\nu\tanh(\gb)=1$ at the origin.

In fact, let us prove that $x_n = O(\sqrt{p_n})$.
First, note that  $\nu g_{\gb}(x) \leq x - c$ for any $x\geq 1$ with the constant $c =1-\nu g_{\gb}(1)>0$ (by concavity), so that $f_{n,\gb}(x) \leq 2\gb p_n + x -c$ for all $x\geq 1$.
If $n$ is large enough so that $2\gb p_n <c$, this inequality cannot be verified at $x=x_n$, so we must have $x_n<1$.
Then, writing that there is some constant $c_{\gb}>0$ such that $\nu g_{\gb}(x) \leq x- c_{\gb} x^2$ for all $x\in (0,1)$, we get that $x_n = f_{n,\gb}(x_n)\leq 2\gb p_n +x_n-c_{\gb} x_n^2 $ for all $n$ large enough.
We conclude easily that $x_n \leq c'_{\gb} \sqrt{p_n}$ for $n$ large, which is what was claimed.
\end{proof}

\section{Lower bound: comparison with an Ising model on a pruned tree}
\label{sec:lower}

First of all, we use the Ising model with Bernoulli boundary external field $(h_v)_{v\in \partial \bT}$, \textit{i.e.}\ model \textbf{(c)} in Section~\ref{sec:results}, to get a lower bound for the magnetization of the Ising model with Bernoulli external field in the whole tree, since the magnetization is clearly higher in the second case.

\subsection{From a Bernoulli external field to a pruned tree}
\label{sec:comparePruned}

Let us show in this section that, conditionally on the realization $\bt$ of the Galton--Watson tree $\bT$ and on the realization $h$ of a Bernoulli external field $(h_v)_{v\in \partial\bT}$, the magnetization is equal to the magnetization on a \textit{pruned} (sub)-tree $\bt^*$ with plus boundary external field $h_+ = (h_v)_{v\in \partial \bt^*}$, \textit{i.e.}\ with $h_v=+1$ for all $v\in \partial \bt^*$. 
Informally, the pruned tree $\bt^*$ is obtained by removing all branches in $\bt$ that do not lead to a leaf with $h_v=+1$; we call this procedure the \emph{pruning of dead branches}, see~Figure~\ref{fig:pruned} for an illustration.

To be more formal, given a tree $\bt$ of depth $n$ and $(h_v)_{v\in \partial \bt}$ an external field with value in $\{0,1\}$, let us define indicator variables $(Y_u)_{u\in \bt}$ as follows:
\begin{equation}
\label{def:Xv}
Y_u =
\begin{cases}
+1 & \quad \text{if there exists $v\geq u$, $v\in \partial \bt$ such that $h_v =+1$} \\
0 & \quad \text{otherwise} \,.
\end{cases}
\end{equation}
We interpret having $Y_u=+1$ as the fact that the vertex $u$ belongs to a living branch, and so has ``survived to the pruning''.
Alternatively, we can construct $(Y_u)_{u\in \bt}$ iteratively starting from the leaves:
\begin{itemize}
\item for $v\in \partial \bt =t_n$, set $X_v :=h_v$;
\item iteratively, for $u\in \bt\setminus \partial \bt$, we set $Y_u = +1$ if and only if there is some $v \leftarrow u$ with  $Y_v=+1$ and $Y_u=0$ otherwise.
\end{itemize} 
We then define the \emph{pruned tree} as
\begin{equation}
\label{defpruned}
\bt^* = \mathrm{Pruned}_{h}(\bt) := \{v \in \bt, Y_v =+1\} \,.
\end{equation}

\begin{figure}[tb]
\centering
\includegraphics[scale=0.35]{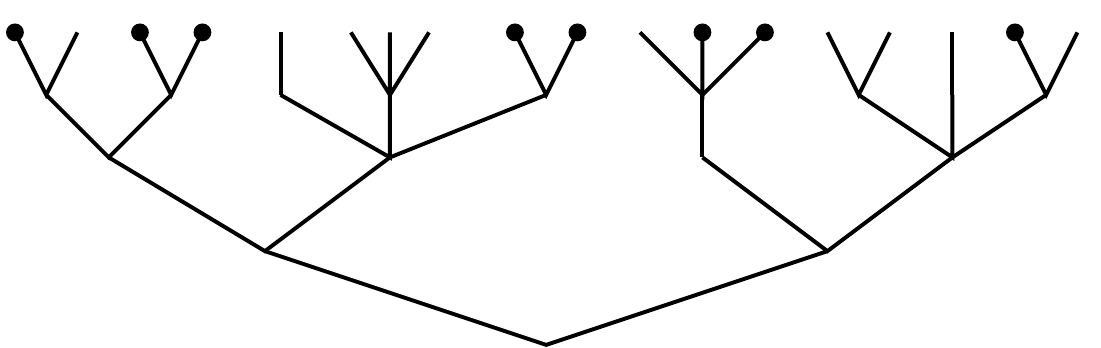}
\qquad
\includegraphics[scale=0.35]{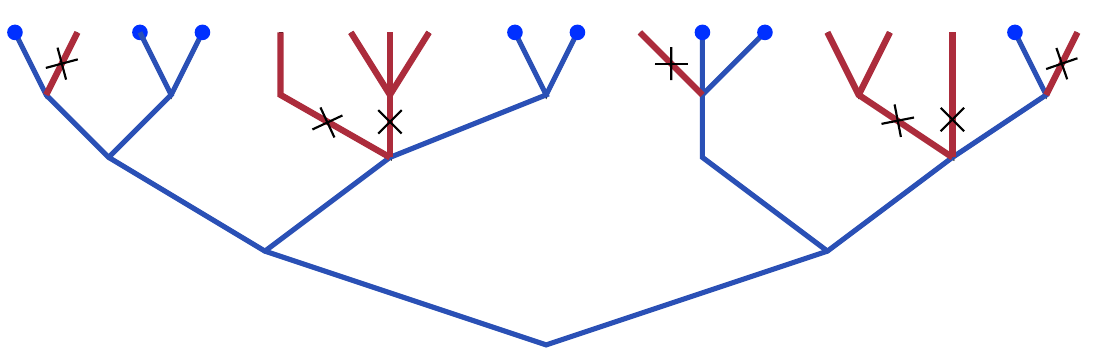}
\caption{\small Illustration of the pruning procedure. On the left, we have represented the tree $\bt$ and the leaves with $h_v=+1$ with dots.
The picture on the right represents how the pruning procedure proceeds: in blue are all the ``living branches'' leading to a leaf with $h_v=+1$; in red are all the  ``dead branches'' leading to $h_v=0$, that have to be pruned.
The pruned tree $\bt^*$ consists in keeping only the living (blue) branches.
}
\label{fig:pruned}
\end{figure}

\noindent
Let us consider two following log-likelihood ratios, at inverse temperature~$\gb$:
\begin{itemize}
\item For $u\in \bt$, $r_{\bt,\gb}^{h} (u)$ is the log-likelihood ratio for the Ising model on $\bt$ with external field $h \in \{0,1\}^{\partial \bt}$ on the boundary;

\item For $u\in \bt^*$, $r_{\bt^*,\gb}^{h_{+}} (u)$ is the log-likelihood ratio for the Ising model on $\bt^*$  with external field $h_{+} \equiv 1$ on the boundary $\partial \bt$ (note that we have $h_+ =h |_{\partial \bt^*}$).
\end{itemize}
We now show the following lemma.

\begin{lemma}
\label{lem:tt*}
We have $r_{\bt,\gb}^h(u) =0$ for any $u\in \bt\setminus \bt^*$ and $r_{\bt,\gb}^h(u) =r_{\bt^*,\gb}^{h_+}(u)$ for any $u\in\bt^*$.
In particular, the log-likelihood ratio of the root verifies $r_{\bt,\gb}^{h}(\rho) =r_{\bt^*}^{h_+}(\rho)$.
\end{lemma}

\begin{proof}
We simply need to write the following Lyons recursions (see also Remark~\ref{rem:boundary}).
On the tree $\bt$, we have
\[
r_{\bt,\gb}^h(u)= 
\begin{cases}
\sum \limits_{v \colon u\rightarrow v} g_{\beta}\big(r_{\bt,\gb}^h(v)\big), &\qquad \text{if } u \notin \partial\textbf{t} \,, \\
2 \beta h_v, &\qquad \text{if } v \in \partial\bt \,.
\end{cases} 
\]
Since $g_{\gb}(0)=0$, this readily proves that $r_{\bt,\gb}^h(u)=0$ if all $v\in \partial \bt$ descendants of $u$ verify $h_v=0$, \textit{i.e.}\ if $Y_u=0$ and $u$ has been pruned ($u\in \bt\setminus \bt^*$).
On the tree $\bt^*$, we have
\[
r_{\bt^*,\gb}^{h_+}(u)=
\begin{cases}
\sum\limits_{v \colon u\rightarrow v} g_{\beta}\big( r_{\bt^*,\gb}^{h_+}(v)\big),& \qquad \text{if } v\notin \partial\textbf{t}^* \,, \\
2 \beta, & \qquad \text{if } v \in \partial\textbf{t}^* \,.
\end{cases}
\]
With a slight abuse of notation, we can extend the definition to the tree $\bt$ by setting $r_{\bt^*,\gb}^{h_+}(u) =0$ if $u \in \bt\setminus \bt^*$.
Now, this extended definition of $r_{\bt^*,\gb}^{h_+}$ yields exactly the same recursion as $r_{\bt,\gb}^{h}(u)$ so $r_{\bt^*,\gb}^{h_+}(u)=r_{\bt,\gb}^{h}(u)$ for all $u\in \bt$.
\end{proof}

With Lemma~\ref{lem:tt*}, the rest of the paper then consist in studying the Ising model on the pruned version $\bT^*_n$ of a Galton--Watson tree $\bT_n$, with plus boundary external field.
More precisely, we show the following.

\begin{theorem}
\label{thm:reduced}
Let $\bT_n$ be a Galton--Watson tree of depth $n$
whose offspring distribution~$\mu$ satisfies Assumption~\ref{hyp:branching} and has a finite moment of order $q$ for some $q>1$, and let $(p_n)_{n\geq 0}$ be a vanishing sequence.
Let~$\bT^*_n$ be the pruned version of $\bT_n$ with $(h_v)_{v\in \partial \bT}$ given by i.i.d.\ Bernoulli random variables with parameter $p_n$, \textit{i.e.}\ $\bT_n^* = \mathrm{Pruned}_h(T_n^*)$.
Then, for the Ising model on $\bT_n^*$ with plus boundary external field on $\partial \bT_n^*$, the root is asymptotically magnetized if and only if
\[
\nu \tanh(\gb) >1 \qquad \text{ and } \qquad
\liminf_{n\to\infty} (\nu \tanh(\gb))^n p_n >0 \,. 
\] 
\end{theorem}

\subsection{Relation between the magnetization and $p$-capacity}
\label{sec:PemPer}

A first crucial tool for proving Theorem~\ref{thm:reduced} is the following observation, made by Pemantle and Peres~\cite{PemPer10}: the magnetization of the root (more precisely the log-likelihood ratio) for the Ising model on a tree with plus boundary condition is comparable to some $L^p$-capacity, or $p$-capacity, of the tree, equipped with a specific set of resistances.
%, Pemantle and Peres~\cite{PemPer10} observed that 
We state this as Theorem~\ref{thm:PemPer} below, but let us introduce the necessary notation first.

\subsubsection{Non-linear $L^p$-capacity of trees}

Let us start by defining the $L^p$- or $p$-capacity for a given (finite or infinite) tree $\bt=(V,E)$ rooted at a vertex $\rho$; we use analogous notation as in~\cite{AVB24}, see also \cite[\S2]{CNS21} for a brief introduction (the notation in~\cite{PemPer10} actually corresponds to the standard definitions of the $\frac{p}{p-1}$-capacity).
We let $\partial \bt$ be the set of maximal paths oriented away from the root. If the tree is finite, we can identify $\partial \bt$ with the set of leaves of $\bt$. For infinite trees, we assume that there is no leaf and that all paths in $\partial \bt$ are infinite. 

The tree $\bt$ is equipped with a set of resistances on its edges: to each edge $e  =uv$ with $u\rightarrow v$, assign a resistance $R_u:=R(e) \in \RR_{+}$ and a conductance $C_u:= C(e) = R(e)^{-1}$.
We say that a function $\theta: V \mapsto \RR_{+}$ is a flow on the tree $\bt$ if for every $u \in V$ it verifies 
$\theta(u) = \sum_{v \colon u \rightarrow v} \theta(v)$, \textit{i.e.}\ if the inflow is equal to the outflow at every vertex of the tree.
Additionally, we define the strength $\vert \theta \vert$ of a flow $\theta$ as $\vert \theta \vert  := \sum_{v \colon \rho \rightarrow v} \theta(v)$, that is, the total outflow from the root $\rho$.
A flow $\theta$ with $\vert \theta \vert = 1$ is called a \emph{unit} flow.

\begin{definition}[$L^p$-resistance and capacity by the Thomson's principle]
\label{def:pcapacity} 
Let $p > 1$. Then we define the effective $p$-resistance of the tree $\bt$ equipped with resistances $(R_e)_{e\in V}$ between the root $\rho$ and the leaves $\partial \bt$ as
\begin{align*}
\mathcal{R}_p(\rho\leftrightarrow \partial \bt) := \inf_{\theta : |\theta| =1} \bigg\{ \sum_{u\in \bt} R_u^{\frac{1}{p-1}} \theta(u)^{\frac{p}{p-1}}  \bigg\}^{p-1}\,.
\end{align*}
The $p$-capacity is then defined as
\[
\capa_p(\rho \leftrightarrow \partial \bt) = \mathcal{R}_p(\rho\leftrightarrow \partial \bt)^{-1} \,.
\]
By an abuse of terminology, we call $\capa_p(\rho \leftrightarrow \partial \bt)$ the $p$-capacity of the tree $\bt$ and we write $\capa_p(\partial \bt)$ to simplify notation.
\end{definition}

\begin{remark}
We stress that when $p=2$, the $2$-capacity $\capa_2(T)$ reduces to the usual electrical effective conductance between the root $\rho$ and the leaves of the tree; we refer to \cite[Ch.~2, 3 \& 9]{LyonsPeres} for an extensive introduction on electrical networks on graphs.
\end{remark}

\begin{example}[Spherically symmetric trees]
\label{ex:spherical}
Let $\bt$ be a spherically symmetric tree of depth $n \in \NN$, that is, such that all vertices of the same generation have the same number of direct descendants (for instance, a regular tree).
If we equip the tree $\bt$ with resistances $R_u=R_{|u|}$ that depend solely on the vertex generation, then $\capa_{p}(\bt)$ can be computed explicitly.
Indeed, one can check that the \textit{unit} flow (\textit{i.e.}\ $\theta(v) = |t_k|^{-1}$ for all $v\in t_k$) realizes the infimum in Definition~\ref{def:pcapacity}.
We thus obtain that
$\capa_{p}(\bt) = \big( \sum_{k=1}^n ( R_k/|t_k|)^{\frac{1}{p-1}} \big)^{-(p-1)}$.
\end{example}

Then, from simple series and parallel laws for the $p$-capacity, one can  establish a recursive expression for the $p$-capacity of a tree, see~\cite[Eq.~(1.4)]{AVB24} (see also \cite[Lem.~3.1]{PemPer10}, replacing $p$ by $\frac{p}{p-1}$).
Recall that for a vertex $u \in V$, $\bt(u)$ denotes the sub-tree of $\bt$ consisting of $u$ (as a root) and all descendants of $u$.

\begin{lemma}
\label{LemPemPer}
Let $\bt$ be a locally finite tree with root $\rho$. Fix $p > 1$ and for any vertex $u \in V$, define
\[
\phi_p (u) := R_u \capa_p(\bt(u)),
\]
where $R_{\rho} = 1$ by convention; in particular, $\phi(u)=R_u$ if $u$ is a leaf. Then, for any vertex $u\in V$, we have 
\begin{equation}
\label{defRecCapp}
\phi_p(u) = \sum_{v \colon u \rightarrow v} \dfrac{R_u}{R_v} \dfrac{\phi_p(v)}{(1 + \phi_p(v)^s)^{1/s}} \,, \qquad \text{ with } s=\frac{1}{p-1} \,.
\end{equation}
\end{lemma}

The iteration~\eqref{defRecCapp} is extremely useful and is the starting point of the estimates of the $p$-capacity of (homogeneous) Galton--Watson trees performed in~\cite{AVB24}.
In particular, \cite[\S1.3.1]{AVB24} shows that if $\bT_n$ is a Galton--Watson tree equipped with resistances $R_u =R^{-|u|}$ for some $R>0$, then under a $\frac{p}{p-1}$-moment condition, we have that 
\begin{equation}
\label{eq:capaGW}
\capa_p(\bT_n) \asymp_{\PP} \bigg( \sum_{k=1}^n  (R \nu )^{- k s} \bigg)^{-1/s} \,\qquad \text{ with } s := \frac{1}{p-1} \,,
\end{equation}
where $a_n\asymp_{\PP} b_n$ means that the ratio $a_n/b_n$ is tight in $(0,+\infty)$; note the analogy with the formula for spherically symmetric trees of Example~\ref{ex:spherical}.

\begin{remark}
\label{rem:critical}
We stress that the moment assumption is most crucial in the critical case $R=\nu^{-1}$, which gives $\capa_p(\bT_n) \asymp_{\PP} n^{-1/s}$; without the moment assumption, one gets a different behavior.
Let us quickly summarize our main findings, in informal terms: we refer to~\cite[Thms.~1.8 \& 1.10]{AVB24} for precise statements.
\begin{itemize}
\item[(i)] If the offspring distribution admits a finite $q=\frac{p}{p-1}$-moment, then we have that  $n^{-\frac{1}{q-1}} \capa_p(\bT_n)$ converges almost surely and in $L^q$ to a tree-dependent constant;
\item[(ii)] If the offspring distribution has a tail $\sum_{k\geq n} \mu(k) = n^{-\alpha +o(1)}$ for some $\alpha \in (1,3)$ as $n\to\infty$, then we have that $\capa_p(\bT_n) = n^{-\frac{1}{\alpha-1} +o(1)}$ as $n\to\infty$.
\end{itemize}
In any case, the moment condition does not change the critical value $R=\nu^{-1}$, but only the critical behavior. 
\end{remark}

\subsubsection{Relation to the log-likelihood ratio via Lyons' iteration}

Lemma~\ref{LemPemPer} is reminiscent of the iteration~\eqref{eq:Lyonsiteration} for the log-likelihood ratio, namely $r_{\gb}(u) = \sum_{v \colon u\rightarrow v} g_{\gb}(r_{\gb}(v))$. Additionally, \cite[Lem.~4.2]{PemPer10} observes that the function ${g_{\gb}(x) = \log (\frac{e^{2\gb} e^{x}+1}{e^{2\gb} +e^x})}$ verifies
\[
\frac{\tanh(\gb) x}{ (1+ c_2 x^2 )^{1/2}} \leq g_{\gb}(x) \leq \frac{\tanh(\gb) x}{ (1+ c_1 x^2 )^{1/2}} \,,
\]
for some $c_1,c_2>0$.
Therefore, a direct consequence of Theorem~3.2 in~\cite{PemPer10} and of the recursion~\eqref{eq:Lyonsiteration} is the following (again, replacing the $3$-capacity by the  $\frac32$-capacity in~\cite{PemPer10}).

\begin{theorem}[Theorem 3.2 in~\cite{PemPer10}]
\label{thm:PemPer}
Let $\bt$ be a tree with a set of leaves $\partial \bt$. 
There are constants $\kappa_1,\kappa_2$ such that, for the Ising model on $\bt$ with plus boundary condition,  the log-likelihood ratio of the root verifies
\[
\kappa_1 \capa_{\frac32}(\bt) \leq r_{\bt,\gb}^{+}(\rho) \leq \kappa_2 \capa_{\frac32}(\bt)\,.
\]
Here $\capa_{\frac32}(\bt)$ is the $\frac32$-capacity of the tree $\bt$ equipped with resistances $R_u := \tanh(\gb)^{-|u|}$.
\end{theorem}

%This result is used in~\cite{PemPer10} to obtain a criterion for the magnetization of the Ising Model with plus boundary conditions on an infinite tree $\bt$: Theorem~2.2 in \cite{PemPer10} shows that the root $\rho$ is magnetized if and only if $\capa_{\frac32}(\bt)>0$.

The goal of the rest of the paper is therefore to estimate the $\frac32$-capacity of the pruned tree $\bT_n^*$.
We will prove the following more general result.
% note that Theorem~\ref{thm:PemPer} applies with a plus boundary condition, but can be adapted to treat the case of a plus boundary external field, see in particular Remark~\ref{rem:boundary}.

\begin{proposition}
\label{prop:capacitypruned}
Let $\bT_n$ be a Galton--Watson tree with offspring distribution $\mu$ satisfying Assumption~\ref{hyp:branching}, and let~$\bT^*_n = \mathrm{Pruned}_h(\bT_n)$ be the pruned version of $\bT_n$ with $(h_v)_{v\in \partial \bT}$ given by i.i.d.\ Bernoulli random variables with parameter $p_n \in (0,1]$, as defined in Section~\ref{sec:comparePruned}.
Let $p>1$ and $q:= \frac{p}{p-1}$, and define 
\[
\alpha_n =\alpha_n(p) :=
\begin{cases}
 p_n (\tanh(\gb) \nu)^{n} &  \quad \text{ if } \tanh(\gb)\nu\neq 1 \,,\\
\min\{n^{-1/(q-1)}, p_n\} & \quad \text{ if } \tanh(\gb)\nu =1\,.
\end{cases}
\]

\textbullet\ \emph{Upper bound.} There is a constant $c>0$ such that 
\begin{equation}
\label{upperTn*}
\bE\otimes \EE \big[ \capa_{p}(\bT^*_n) \big]  \leq c \,\alpha_n  \,.
\end{equation}
%with 
%\[
%\alpha_n :=
%\begin{cases}
% p_n (\tanh(\gb) \nu)^{n} &  \quad \text{ if } \tanh(\gb)\nu\neq 1 \,,\\
%\min\{n^{-1/(q-1)}, p_n\} & \quad \text{ if } \tanh(\gb)\nu =1\,.
%\end{cases}
%\]
%In particular, if $\tanh(\gb)\nu \leq 1$ or if $\lim_{n\to\infty} p_n (\tanh(\gb) \nu)^n =0$, then $\capa_{p}(\bT^*_n)$ converges to $0$ in $L^1$, hence in probability. 

\textbullet\ \emph{Lower bound.} If $p\geq 2$ and the offspring distribution $\mu$ admits a finite moment of order $q = \frac{p}{p-1} \in (1,2]$ and if $\liminf_{n\to\infty} p_n \nu^n =+\infty$, we have that 
\begin{equation}
\label{lowerTn*}
\lim_{\gep \downarrow 0}\limsup_{n\to\infty} \bP\otimes \PP \big( \capa_{p} (\bT^*_n) \leq \gep \alpha_n \big) =0  \,.
\end{equation}
%with
%\[
%\tilde \alpha_n :=
%\begin{cases}
%\min\{1, p_n (\tanh(\gb) \nu)^{n} \} &  \quad \text{ if } \tanh(\gb)\nu\neq 1 \,,\\
%\min\{n^{-1/(q-1)}, p_n\} & \quad \text{ if } \tanh(\gb)\nu =1\,.
%\end{cases}
%\]
%In particular, if $\tanh(\gb) \nu>1$ and $\liminf_{n\to\infty} p_n (\tanh(\gb) \nu)^n >0$, then $\capa_{p} (\bT^*_n)$ remains asymptotically bounded away from $0$; more precisely, $\{\capa_{\frac32} (\bT^*_n)^{-1}\}_{n\geq 0}$ is tight.
\end{proposition}

\begin{proof}[Conclusion of the proof of Theorem~\ref{thm:reduced}]
Let us start with the case where $\tanh(\gb)\nu \leq 1$ or $\lim_{n\to\infty} p_n (\tanh(\gb) \nu)^n =0$. Then, Proposition~\ref{prop:capacitypruned}-\eqref{upperTn*} with $p=\frac32$ shows that $\capa_{\frac32}(\bT^*_n)$ converges to $0$ in $L^1$, hence in probability. 
Thanks to Theorem~\ref{thm:PemPer}, this shows that the root is not asymptotically magnetized.

If on the other hand $\tanh(\gb) \nu>1$ and $\liminf_{n\to\infty} p_n (\tanh(\gb) \nu)^n >0$, then applying Proposition~\ref{prop:capacitypruned}-\eqref{lowerTn*} with $p=\frac{q}{q-1} \geq 2$ (with $q\in (1,2]$ such that $\mu$ admits a finite moment of order $q$), we obtain that $\capa_{p} (\bT^*_n)$ remains asymptotically bounded away from $0$; more precisely, $\{\capa_{p} (\bT^*_n)^{-1}\}_{n\geq 0}$ is tight.
Now, we can use the fact that $p\mapsto \capa_{p} (\bT^*_n)$ is non-increasing (this directly follows from the Dirichlet principle for the $p$-capacity, see e.g.\ \cite[Thm.~2.10]{CNS21}), so $\capa_{\frac32} (\bT^*_n) \geq \capa_{p} (\bT^*_n)$ also remains asymptotically bounded away from $0$.
Combined with Theorem~\ref{thm:PemPer}, this shows that the root is asymptotically magnetized.
\end{proof}

\begin{remark}
From the monotonicity property of the $p$-capacity, we have that $\alpha_n(\frac{q}{q-1}) \lesssim \capa_{\frac32}(\bT^*_n) \lesssim \alpha_n(\frac32)$, where $a_n\lesssim b_n$ if the ratio $(a_n/b_n)_{n\geq 0}$ is tight, and similar bounds hold for the root magnetization; we stress that these estimates should enable one to \emph{quantify} the influence of the interfering vertices in random graphs.
For instance, when $\tanh(\gb)\nu\neq 1$, observe that $\alpha_n(p)$ does not depend on $p$, which means that we have identified the correct order $\min\{1,p_n (\tanh(\gb)\nu)^n\}$ for the root magnetization, with only a $q$-moment assumption for some $q>1$.
In the critical case $\tanh(\gb) \nu = 1$, the moment condition appears in the upper and lower bound which differ in general (they match when $p_n \leq n^{-1/(q-1)}$).
It would be interesting to obtain matching upper and lower bounds in the critical case.
However, since both bounds converge to zero, we do not develop further in this direction to avoid lengthening the paper.
Nevertheless, the techniques of~\cite{AVB24} should enable one to obtain sharp bounds, see also Remark~\ref{rem:critical}.
\end{remark}

%
%\begin{remark}
%\label{rem:critical}
%For the critical Ising model on the Galton--Watson tree, \textit{i.e.} when one has $\tanh(\gb) \nu = 1$ and $p_n \equiv 1$, then Proposition~\ref{prop:capacitypruned} and Theorem~\ref{thm:PemPer} give (also applying Markov's inequality) that for any $\gep>0$
%\[
%\lim_{\gep\downarrow 0} \liminf_{n\to\infty}\bP\Big(  \frac{\gep}{n}\leq r_{n,\gb_c}^+ (\rho) \leq \frac{\gep^{-1}}{\sqrt{n}} \Big) =1 
%\]
%%Roughly speaking, it says that the magnetization of the root $r_{n,\gb_c}^+(\rho)$ is bounded from above by $O(1/\sqrt{n})$ and from below by $O(1/n)$.
%{\blue
%These bounds can in fact be improved, as done in our subsequent paper~\cite{AVB24}.
%Let us quickly summarize our main findings, in informal terms: we refer to~\cite[Thm.~2.6]{AVB24} for precise statements.
%\begin{itemize}
%\item[(i)] If the offspring distribution admits a finite third moment, then we have that  $\sqrt{n} r_{n,\gb_c}^+ (\rho)$ converges almost surely and in $L^3$ to a tree dependent constant;
%\item[(ii)] If the offspring distribution has a tail $\sum_{k\geq n} \mu(k) = n^{-\alpha +o(1)}$ for some $\alpha \in (1,3)$ as $n\to\infty$, then we have that $r_{n,\gb_c}^+(\rho) = n^{-\frac{1}{\alpha-1} +o(1)}$ as $n\to\infty$.
%\end{itemize}
%
%}
%%To our knowledge, such bounds do not appear in the literature; our lower bound seems not to be optimal so we believe that the root magnetization is of order $1/\sqrt{n}$, at least when the offspring distribution admits a finite second moment.
%\end{remark}

\noindent
Now, the rest of the proof consists in proving Proposition~\ref{prop:capacitypruned}, \textit{i.e.}\ in estimating the $p$-capacity of the pruned tree $\bT_n^*$.
We proceed in two steps, which are split into the next two sections:
\begin{itemize}
\item In Section~\ref{sec:pruning}, we study in detail the properties of $\bT_n^*$.
First of all, we prove that, under $\bP\otimes \PP$, $\bT_n^*$ is an inhomogeneous branching process, whose distribution is explicit. 
Second, we show that the pruned tree exhibits a sharp phase transition: we identify some $k^*= n - \log_{\nu}(\frac{1}{p_n})$ such that the pruned tree looks likes the original Galton--Watson tree up $k^*-O(1)$ and then has very thin branches from $k^*+O(1)$ to $n$.

\item  In Section~\ref{sec:capacitypruned}, we turn to the estimate of Proposition~\ref{prop:capacitypruned}. 
First, we deal with a generic inhomogeneous Galton--Watson tree because it has its own interest; then we use the estimates of Section~\ref{sec:pruning} to conclude the proof.
\end{itemize}

\begin{remark}
\label{rem:order}
In view of the above comments, we can already explain, at least when $\tanh(\gb) \nu > 1$,  where the order $\capa_{p} (\bT_n^*) \approx \min\{ 1, p_n (\tanh(\gb) \nu)^n \} $ comes from in Proposition~\ref{prop:capacitypruned}.
Similarly to~\eqref{eq:capaGW} (and in analogy with Example~\ref{ex:spherical}), one expects that under sufficient moment conditions,
\[
\capa_{p} (\bT_n^*) \asymp_{\PP}   \bigg( \sum_{k=1}^n  \Big( \tanh(\gb)^{k} \prod_{i=1}^k \nu_i^* \Big)^{- s} \bigg)^{-1/s} \,\qquad \text{ with } s := \frac{1}{p-1}\,,
\]
where $\nu_i^*$ is the mean offspring number of the $i$-th generation of $\bT_n^*$ (this is actually true in our context, see Proposition~\ref{prop:capacity}, but requires technical work).
By the phase transition property alluded above, we roughly have $\nu_i^* =\nu$ for $i\leq k^*$ and $\nu_i^* =1$ for $i>k^*$, see Lemma~\ref{lem:boundmk} for a more precise statement. Thus, in the case $\tanh(\gb)\nu>1$, we get that $\capa_{p} (\bT_n^*)$ is of the order
\begin{align*}
  \bigg( \sum_{k=1}^{k^*} (\tanh(\gb) \nu)^{-ks} + \sum_{k=k^*}^n \tanh(\gb)^{-ks} \nu^{-k^*s}  \bigg)^{-1/s} \asymp  \Big( C + \big(p_n (\tanh(\gb)\nu)^n \big)^{-s} \Big)^{-1/s} \,,
\end{align*}
where we have used for the last identity that $\tanh(\gb)<1$ and that $\nu^{k^*} = p_n \nu^{n}$ by the definition of $k^*$.
This corresponds to the claim in Proposition~\ref{prop:capacitypruned}; the case $\tanh(\gb)\nu=1$ is a bit more subtle.
\end{remark}

\section{Detailed study of the pruned tree}
\label{sec:pruning}

\subsection{The pruned tree is an inhomogeneous Galton--Watson tree}

Recall that we denote by $\bP$ the law of $\bT$ and by $\PP$ the law of the i.i.d.\ Bernoulli random variables $(h_v)_{v\in \partial \bT}$.
The main result of this section is that $\bT_n^*$, under $\bP\otimes \PP$, is a Branching Process with \textit{inhomogeneous} \textit{$n$-dependent} generation offspring distributions, that we denote $(\mu_k^*)_{0 \leqslant k \leqslant n-1}$.
The fact that $\bT^*_n$ is a branching process is a priori not obvious, since the pruning of a dead branch in the tree depends on random variables at the leaves.

%\subsubsection{Notation, statement of the result}

Before we state the main result of this section, let us introduce some notation.
For $k\in\{0,\ldots, n\}$, we denote $T_k$, respectively $T^*_k$, the $k$-th generation of $\mathbf{T}$, respectively $\bT_n^*$.
We set, for $k \in \{0,\ldots, n\}$,
\begin{equation}
\label{GammaPruningProb}
\gamma_k := \bP\otimes \PP (Y_u = 0) \quad \text{ for a generic }  u\in T_k\,,
\end{equation}
where $Y_u =\ind_{\{u\in \bT^*_n\}}$ is a Bernoulli random variable that indicates whether the vertex~$u$ has survived the pruning, see~\eqref{def:Xv}.

Here, $\gamma$ is indexed by their generations starting from the root (\textit{i.e.}\ from the bottom of the tree to the top), but it is also helpful to index quantities by their distances to the leaves (\textit{i.e.}\ from the top of the tree to the bottom).
In the rest of the paper, we will denote with a bar quantities that are indexed by the distance to the leaves: for instance, for $k \in \{0,\ldots, n\}$, we set 
\[
\bar \gamma_k := \gamma_{n-k} \,.
\] 
We show below that we can rewrite $\bar\gamma_k$ as $\bar{\gamma}_k = \bE[ (1-p_n)^{|T_k|}]$, see Lemma~\ref{defGamma}-\eqref{def:bargammak}.

We can now state our main result on the random pruned tree $\bT_n^*$.

\begin{proposition}
\label{prop:pruning}
The tree $\bT^*_n = \mathrm{Pruned}_{h}(\bT_n)$ obtained by the pruning of the Galton--Watson tree $\bT_n$ by i.i.d.\ Bernoulli random variables $(h_v)_{v\in \partial\bT_n}$ is, under $\bP\otimes \PP$, an \emph{inhomogeneous} branching process. Its offspring distributions are $(\tilde \mu_0^*, \mu^*_1,\ldots, \mu_{n-1}^*)$, where~$\mu_k^*$ is defined by
\begin{equation}
\label{defmuk}
\mu_k^*(d) =\dfrac{1}{1-\gamma_{k}} \sum_{\ell \geqslant 0} \mu(d+\ell) \binom{d+\ell}{\ell} (\gamma_{k+1})^\ell (1-\gamma_{k+1})^d, \qquad d\in \NN \,,
\end{equation}
and  $\tilde \mu_0^*(d) = (1-\gamma_0) \mu_0^*(d)$ for $d\geq 1$, $\tilde \mu_0^*(0) =\gamma_0$.
\end{proposition}

\noindent
Let us note that $\mu_k^*$ also depends on $n$, through the parameter $\gamma$ (see Lemma~\ref{defGamma}).
Also, notice that $\tilde \mu_0^*(0) =\gamma_0$ is the probability that the root is pruned (\textit{i.e.}\ the whole tree is pruned), and $\mu_0^*$ is the law~$\tilde \mu_0^*$ conditioned on being non-zero.
Hence, conditionally on the root not being pruned, $\bT^*_n$ is an \emph{inhomogeneous} branching process with offspring distribution~$( \mu_k^*)_{0\leq k\leq n-1}$.

\begin{remark}
After some discussion with Loren Coquille, we realized that Proposition~\ref{prop:pruning} could be generalized to some cases where the Bernoulli random variables $(h_v)_{v\in \partial \bT_n}$ are not i.i.d., but correlated in a ``hierarchical'' manner.
More precisely, let $(\tilde h_v)_{v\in \bT_n}$ be independent Bernoulli random variables (attached to all vertices) with respective parameters $\tilde p_v := \tilde p_{|v|}$, and  define $h_v := \prod_{u\colon u \leq v}\tilde h_u $ for each $v\in \partial \bT_n$. 
Note that the Bernoulli variables $(h_v)_{v\in \partial \bT_n}$ are not independent but their dependence respects the tree structure; note also that we recover i.i.d.\ Bernoulli variables if one sets $\tilde p_k = 1$ for all $k<n$.
Then, for this choice of Bernoulli field $h=(h_v)_{v\in \partial \bT_n}$, one can show that under $\bP\otimes \PP$ the pruned tree $\bT_n^* = \mathrm{Pruned}_{h}(\bT_n)$ is a Branching process with inhomogeneous $n$-dependent offspring distributions; we do not provide a proof here to avoid unnecessary technicalities, but it follows the lines as the proof of Proposition~\ref{prop:pruning}.
\end{remark}

\subsubsection{Preliminary observations}

Before we turn to the proof of Proposition~\ref{prop:pruning}, let us make some comments.
Let $G$ be the generating function for the offspring distribution, that is
\[
G(s) := \bE[s^X] = \sum_{d=1}^{\infty} \mu(d) s^d \qquad s\in [0,1] \,,
\]
where $X$ is a random variable of distribution~$\mu$.
Then, the parameters $\gamma_k$ can easily be expressed iteratively in terms of the generating function $G$.

\begin{lemma}
\label{defGamma}
The sequence $(\bar \gamma_k)_{0\leq k \leq n}$ is characterized by the following iteration:
\[
\bar{\gamma}_0 = 1-p_n \quad \text{ and for } k\in \{ 1, \ldots, n\}\,, \ 
\bar{\gamma}_k  = G(\bar{\gamma}_{k-1}) \,.
\]
Equivalently, we have
\begin{equation}
\label{def:bargammak}
\bar{\gamma}_k = \bE\Big[ (1-p_n)^{|T_k|} \Big] \,.
\end{equation}
\end{lemma}

\begin{proof}
First of all, we have $\gamma_n =\PP\otimes \bP(Y_v =0)$ for $v$ a leaf, so by definition~\eqref{def:Xv} of $Y_v$, we have $\gamma_n = \PP(h_v =0) =1-p_n$, \textit{i.e.}\ $\bar\gamma_0=1-p_n$.

Now, notice that for a generic $u\in T_k$ with $k<n$,  by definition of~$Y_u$ we have that $Y_u=0$ if and only if $Y_{u'}=0$ for all $u'\leftarrow u$. In other words, for a fixed realization of $\bT$ we have
\[
\PP(Y_u=0) = \prod_{u', u\rightarrow u'}  \PP(Y_{u'}=0),
\]
since the variables $(h_v)$ are independent on the different sub-trees $\bT(u')$, $u'\leftarrow u$.
Taking the expectation with respect to the tree and using the branching property we therefore get that for $k\leq n-1$,
\[
\gamma_k = \EE\big[ (\gamma_{k+1})^X \big] = G(\gamma_{k+1}) \,.
\]
Using that $\bar \gamma_k = \gamma_{n-k}$, this gives the desired iteration.

To deduce the formula~\eqref{def:bargammak} for $\bar\gamma_k$, we notice that $\bar \gamma_k = G^{\circ k}(1-p_n)$, with $G^{\circ k} = G\circ \cdots \circ G$ ($k$ times). 
We obtain the desired formula since $G^{\circ k}$ is the generating function of $|T_k|$.
\end{proof}

We also have some nice interpretation of the distribution $\mu_k^*$.
Conditionally on the tree not being pruned, we know that each vertex $u\in \bT^*$ has at least one descendant in $\partial \bT$ that has not been pruned; in other words, $\mu_k^*$ is supported on $\NN$.
It turns out that for $u\in T^*_k$, the number of children of vertex $u$ in $\bT^*$ can be constructed as follows: take~$X$ a random variable with distribution $\mu$, so $u$ has $X$ descendants in $\bT$; prune these descendants with probability~$\gamma_{k+1}$ independently, but conditionally on having at least one surviving descendant.

\begin{definition} 
\label{def:ZeroTBin}
Let $n \in \NN$ and $p\in (0,1]$. A random variable $\hat{B}$ is said to follow a \textit{zero-truncated binomial} of parameters $n$ and $p$, and we write $\hat{B} \sim \widehat{\Bin}(n,p)$ if
\[
\P (\hat{B} = k) = \dfrac{1}{1-(1-p)^n} \binom{n}{k} p^k (1-p)^{n-k} \qquad k\in \{1,\ldots, n\}.
\]
Put otherwise, we have $\P(\hat{B} =k) = \PP(B=k \mid B > 0)$ with $B\sim \Bin(n,p)$.
\end{definition}

For a random variable $X$, we also write $W \sim \Bin(X,p)$ if 
%for any $z \in \NN$,
%\[
%\P(W=z) = \sum_{d=1}^{\infty} \P(X=d) \P( B = z) \,, \quad \text{ with } B\sim \Bin(d,p) \,,
%\]
%\textit{i.e.}\ 
$W \sim \Bin(d,p)$ conditionally on $X=d$.
A similar notation holds for $W \sim \widehat{\Bin}(X,p)$.

\begin{lemma}
\label{ZeroTBin}
For all $k\in \{0,\ldots, n-1\}$, we have that $\mu_k^* = \widehat{\Bin}(X, 1-\gamma_{k+1})$ where $X$ is a random variable of distribution $\mu$.
% let $X_k^*$ be a random variable with law $\mu_k^*$.
%Then we have that $X_k^*\sim \widehat{\Bin}(X, 1-\gamma_{k+1})$, where $X$ is a random variable of distribution $\mu$.
%Put otherwise, if $B_k\sim \mathrm{Bin}(X, 1-\gamma_{k+1})$, then $\P(X_k^*=z) = \P(B_k=z \mid B_k >0)$ for $z\in \NN$.
\end{lemma}

\begin{proof}
For $k\in \{0,\ldots, n-1\}$, let $\tilde X_k \sim\widehat{\Bin}(X, 1-\gamma_{k+1}) $.
Then, for $z\in \NN$, we have
\[
\P(\tilde X_k =z) = \P( B_k = z \mid B_k >0 ) = \dfrac{\P (B_k= z)}{1-\P( B_k =0) } \,,
\]
where $B_k \sim \Bin(X,1-\gamma_{k+1})$. 
Then, notice that
\[
\P(B_k = 0) = 
  \sum_{d =1}^{\infty} \mu(d)  (\gamma_{k+1})^d = \bE\big[ (\gamma_{k+1})^X \big]  =\gamma_k \,,
\]
the last inequality deriving from Lemma~\ref{defGamma}.
Therefore, we end up with
\[
\P(\tilde X_k =z) = \dfrac{1}{1-\gamma_{k}} \P (B_k = z) = \dfrac{1}{1-\gamma_{k}}\sum_{d =1}^{\infty} \mu(d) \binom{d}{z} (\gamma_{k+1})^{d-z} (1-\gamma_{k+1})^z = \mu_k^*(z),
\]
which concludes the proof.
\end{proof}

\subsubsection{Proof of Proposition~\ref{prop:pruning}}

Our goal is to write $\bP\otimes\PP ( \mathbf{T}^*_n = \bt^* )$ in terms of the offspring distributions $(\mu_k^*)_{0 \leqslant k \leqslant n-1}$. 
The approach for doing this is to consider all the possible trees $\mathbf{t}$ (sampled from $\mathbf{T}$) such that after pruning we may obtain $\bt^*$. 

First of all, notice that we clearly have $\PP\otimes \bP(\bT^*_n =\emptyset) = \gamma_0$. Indeed, conditionally on the tree $\bT_n$, the probability that the whole tree is pruned is $\PP( h_v =0 \, \forall\, v\in T_n) = (1-p_n)^{|T_n|}$, so $\PP\otimes \bP(\bT^*_n =\emptyset) =\bE[(1-p_n)^{|T_n|}] =\gamma_0$, see Lemma~\ref{defGamma}.

We now work with $\bt^* \neq \emptyset$ of depth $n$ and we write
\begin{equation*}
\PP\otimes \bP (\bT^*_n=\bt^* ) = \sum_{\mathbf{t} \in \BP_n(\mu)} \PP\otimes \bP\big(\mathbf{T} = \mathbf{t},  \mathrm{Pruned}_h(\mathbf{t}) = \bt^* \big)\,.
\end{equation*}
Here, $\BP_n(\mu)$ is the set of all trees $\mathbf{t}$ of depth $n$ generated by a Branching Process of offspring distribution $\mu$.
% and $\text{Pruned}(\bt)$ denotes the pruned tree obtained from $\bt_n$ with $(h_v)_{v\in \partial \bt}$ i.i.d.\ $\mathrm{Bern}(p_n)$.

\smallskip
\noindent
{\it Preliminary calculation: offspring distribution.}
For $u\in \bT_n$, we denote $X_u^*$ the number of descendants of $u$ inside the pruned tree $\bT^*_n$.
% necessarily $d_v^* \geq 1$ since the vertex $v$ has not been pruned (we consider only cases where the root has not been pruned).
Notice that the number of successor of the vertex $u$ on the tree $\mathbf{T}_n$ (\textit{i.e.}\ before pruning), that we denote $X_u$, has to verify $X_u \geqslant X_u^*$.
We have
\[
\{ X^{*}_u=d \} = \bigcup_{\ell \geqslant 0} \bigcup_{ V\subset S(u) , |V|=\ell} \big\{ X_u=d+\ell,  Y_{v}=0 \text{ for } v\in V,\  Y_{v}=1  \text{ for } v \in S(u)\setminus V \big\} \,,
\]
meaning that the vertex $u$ has $d+\ell$ successors (on $\mathbf{T}$) for some $\ell \geq 0$, and exactly $\ell$ of them get pruned.

Using this representation, we can compute the offspring distribution of a vertex~$u$ at generation~$k$.
Using that vertices $v\leftarrow u$ in generation $k+1$ have a probability~$\gamma_{k+1}$ of being pruned and that the events $\{Y_v=0\}_{v\leftarrow v}$ are independent (they depend on Bernoulli variables $(h_v)$ in different sub-trees), we get for $d\geq 0$
\begin{equation}
\label{CompPdv}
\PP \otimes \bP (X^*_u = d )	= \sum_{\ell \geqslant 0} \mu (d+\ell) \binom{d+\ell}{\ell} (\gamma_{k+1})^{\ell} (1-\gamma_{k+1})^{d} \,.
\end{equation}
Notice that 
%we have exactly $\{X_u^* =0\} = \{Y_u=0\}$ since if $Y_u=1$ then $u$ has at leat one descendant leaf $w$ with $h_w=1$; in particular, 
\eqref{CompPdv} generalizes the iteration $\gamma_k := \PP \otimes \bP (Y_u = 0) = G(\gamma_{k+1})$ of Lemma~\ref{defGamma}.

Now, notice that conditionally on $\bT_n^*\neq \emptyset$, a vertex $u\in \bT^*_n$ has at least one descendant in~$\bT^*_n$.
Note that we have 
$\bP\otimes \PP (X_u^* >0 ) = \bP\otimes \PP(Y_u = 1) =1- \gamma_{|u|}$, since $u$ has at least one descendant in the pruned tree if and only if it survives the pruning, \textit{i.e.} $Y_u=1$. 
We therefore get the offspring distribution of a vertex $u\in T^*_k$ in $\bT^*_n$ (conditioned on having $\bT_n^* \neq \emptyset$):
\begin{equation*}
\bP\otimes \PP (X^*_u = d \mid X_u^*> 0) = \dfrac{1}{1-\gamma_k{k}} \sum_{\ell \geqslant 0} \mu (d+\ell) \binom{d+\ell}{\ell} (\gamma_{k+1})^\ell (1-\gamma_{k+1})^d = \mu_k^*(d) \,.
\end{equation*}
It remains to show that the number of descendants in the different generations are independent and that $(X_v^*)_{v \in T_k}$ are also independent.

\smallskip
\noindent
{\it Main calculation: probability of having a given pruned tree.}
Let $\bt^*$ be a non-empty tree of depth $n$ which is a possible candidate for being a pruned version of a Galton--Watson tree $\bT$.
We let $d_u^*$ denote the number of descendants of $u\in \bt^*$; recall that each vertex $u \in \bt^*$ (in generation $k<n$) has to be such has it has at least one descendant, that is $d_u^*\geq 1$.

In order to have $\bT^*_n =\bt^*$, one must have $\bt^*$ as a \textit{squeleton} for $\bT$, see Figure~\ref{fig:pruning}.
Then, similarly as above, we can write the event $\{\bT^*_n =\bt^*\}$ as
\[
 \bigcap_{k=0}^{n-1} \bigcap_{u\in t_k^*} 
\bigcup_{\ell_u \geq 0} \bigcup_{V\subset S(u), |V|=\ell_u}  \big\{ X_u=d_u^*+\ell_u,  Y_{v}=0 \text{ for } v\in V,\  Y_{v}=1  \text{ for } v \in S(u)\setminus V \big\} \,,
\]
meaning that each vertex $u \in \bt^*$ must have $d_u^*+\ell_u$ successors in $\bT$ for some $\ell_u \geq 0$, with exactly $\ell_u$ of them getting pruned; the descendants of $u$ inside $\bt^*$ are the $v$ with $Y_v=1$.

Now, note that the events $\{Y_v =0\}$ in the above are all independent because they depend on different sub-trees of $\bT$ (that lead to leaves with only $h\equiv 0$); recall that $\bP\otimes\PP(Y_v=0) = \gamma_{k+1}$ for $v \in T_k$.
On the other hand, the events $\{Y_v=1\}$ are not independent, since all vertices $v \in \bT$ such that $Y_v=1$ are ancestors of leaves with $h_{\cdot}=+1$.
In particular, having $Y_w=1$ for a leaf $w\in \partial \bT_n$ implies that $Y_v=1$ for all ancestors $v$ of $w$, see Figure~\ref{fig:pruning}.
However, all the events $\{Y_v=1\}$ can be grouped into a simpler event $\{h_w=+1 \,, \forall\, w\in t_n^*\}$, which is independent of all events $\{Y_v=0\}$; note that $\bP\otimes\PP(h_w=1 \,, \forall\, w\in t_n^*) = (p_n)^{|t_n^*|}$.

All together, referring to Figure~\ref{fig:pruning} for an illustration of the computation, and using also that the $(X_v)_{v\in \bT}$ are independent with distribution $\mu$, we obtain that
\begin{equation}
\label{computePt*}
\PP\otimes \bP \big( \bT^*_n=\mathbf{t}^* \big)  = \prod_{k=0}^{n-1} \prod_{u \in t_k^*} \sum_{\ell_u \geqslant 0}  \mu(d_u^*+\ell_u) \binom{d_u^*+\ell_u}{\ell_u} (\gamma_{n})^{\ell_u} \times (p_n)^{|t_n^*|}  \,.
\end{equation}

\begin{figure}[tbp]
\centering
\includegraphics[scale=0.6]{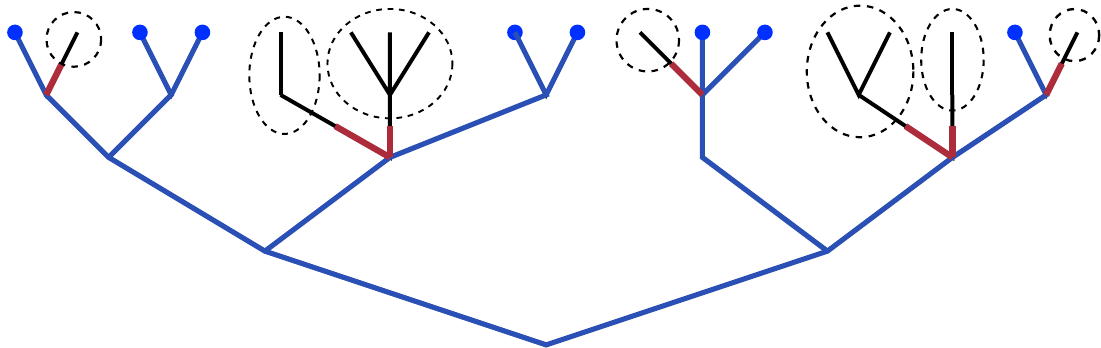}
\caption{\small Illustration of the computation in~\eqref{computePt*}: the tree $\bt^*$ is represented in thick (blue) lines and the tree~$\bT$ whose pruned version is $\bt^*$ is represented in lighter (gray) lines.
For every $h_w=1$ in the leaves (represented by a blue dot), all the ancestors $v$ of $w$ automatically have $Y_v=1$: this contributes to~\eqref{computePt*} by a factor $(p_n)^{|t_n^*|}$.
A branch that is pruned just above generation $k$ contributes to the probability by a factor $\gamma_{k+1}$, and all these pruning events (represented by a red segment) are independent since they depend on distinct subtrees of $\bT$ (the corresponding subtrees are circled with dashed lines).
}
\label{fig:pruning}
\end{figure}

\noindent
We can now reformulate~\eqref{computePt*}.
Notice that $p_n=1-\gamma_n$, and that $|t_n^*| = \sum_{u\in t_{n-1}^*} d_u^*$: we therefore get that
$(p_n)^{|t_n^*|} \prod_{u \in t_{n-1}^*} \sum_{\ell_u \geqslant 0}  \mu(d_u^*+\ell_u) \binom{d_u^*+\ell_u}{\ell_u} (\gamma_{n})^{\ell_u}$ is equal to
\[
\prod_{u \in t_{n-1}^*} \sum_{\ell_u \geqslant 0}  \mu(d_u^*+\ell_u) \binom{d_u^*+\ell_u}{\ell_u} (\gamma_{n})^{\ell_u} (1-\gamma_n)^{d_u^*}
= (1-\gamma_{n-1})^{|t_{n-1}^*|} 
\prod_{u \in t_{n-1}^*}  \mu_{n-1}^*(d_u^*) \,,
\]
where we have used the definition~\eqref{defmuk} of $\mu_{k}^*$ to get that $\sum_{\ell \geqslant 0}  \mu(d+\ell) \binom{d+\ell}{\ell} (\gamma_{n})^{\ell} (1-\gamma_n)^{d} = (1-\gamma_{n-1}) \mu_{n-1}^*(d)$.
Similarly, we get that
\[
(1-\gamma_{n-1})^{|t_{n-1}^*|} \prod_{u \in t_{n-2}^*} \sum_{\ell_u \geqslant 0}  \mu(d_u^*+\ell_u) \binom{d_u^*+\ell_u}{\ell_u} (\gamma_{n-1})^{\ell_u} = (1-\gamma_{n-2})^{|t_{n-2}^*|} 
\prod_{u \in t_{n-2}^*}  \mu_{n-2}^*(d_u^*).
\]
Therefore, iterating, we finally obtain from~\eqref{computePt*} that for any non-empty $\bt^*$,
\[
\PP\otimes \bP (\bT^*_n=\mathbf{t}^*  )  = (1-\gamma_0)  \prod_{k=0}^{n-1} \prod_{u \in t_k^*}  \mu_{k}^*(d_u^*)  \,.
\]
This concludes the proof since this gives that for any $\bt^*$
\[
\PP\otimes \bP (\bT^*_n=\mathbf{t}^* )  = \tilde \mu_0^*(d_\rho^*) \prod_{k=1}^{n-1} \prod_{u \in t_k^*}  \mu_{k}^*(d_u^*)  =: \bP^*(\bT^*_n=\bt^*) \,,
\]
where $\bP^*$ is the law of some inhomogeneous branching process of depth $n$, with offspring distribution $(\tilde \mu^*_0, \mu^*_1,\ldots, \mu^*_n)$. \qed

%\section{Transition in the shape of the pruned tree $\bT^*_n$}
%\label{sec:shape}

\subsection{A sharp transition for the parameters $(\gamma_k)_{0\leq k\leq n}$}
\label{sec:parameters}

Recall that $X$ denotes a random variable with law $\mu$ and that $G(s) := \E[s^X]$ denotes its generating function.
In all this section, we assume that $\mu(0)=0$, $\mu(1)<1$ and $\mu$ admits a finite moment of order $q\in (1,2]$, that we denote $m_q := \E[X^q] <\infty$.
%For $k\geq 1$, we denote $m_q := \E[X^k] \in (1,\infty]$ and we assume that $\nu:=m_1<+\infty$ and $m_2<+\infty$; we denote $\sigma^2 =\mathrm{Var}(X) = m_2-m_1^2$.
We also let
\[
d_0 :=\min\big\{ d\geq 1, \mu(d)>0 \big\} \qquad (d_0\geq 1) \,.
\]

%We now study how the parameters $(\gamma_k)_{0\leq k\leq n}$ vary.
Recall that $\gamma_k$ is defined in~\eqref{GammaPruningProb} as the probability that a vertex at generation $k$ is pruned, and also that we have set $\bar \gamma_k = \gamma_{n-k}$.
The main result of this section is that the parameters $(\gamma_k)_{0\leq k\leq n}$ exhibit a (sharp) phase transition.
Let us define:
\begin{equation}
\label{defk*}
k^* = k^*(p_n) := \log_{\nu} \big( p_n \nu^n \big)\,,\quad \text{ and } \quad \bar k^*= n - k^* = \log_{\nu}\Big(\frac{1}{p_n}\Big) \,,
\end{equation}
where $\log_\nu(x) = \frac{1}{\log \nu} \log x$.
In the following, we omit the integer part to simplify notation and often treat $k^*,\bar{k}^*$ as integers (in practice, one should round above or below depending on whether we are after lower or upper bounds).
Let us note that $k^*\to \infty$ as soon as $p_n \nu^n \to \infty$ and that $\bar k^* \to\infty$ as soon as $p_n\to 0$.

\begin{proposition} 
\label{PhaseTransGammas} 
Assume that $\mu$ admits a finite moment of order $q\in (1,2]$.
There are constants $c_1 \in (0,1)$ and $ 0 < c_2 \leqslant c_3$ (that depend only on the distribution $\mu$) such that:
\[
\text{for all } k \geq k^*  \qquad
c_1 \, \nu^{-(k-k^*)}  \leqslant  1-\gamma_k \leqslant  \nu^{-(k-k^*)}  \,,
\]
and, depending on whether $d_0=1$ (\textit{i.e.} $\mu(1)\in(0,1)$) or $d_0\geq 2$,
\[
\text{for all } k \leq k^*  \qquad
\left\{ \begin{split}
\text{if } d_0=1\,, \quad & \quad\ \   c_2 \mu(1)^{k^*-k} \leqslant  \gamma_k \leqslant c_3 \mu(1)^{k^*-k} \,,\\
\text{if } d_0\geq 2\,, \quad & \quad\  e^{- c_3 (d_0)^{k^*-k}} \leqslant  \gamma_k \leqslant e^{- c_2 (d_0)^{k^*-k} }\, .
\end{split}\right.
\]
\end{proposition}

\noindent
The lower bounds in the above are not particularly relevant to our purpose (they however have some technical use in the proofs), so let us write more compact upper bounds that we will use repeatedly in the proof (whether $d_0=1$ or $d_0\geq 2$). There is a constant $c_4>0$ such that 
\begin{equation}
\label{eq:boundsparameters}
\gamma_k \leqslant e^{-c_4(k^*-k)}  \quad \text{for } k \leqslant k^*
\qquad \text{ and } \qquad
1-\gamma_k \leqslant \nu^{-(k-k^*)}   \quad   \text{for } k \geqslant k^*\, .
\end{equation}

The interpretation of Proposition~\ref{PhaseTransGammas} (or of~\eqref{eq:boundsparameters}) is as follows. 
On one hand, when $k$ is much smaller than $k^*$ in the sense that $k^*-k\gg 1$, then $\gamma_k$ is very small, meaning that vertices have an exponentially small chance of being pruned.
On the other hand, when $k$ is much larger than $k^*$ in the sense that $k-k^*\gg 1$, then $1-\gamma_k$ is very small (or $\gamma_k$ is close to $1$), meaning that vertices are very likely being pruned.
The transition is \emph{sharp}, in the sense that it occurs in a window of size $O(1)$.

One may therefore have the following picture in mind for the pruned tree $\bT_n^*$: for $k\leq k^*$, $\bT_n^*$ resembles the original Galton--Watson tree, while for $k\geq k^*$, $\bT_n^*$ has thin long branches (\textit{i.e.}\ with only one descendant per vertex) recalling that the vertices of the pruned tree are conditioned on having at least one descendant.
This picture is made more formal in Proposition~\ref{prop:d_TV} below, see also Figure~\ref{fig:BeforeAfterPruning} on page~\pageref{fig:BeforeAfterPruning} for an illustration.

\begin{remark}
In the case where $p_n=\alpha^n$ with $\alpha \in (0,1)$, then we have that $\bar k^* =  \eta_{\alpha} n$ with $\eta_{\alpha} :=\log_{\nu}(\frac1\alpha)\in (0,1)$. This shows that the transition occurs sharply around generation $k^* = (1-\eta_\alpha) n$.
\end{remark}

\subsubsection{Some preliminaries}

A first observation is that, from the iterative definition of $\bar \gamma_k$ (see Lemma~\ref{defGamma}) and since~$G$ is convex and $G(1)=1$, we get that $\bar\gamma_{k+1} = G(\bar \gamma_k) \leq \bar \gamma_k$.
We therefore have the following.
\begin{lemma}
\label{MonoGamma}
The sequence $(\bar \gamma_k)_{0\leqslant k \leqslant n}$ is non-increasing; equivalently, $(\gamma_k)_{0\leqslant k \leqslant n}$ is non-decreasing.
\end{lemma}

%As the parameters $(\gamma_k)_{0 \leqslant k \leqslant n}$ are defined recursively in terms of the function $G$, 
In order to study the iterative relation $\bar\gamma_{k+1} = G(\bar \gamma_k)$, let us state the following useful lemma, whose proof is elementary (we include it in Appendix~\ref{app:proofs} for completeness).

\begin{lemma}
\label{lem:BoundsG} 
Assume that $X\sim \mu$ admits a finite moment of order $q\in (1,2]$, and denote $\nu = \E[X]$, $m_q =\E[X^q]$. Then we have the following bounds on $G$: for all  $s\in [0,1]$,
\begin{align}
1 + \nu (s-1) \leqslant &\, G(s) \leqslant 1 + \nu (s-1) + c_q m_q  (s-1)^q,  \label{BoundGin1} \\
\mu(d_0) s^{d_0} \leqslant &\, G(s) \leqslant \mu(d_0) s^{d_0} + s^{d_0+1} , \label{BoundGin0}
\end{align}
where $c_q$ is a universal constant that depends only on $q$.
\end{lemma}

%Lemma~\ref{defGamma} shows that $\bar \gamma_k$ is defined by the iteration $\bar \gamma_{k+1} =G(\bar \gamma_k)$.
Note that we also easily get that the parameters $(1-\bar{\gamma}_k)_{0 \leqslant k \leqslant n}$ can be recursively determined in terms of the function $F(t):=1-G(1-t)$.
Indeed, for $0 \leqslant k \leqslant n$, we have
\[
1- \bar\gamma_{k+1} =1-G(\bar{\gamma}_k) = F(1-\bar{\gamma}_k) \,.
\]
A direct application of~\eqref{BoundGin1} give the following bounds on $F$, that will turn useful in the proofs: for all $t \in [0,1]$,
\begin{equation}
\label{BornesF}
\nu t (1-C_{\mu} t^{q-1}) \leqslant F(t) := 1-G(1-t) \leqslant \nu t \,,
\end{equation}
where $C_{\mu}=c m_q/\nu$.

We conclude by stating the following lemma on recursively defined sequences $u_{j+1}=G(u_j)$.
Again, its proof is elementary but included in Appendix~\ref{app:proofs} for completeness.
\begin{lemma}
\label{Recursionzero} 
Let $\alpha \in (0,1)$ and $u_0 \leqslant 1- \alpha$. Define recursively the sequence $(u_j)_{j\geq 0}$ by the relation $u_{j+1}=G(u_j)$. 
%Let $d_0 := \min\{d \geq 1, \mu(d)>0\}$.
Then:
\begin{itemize}
\item[(i)] If $d_0=1$, \textit{i.e.}\ $\mu(1)\in (0,1)$, there is a constant $C_{\alpha} = C_{\alpha}(\mu)$ such that for all $j \geqslant 1$
\begin{equation}
\label{RecursionIn0}
u_0 \, \mu(1)^{j} \leqslant u_j \leqslant C_{\alpha}\,  u_0 \mu(1)^j.
\end{equation}
\item[(ii)] If $d_0\geq 2$, there is a constant $c_{\alpha}>0$ such that for all $j\geq 1$,
\begin{equation}
(u_0)^{d_0^{j-1}} \mu(d_0)^{\frac{ d_0^j}{d_0-1}} \leqslant u_j \leqslant  e^{- c_{\alpha} d_0^j}\,.
\end{equation}
\end{itemize}
(Note that the decay is doubly exponential in the second case.)
\end{lemma}

\subsubsection{Proof of Proposition~\ref{PhaseTransGammas}}
For our purposes, it will be more convenient to work with the iteration from Lemma~\ref{defGamma} which goes from the leaves to the root.
We therefore prefer to work with $\bar\gamma_k$ instead of $\gamma_k$, and we want to show the following:
\begin{align*}
c_1 \, \nu^{-(\bar{k}^*-k)} & \leqslant  1-\bar{\gamma}_k \leqslant \nu^{-(\bar{k}^*-k)}   &  \text{for all } k \leqslant \bar{k}^*, \\
c_2 \mu(1)^{k-\bar{k}^*} & \leqslant  \bar{\gamma}_k \leqslant c_3 \mu(1)^{k-\bar{k}^*} & \text{for all } k \geqslant \bar{k}^* \,.
\end{align*}
The second line being valid if $d_0=1$ (\textit{i.e.}\ $\mu(1)\in(0,1)$),
and replaced with 
\[
e^{- c_3 (d_0)^{k-\bar k^*}}\leq \bar \gamma_k\leq e^{- c_2 (d_0)^{k-\bar k^*}} 
\] 
in the case where $d_0\geq 2$.
The main idea of the proof is to use the iteration $1-\bar \gamma_{k+1} = F(1-\bar \gamma_k)$ together with the estimate~\eqref{BornesF}, until $1-\bar \gamma_k$ is not small enough; we show that this occurs around $k=\bar k^*$.
Then, we apply Lemma~\ref{Recursionzero} to the iteration $\bar \gamma_{k+1} = G(\bar \gamma_k)$ starting from $k=\bar k^*$ for which $\bar\gamma_{\bar k^*}$ is not too small.

\smallskip
\noindent
{\it Step 1. }
We start with some intermediate result which makes use of~\eqref{BornesF}.
\begin{lemma}
\label{K1*}
Let $\bar{k}_1^*=\min \{ 0 \leqslant k \leqslant n \, : \, \sum_{i=0}^k C_{\mu} (1- \bar{\gamma}_i)^{q-1} > 1/2 \}$, with $C_{\mu}= c m_q/\nu$ as in~\eqref{BornesF}.
Then, for all $k \leqslant \bar{k}_1^*$ we have
\begin{equation}
\dfrac{1}{2} \nu^k p_n \leqslant 1-\bar{\gamma}_k \leqslant \nu^k p_n.
\end{equation}
The upper bound is valid for any $0\leq k\leq n$.
\end{lemma}

\begin{proof}
Let $0 \leqslant k \leq n$. As $1-\bar{\gamma}_k = F(1-\bar{\gamma}_{k-1})$, the right-hand side of the inequality \eqref{BornesF} gives us
$1-\bar{\gamma}_k = F(1-\bar{\gamma}_{k-1}) \leqslant \nu (1-\bar{\gamma}_{k-1}).$
Applying this inequality iteratively to all $1-\bar{\gamma}_i$ with $i \leqslant k-1$, we obtain
\[
1-\bar{\gamma}_k \leqslant \nu^k (1-\bar{\gamma}_{0}) = \nu^k p_n,
\]
where we have used that $1-\bar{\gamma}_0=p_n$. This proves the right-hand side of the inequality.

Using the left-hand side of~\eqref{BornesF}, we have $1-\bar{\gamma}_k \geqslant \nu (1-\bar{\gamma}_{k-1}) (1-C_{\mu} (1-\bar{\gamma}_{k-1})^{q-1})$.
Applying this inequality iteratively to all $1-\bar{\gamma}_i$ such that $i \leqslant k-1$, we have
\begin{equation}
\label{eq:boundbardeltak}
1-\bar{\gamma}_k \geqslant \nu^k (1-\bar{\gamma}_0) \prod_{i=0}^{k-1} \Big( 1-C_{\mu} (1-\bar{\gamma}_{i})^{q-1} \Big) \,.
\end{equation}
By definition of $\bar{k}_1^*$, and the fact that $C_{\mu}$ and $(1-\bar{\gamma}_j)_{0 \leqslant j \leqslant n}$ are non-negative, we have that $C_{\mu} (1-\bar{\gamma}_k)^{q-1} \leqslant 1/2$ for all $k \leqslant \bar{k}_1^*$. Thus, applying the Weierstrass inequality $\prod_{i=0}^{k-1} (1-x_i) \geq 1- \sum_{i=0}^{k-1} x_i$ for $x_i \in [0,1]$, we get that for $0\leq k \leq \bar{k}_1^*$,
\[
\prod_{i=0}^{k-1} \Big( 1-C_{\mu} (1-\bar{\gamma}_{i})^{q-1} \Big) \geqslant 1- \sum_{i=0}^{k-1} C_{\mu} (1-\bar{\gamma}_i)^{q-1} \geq \frac12 \,,
\]
where the last inequality follows again thanks to the definition of $\bar{k}_1^*$.
This concludes the left-hand side of the inequality, using again that $1-\bar{\gamma}_0=p_n$.
\end{proof}

\smallskip
\noindent
{\it Step 2. }
We now show the bounds on $1-\bar \gamma_k$.
First of all, since the bound $1-\bar{\gamma}_k\leq \nu^k p_n$ in Lemma~\ref{K1*} is valid for any $0\leq k\leq n$, this  gives the desired upper bound, recalling that~$\bar{k}^*$ is such that $\nu^{\bar{k}^*} p_n=1$, see~\eqref{defk*}.

For the lower bound on $1-\bar{\gamma}_k$, thanks to Lemma~\ref{K1*} we have a lower bound for all $k\leq \bar{k}_{1}^*$. If $\bar{k}^*\leq \bar{k}_{1}^*$ this concludes the proof, otherwise we need to control $1-\bar\gamma_k$ for $\bar{k}_{1}^*< k \leq \bar{k}^*$.
In fact, we show that $\bar k_1^*$ is comparable to $\bar k^*$, in the sense that there exist two constants $L_1, L_2$, that only depend on the the law $\mu$, such that 
\begin{equation}
\label{k*k1*}
L_1 + \bar{k}^* \leq \bar{k}_1^* \leqslant \bar{k}^* + L_2 \,.
\end{equation}

To get the upper bound in~\eqref{k*k1*}, we use that by the upper bound in Lemma~\ref{K1*}, we have 
\begin{equation*}
\sum_{i=0}^{\bar{k}_1^*} (1-\bar{\gamma}_i)^{q-1} \leq p_n^{q-1} \sum_{i=0}^{\bar{k}_1^*} \nu^{i(q-1)} = p_n^{q-1}\dfrac{\nu^{(\bar{k}_1^*+1)(q-1)} -1}{\nu^{q-1}-1}  \leq  \dfrac{\nu^{q-1}}{\nu^{q-1}-1} \, \nu^{(\bar{k}_1^*-\bar{k}^*)(q-1)} \,,
\end{equation*}
where we have again used that $p_n=\nu^{\bar k^*}$.
Thus, by definition of $\bar{k}_1^*$, we have that $\frac{1}{2 C_{\mu}} \leq \frac{\nu^{q-1}}{\nu^{q-1}-1} \nu^{(\bar{k}_1^*-\bar{k}^*)(q-1)}$,
which yields $\bar{k}_1^* \geq \bar{k}^* + L_1$, for $L_1 := \frac{1}{q-1}\log_{\nu} \big( \frac{\nu^{q-1} -1}{2 \nu^{q-1} C_{\mu}} \big)$.

For the lower bound in~\eqref{k*k1*}, using the lower bound in Lemma~\ref{K1*}, we have similarly as above
\begin{equation*}
\sum_{i=0}^{\bar{k}_1^*-1} (1-\bar{\gamma}_i)^{q-1} \geqslant \dfrac{1}{2^{q-1}} p_n^{q-1} \sum_{i=0}^{\bar{k}_1^*-1} \nu^{i(q-1)} \geq  \dfrac{1}{2^{q-1}} p_n^{q-1} \, \nu^{(\bar{k}_1^*-1)(q-1)} = \frac{1}{(2\nu)^{q-1}} \, \nu^{(\bar{k}_1^*-\bar{k}^*)(q-1)} .
\end{equation*}
Thus, by definition of $\bar{k}_1^*$ we have $\frac{1}{(2\nu)^{q-1}}  \, \nu^{(\bar{k}_1^* -\bar{k}^*)(q-1)} \leqslant \frac{1}{2 C_{\mu}}$,
which yields the bound $\bar{k}_1^* \leqslant \bar{k}^* + L_2$, for $L_2 := \frac{1}{q-1}\log_{\nu} \big( \frac{(2\nu)^{q}}{2C_{\mu}} \big)$.

Therefore, by the lower bound in Lemma~\ref{K1*} and recalling the definition of $\bar{k}^*$, we obtain that $1-\bar{\gamma}_{\bar{k}_1^*} \geq \alpha_1$, for $\alpha_1 :=\frac{1}{2} \nu^{L_1} \leq \frac12$ (recall we are treating the case $\bar{k}_1^* < \bar{k}^*$ so $L_1\leq 0$).
Using that $(1-\bar{\gamma}_k)_{k\geq 0}$ is non-decreasing, we get that $1-\bar{\gamma}_{k}\geq \alpha_1$ for all $ \bar k_1^* \leq k\leq \bar k^*$.
Adjusting the constant (and since $\nu^k p_n$ is of order one for $\bar{k}_1^*\leq k\leq \bar{k}^*$), we therefore get that for all $k \leqslant \bar{k}^*$, 
\begin{equation}
\label{ineq:bardeltak}
1-\bar{\gamma}_k \geq c_1 \nu^k p_n,
\end{equation}
which gives the desired bound, using the definition of $\bar{k}^*$.

\smallskip
\noindent
{\it Step 3. }
We now conclude the proof by showing the bounds on $\bar{\gamma}_k$ for $k \geqslant \bar{k}^*$. 
We have proven above that $1-\bar{\gamma}_{\bar{k}^*} \geqslant  \alpha_1$ for some $\alpha_1<1$ that depends only on the law $\mu$. 
Indeed, this is in the previous paragraph if $\bar{k}_1^{*}< \bar{k}^*$, and follows simply from Lemma~\ref{K1*} if $\bar{k}_1^{*} \geq  \bar{k}^*$ since then $1-\bar{\gamma}_{\bar{k}^*} \geq \frac12 \nu^{\bar k^*} p_n =\frac12$.
We therefore get that $\bar{\gamma}_{\bar{k}^*} \leqslant 1- \alpha_1$. 

We can then apply Lemma~\ref{Recursionzero} to $u_j:=\bar{\gamma}_{\bar{k}^*+j}$, for $j =k-\bar{k}^* \geqslant 0$:
in the case $d_0=1$ ($\mu(1)\in (0,1)$), we obtain 
\[
\bar{\gamma}_{\bar{k}^*} \mu(1)^{k-\bar{k}^*} \leqslant  \bar{\gamma}_k \leqslant C_{\alpha_1} \mu(1)^{k-\bar{k}^*} \,,
\]
and a similar application of Lemma~\ref{Recursionzero} gives the correct upper bound in the case $d_0\geq 2$.

This concludes the upper bound on $\bar \gamma_k$ (the important part), but for the lower bound, we need to show that $\bar\gamma_{\bar k^*} \geq \alpha_2$ for some universal constant $\alpha_2>0$ that depends only on the law~$\mu$.
To see this, note that the upper bound in Lemma~\ref{K1*} gives that $1-\bar \gamma_{\bar k^*-1} \leq \nu^{\bar k^*-1} p_n = \nu^{-1}$, so $\bar\gamma_{\bar k^*} = G(\bar \gamma_{\bar k^* -1}) \leq G(1-\nu^{-1}) =:\alpha_2$.
One can then apply Lemma~\ref{Recursionzero} to obtain the correct lower bounds.
\qed

\subsection{A sharp transition in the shape of the pruned tree}
\label{sec:shape}

We now provide a statement that clarifies the intuition that the offspring distribution~$\mu_k^*$ of the pruned tree $\bT^*_n$ is close to $\mu$ for generations $0 \leqslant k \leqslant k^*$ and close to a Dirac mass at $1$ for generations $k^* \leq k \leq n-1$: this is Proposition~\ref{prop:d_TV}.
We also estimate the mean and ``$q$-variance'' of the offspring distribution $\mu_k^*$ and the mean size of the generation $k$ of the pruned tree $\bT^*_n$: these estimates turn out to be crucial in the proof of our main result, in particular in Section~\ref{sec:estimatecapa} below. 
In all cases, the proofs rely heavily on Proposition~\ref{PhaseTransGammas} and are postponed to Appendix~\ref{app:phasetransition}.
We refer to Figure~\ref{fig:BeforeAfterPruning} for an illustration of this transition in the shape of the pruned tree.

\begin{figure}[t]
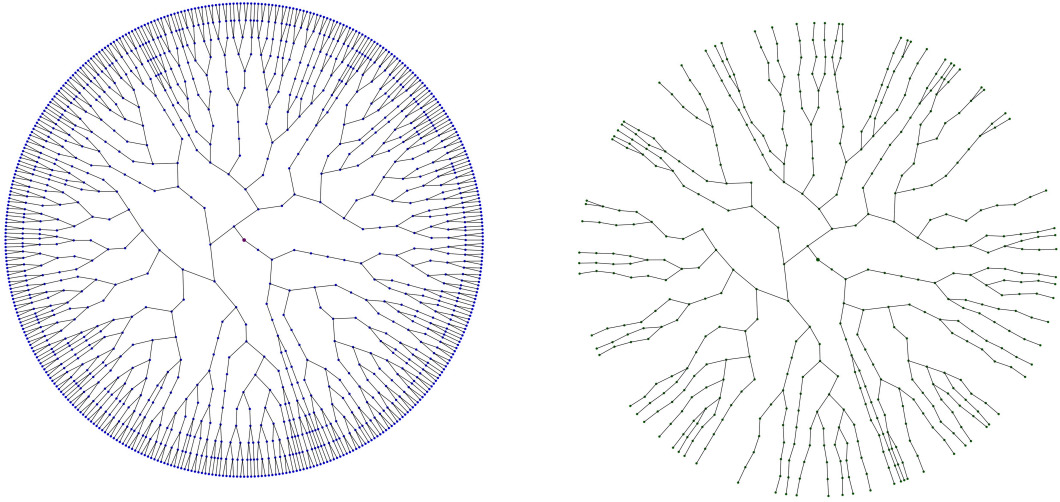

	\centering
	\includegraphics[scale=0.24]{treeStart.pdf}
	\qquad
	\includegraphics[scale=0.24]{treeEnd.pdf}
	\caption{\small Illustration of the transition in the shape of the pruned tree.
On the left, we have represented a tree $\bt$ with mean offspring $\nu = 1.5$ and $n=15$ generations. 
On the right, the picture represents the tree $\bt^*$  obtained by pruning $\bt$ with Bernoulli random variables on the leaves of parameters $p_n=0.2$. 
We can observe a change in the pruned tree $\bt^*$  at generation $k^*\approx 11$ (note that it matches with the definition $k^*=n-\log_\nu(p_n)$, see~\eqref{defk*}, since here $-\log_{1.5} (0.2) \approx 3.97$). For $k\leq k^*$, the tree $\bt^*$ looks like the original tree $\bt$, whereas for $k>k^*$ mainly thin branches remain.}
	\label{fig:BeforeAfterPruning}
\end{figure}

%\begin{figure}[!tbp]
%  \centering
%  \begin{minipage}[b]{0.4\textwidth}
%    \includegraphics[width=\textwidth]{treeStart.jpg}
%    \caption{Flower one.}
%  \end{minipage}
%  \hfill
%  \begin{minipage}[b]{0.4\textwidth}
%    \includegraphics[width=\textwidth]{treeEnd.jpg}
%    \caption{Flower two.}
%  \end{minipage}
%\end{figure}

\subsubsection{About the offspring distribution}

Recall that if $\mu$ and $\tilde \mu$ are two probability measures on a measurable space $(\Omega, \mathcal{A})$, the total variation distance between $\mu$ and $\tilde \mu$ is defined by 
\[
d_{\TV} (\mu, \tilde \mu) := \sup_{A \in \mathcal{A}} \vert \mu(A) - \tilde \mu(A) \vert.
\]
An important property of the total variation is that it can be rewritten as
$d_{\TV} (\mu, \nu) = \inf_{X\sim \mu, Y\sim \tilde \mu} \PP ( X \neq Y)$,
where the infimum is taken over all couplings of $\mu,\tilde\mu$, \textit{i.e.}\ all (joint) distributions for pairs of random variables $(X, Y)$ whose marginal distributions are $X \sim \mu$ and $Y \sim \tilde \mu$.

%Recall the definition~\eqref{defmuk} of $\mu_k^*$, and that $\mu_k^*$ is the offspring distribution of the pruned tree~$\bT^*_n$, see Proposition~\ref{prop:pruning}.
%The following proposition formalizes the previous idea about the shape of the pruned tree~$\bT^*_n$. 

\begin{proposition}
\label{prop:d_TV}
There are some constants $c_4,c_4'>0$ such that
\begin{align*}
d_{\TV} (\mu_{k}^*, \mu)  & \leqslant c_4' e^{- c_4 (k^*-k)} \qquad  \text{ for }  k \leq  k^*\,,
 \\
d_{\TV} (\mu_{k}^*, \delta_1) & \leqslant   2  \nu{-(k-k^*)} \qquad \ \  \text{ for }  k \geqslant k^*
 \,,
\end{align*}
where we denoted by $\delta_1$ the Dirac mass at $1$.
\end{proposition}

\subsubsection{About the mean and $q$-variance of $\mu_k^*$}

For $0 \leqslant k < n$, we let $X_k^*$ be a random variable with distribution $\mu_k^*$ and define its mean $\nu_k^* :=\E[X_k^*]$ and its \emph{$q$-variance} $\sigma_{q,k}^* := \E[(X_k^*)^q] - \E[X_k^*]^q$ for $q\in (1,2]$ (note that one recovers the variance when $q=2$).
We now give estimates on $\nu_k^*$ and $v_{q,k}^*$, and in particular we observe the phase transition around $k^*:=\log_\nu (\nu^n p_n)$.
% but it is of no particular use for the sequel so we postpone estimates on $(\sigma_k^*)^2$ to Appendix~\ref{app:variance}.

%\begin{remark}
%Recall that in $\bT_n^*$ the offspring distribution of the root is $\tilde \mu_0^*$, see Proposition~\ref{prop:pruning}. However, when the root is conditioned on having at least one descendant (\textit{i.e.}\ $\bT_n^*\neq \emptyset$), which is an event of probability $1-\gamma_0$, the offspring distribution of the root is~$\mu_0^*$.
%\end{remark}

%As we have written the means and the variances of the pruned tree in terms of the parameters $(\gamma_k)_{0 \leqslant k < n}$, 
%We now show that the phase transition observed in Proposition~\ref{PhaseTransGammas} translates into the means $(\nu_k^*)_{0\leq k < n}$.
%Recall the definition~\eqref{defk*} of $k^*:=\log_\nu (\nu^n p_n)$.

\begin{lemma}[Mean of $\mu_k^*$]
\label{lem:boundmk}
There are constants $c_4, c_5$ (that depend only on the law $\mu$) such that
\begin{align}
\label{Boundmk-2}
&  0 \leqslant  \nu-\nu_k^*  \leqslant c_5 e^{- c_4 (k^*-k)} &  \text{ for }  k\leq k^*\,, \\
& 0 \leqslant  \nu_k^* -1 \leqslant c_5 \nu^{-(k-k^*)} &  \text{ for }   k \geq k^* \,,
\label{Boundmk-1}
\end{align}
where $c_4$ is the constant appearing in~\eqref{eq:boundsparameters}.
\end{lemma}

\begin{lemma}[$q$-variance of $\mu_k^*$]
\label{lem:BoundVar}
If $X\sim \mu$ admits a finite moment of order $q\in(1,2]$, then $\sigma_{q,k}^* \leq m_q =\E[X^q]$ for all $k\geq 0$.
Moreover, there is a constant $c_6$ (that depend only on the law~$\mu$) such that:
\begin{equation}
  \sigma_{q,k}^* \leqslant c_6 \nu^{-(q-1) (k-k^*)} 
 \qquad  \text{ for all }  k \geq k^* .
\label{Boundvark-2}
\end{equation}
\end{lemma}

%Note that we cannot deduce Lemma~\ref{lem:boundmk} from Proposition~\ref{prop:d_TV} since we cannot control the difference of the means of $\mu,\tilde \mu$ in terms of their distance in total variation.
\noindent
Let us simply stress here that Lemmas~\ref{lem:boundmk} and~\ref{lem:BoundVar} cannot be deduced from Proposition~\ref{prop:d_TV} since the distance in total variation does not allow one to control moments; we defer the proofs to Appendix~\ref{app:phasetransition}.
%Lemma~\ref{lem:boundmk} confirms that $\nu_k^*$ is very close to~$\nu$ if $k^*-k \gg 1$ whereas $\nu_k^*$ is very close to~$1$ if $k-k^*\gg 1$ (note that this 

%Lemma~\ref{lem:BoundVar} complements this by showing that the variance of~$\mu_k^*$ is close to the variance of $\mu$ if $k^*-k \gg 1$ whereas it is close to~$0$ if $k-k^*\gg 1$.
%This confirms that the pruned tree roughly grows as a Galton--Watson tree with mean $\nu>1$ for $k\leq k^*$ (\textit{i.e.}\ with an exponential growth rate $\nu$) and then the tree does not grow much more since the mean is roughly $1$ for $k\geq k^*$.
%We make this statement precise by controlling the mean size of a generation, see Lemma~\ref{lem:boundMk} below.

\subsubsection{About the size of the generations}

%In this section we study the size of the generations in $\bT_n^*$ when it is conditioned on being non empty. We will show that the size of the pruned tree generations grows exponentially up to a certain generation $k^*$, where it stabilizes. 

Recall that $T_k^*$ denotes the set of vertices in the $k$-th generation of the pruned tree~$\bT^*_n$ and $\vert T_k^*\vert$ its size. 
For ease of notations, let us note $Z_k = \vert T_k^*\vert$, for $0 \leqslant k \leq n$.
By Proposition~\ref{prop:pruning}, when the pruned tree is conditioned to be non-empty, $(Z_k)_{0\leq k\leq n}$ is an inhomogeneous branching process with offspring distribution $\mu_k^*$. 
We can construct $(Z_k)_{0\leq k\leq n}$ as follows: set $Z_0=1$ and for $0 \leqslant k < n$ 
\[
Z_{k+1} = \sum_{a=1}^{Z_k} X_{k,a},
\]
where $(X_{k,a})_{a \geqslant 1}$ is a family of i.i.d.\ random variables with distribution $\mu_k^*$.

More generally and for future use, we may want to study for any $0\leq i \leq j\leq n$ the size $Z_{i,j}$ of the population generated by an individual of generation $i$ up to generation~$j$.
The distribution of $Z_{i,j}$ can be constructed as follows:
$Z_{i,i} =1$ and iteratively, for $j\in\{i,\ldots, n-1\}$,
$Z_{i,j+1} = \sum_{a=1}^{Z_{i,j}} X_{k,a}$.
Let us now define, for $0\leq i <j \leq n$,
\begin{equation}
\label{defMij}
M_{i,j}^*:= \E[Z_{i,j}] = \prod_{k=i}^{j-1} \nu_k^*,
\end{equation}
where $M_{i,i}=1$ by convention.
We can now show how the pruned tree growth stabilizes at generation $k^*$;
recall the definition~\eqref{defk*} of $k^*:= \log_{\nu}( \nu^n p_n)$.

\begin{lemma} 
\label{lem:boundMk}
There are constants $c_4$ (from~\eqref{eq:boundsparameters}), $c_7, c_8$, that depend only on the law $\mu$, such that
\begin{align}
\label{BoundMk-2}
 \big( 1-  e^{-c_4(k^*-k)} \big) \nu^k& \leqslant  M_{0,k}^* \leqslant \nu^k
  & \text{ for }  k \leq k^* \, ,\\
 c_7 \nu^{k^*} & \leqslant M_{0,k}^* \leqslant   c_8 \nu^{k^*}
 &  \text{ for }   k \geq k^* \,.
  \label{BoundMk-1}
\end{align}
This can be summarized in a slightly weaker form as follows: for all $0\leq k \leq n$,
\[
c_{7}' \nu^{k\wedge k^*} \leq M_{0,k}^* \leq  c_8 \nu^{ k\wedge k^*} \,.
\]
\end{lemma}

\noindent
This lemma properly shows that $M_{0,k}$ grows as  $\nu^k$ up to $k=k^*$ and then almost completely stops growing.

\section{Estimating the $p$-capacity of $\bT_n^*$: proof of Proposition~\ref{prop:capacitypruned}}
\label{sec:capacitypruned}

In view of Proposition~\ref{prop:pruning}, Proposition~\ref{prop:capacitypruned} amounts to estimating the $\frac32$-capacity of an inhomogeneous branching process. 
We find this question interesting in its own so we try to study this problem in the most general terms. 
We will use the structure on the pruned tree only at the end of the proof, to conclude the argument; the estimates in Section~\ref{sec:shape} turn out to be crucial for this last step.

\subsection{Estimating the capacity of an inhomogeneous branching process}
\label{sec:capainhomogeneous}

Let us consider a tree $\bt^{(n)}$ of depth $n$ generated by a branching process with inhomogeneous, $n$-dependent, offspring distributions $(\mu_k^{(n)})_{0 \leq k < n}$.
Note that the offspring distributions may also depend on $n$, as it is the case for the pruned tree $\bT_n^*$, see Proposition~\ref{prop:pruning} and the definition~\eqref{defmuk} of $\mu_k^*$.

For notational convenience, we will simply write $\mu_k := \mu_k^{(n)}$, and we assume that $\mu_k(0)=0$ for all $k$.
We denote $\nu_k$ the mean of $\mu_k$ (by convention $\nu_n=1$) and for $0\leq i <j \leq n+1$,
\begin{equation}
\label{def:Mij}
M_{i,j} := \prod_{k=i}^{j-1} \nu_k \,,
\end{equation}
with $M_{i,i}=1$ by convention.
Hence, $M_{i,j}$ is the mean size of a population generated by a branching process with offspring distribution $(\mu_k)_{i\leq k<j}$.
For $q\in (1,2]$, we also denote by~$\sigma_{q,k}$ the $q$-variance of $\mu_k$, that is
\[
\sigma_{q,k} := \sum_{i\geq 0} k^q\mu_k(i)  - \Big( \sum_{i\geq 0} k\mu_k(i) \Big)^q \,.
\]
We then have the following result.
% on the $3$-capacity of the tree $T$, with associated resistances $(R_v)_{v\in T}$ given by $R_v := R^{-|v|}$, for some fixed $R>0$.

\begin{proposition} 
\label{prop:capacity}
Let $\bt^{(n)}$ be a tree of depth $n$ generated by an inhomogeneous branching process with offspring distributions $(\mu_k^{(n)})_{0 \leq k < n}$, and associated resistances $(R_u)_{u\in \bt^{(n)}}$ given by $R_u = R^{-|u|}$ for some fixed $R>0$.
Let $p>1$ and $q=\frac{p}{p-1}$.
Then, we have
\begin{equation}
\label{UpperBoundExpCap} 
\bE \big[ \capa_{p} \big(\bt^{(n)} \big)\big] \leqslant  \bigg( \sum_{k=1}^{n} (R^{k} M_{0,k})^{-s} \bigg)^{-1/s} \quad \text{ with } s:= \frac{1}{p-1} =q-1 \,.
\end{equation}
Moreover, if $p\geq 2$ (so $q\in (1,2]$), then we have that for any $\gep>0$,
\begin{equation}
\label{LowerBoundExpCap}
\bP \left( \capa_{p} \big( \bt^{(n)} \big) \leq \gep \Big(\sum_{k=0}^n  v_{k,n} (R^{k} M_{0,k})^{-s} \Big)^{-1/s} \right) \leqslant \bP\big( |t_n^{(n)}| \leq \gep^{\frac{q-1}{q+1}} M_{0,n} \big) +\gep^{\frac{q-1}{q+1}} \,,
\end{equation}
where we have defined $v_{k,n} := 1 + \sum\limits_{i=k}^{n-1} \sigma_{q,i} (M_{k,i})^{-(q-1)}$ for $0\leq k\leq n-1$.
\end{proposition}

Note that if there is some $i_0$ for which $\mu_i$ does not admit a finite $q$-moment, this means that $v_{k,n}=+\infty$ for $k\geq i_0$. Then, the left-hand side of~\eqref{LowerBoundExpCap} becomes ${\bP ( \capa_{p} ( \bt^{(n)} \big) \leq  0)=0}$, so~\eqref{LowerBoundExpCap} is also verified (but not very informative).

\subsubsection{Upper bound in Proposition~\ref{prop:capacity}: proof of~\eqref{UpperBoundExpCap}}

For $u\in \bt^{(n)}$, recall the definition $\phi_p(u) := R_u \capa_p(\bt^{(n)}(u))$ in Lemma~\ref{LemPemPer}.
Now, consider the random variable $\Phi_{k} := \phi_p(u)$ for some vertex $u \in t_k$ chosen uniformly at random from generation $k$.
Then, by Lemma~\ref{LemPemPer} we get for $0\leq k <n$, with $s:= \frac{1}{p-1} =q-1$,
\[
\Phi_{k} = R\, \sum_{i=1}^{X_k} \dfrac{ \Phi_{k+1}^{(i)} }{(1 + (\Phi_{k+1}^{(i)})^s)^{1/s}},
\]
where $X_k$ is a random variable with distribution $\mu_k$ and $(\Phi_{k+1}^{(i)})_{i\geq 1}$ are i.i.d.\ copies of~$\Phi_{k+1}$ since the number of offspring of vertices in the same generation are independent. 
Note that we also have used that $R_v = R^{-|v|}$ here.

Taking the expectation we obtain by the branching property
\[
\bE [\Phi_{k}] = R\, \nu_k \, \bE \bigg[ \dfrac{ \Phi_{k+1} }{(1 + \Phi_{k+1}^s)^{1/s}} \bigg] \,. 
\]
As the function $x \mapsto x/(1+x^s)^{1/s}$ is concave, we can apply Jensen's inequality to the previous expression: we obtain
\begin{equation}
\label{eq:recIneqCap}
\bE [\Phi_{k}] \leqslant  R\, \nu_k \,
\dfrac{\bE[\Phi_{k+1}]}{(1+ \bE [\Phi_{k+1}]^s)^{1/s}} 
= \dfrac{R\,\nu_k}{(1+ \bE [\Phi_{k+1}]^{-s})^{1/s}}.
\end{equation}

\noindent
Now, defining $z_k:=\bE [\Phi_{k}]^{-s}$, we can rewrite \eqref{eq:recIneqCap} as
$z_k \geqslant (R \nu_k)^{-s} (1+z_{k+1})$.
By applying this inequality recursively,  we obtain that for all $k\leq n$,
\[
z_k \geqslant \sum_{i=k}^{n} \prod_{j=k}^{i} (R \nu_j)^{-s} + z_n \prod_{j=k}^{n-1} (R \nu_j)^{-s} \geq  \sum_{i=k}^{n} (R^{i+1-k} M_{k,i+1})^{-s} \,,
\]
where we have used the definition~\eqref{def:Mij} $M_{k,i+1} := \prod_{j=k}^{i} \nu_k$ for the last inequality.
Since $z_0=\bE[\Phi_{0}]^{-s}$ and $\Phi_{0}=\capa_p(\bt^{(n)})$, we get that
\[
\bE \left[ \capa_p (\bt^{(n)})\right] \leqslant  \bigg( \sum_{i=0}^{n-1} \Big( \dfrac{1}{R^{i+1} M_{0,i+1}} \Big)^s \bigg)^{-1/s} \,.
\]
This concludes the proof, up to an index change.
\qed

\subsubsection{Lower bound in Proposition~\ref{prop:capacity}: proof of~\eqref{LowerBoundExpCap}}
\label{sec:Thomson}

Let us define 
\[
A_n:= \sum\limits_{k=0}^n  v_{k,n}(R^{k} M_{0,k})^{-s}  \qquad \text{ with } s= q-1\,.
\]
Using the $p$-resistance $\mathcal{R}_p(\rho\leftrightarrow \partial \bt^{(n)}) = \capa_p(\bt^{(n)})^{-1}$ from Definition~\ref{def:pcapacity}, we need to show that
%
%Recalling that
%$\capa_2(\bt^{(n)})= \mathcal{R}(\rho \leftrightarrow \partial \bt^{(n)})^{-1}$ from Definition~\ref{def:pcapacity}, we are reduced to showing the following inequality
\begin{equation}
\label{eq:boundresistance}
\bP\left( \mathcal{R}_p(\rho \leftrightarrow \partial \bt^{(n)})^{s} \geq \gep^{-1}  A_n \right) \leq \bP\big( W_{0,n} \leq \gep^{\frac{1}{q+1}} \big) + \gep^{\frac{1}{q+1}} \,,
\end{equation}
where $W_{0,n} := |t_n^{(n)}|/M_{0,n}$.
% see~\eqref{eq:defMartingaleW} below.

\smallskip
\noindent
{\it Step 1.} Our first step is to find an upper bound on the resistance $\mathcal{R}_p(\rho \leftrightarrow \partial \bt^{(n)})$.
By the Thomson principle (see~\cite[\S2.4]{LyonsPeres}, or simply the Definition~\ref{def:pcapacity}), and recalling that $q=\frac{p}{p-1}$ is the conjugate exponent of $p$ and $s=q-1$, we get
\begin{equation}
\label{eq:ThompsonPple} 
\mathcal{R}_p(\rho \leftrightarrow \partial \bt^{(n)})^{s} = \inf_{\theta : \vert \theta \vert = 1} \sum_{u \in \bt^{(n)}} R_u^{s} \theta(u)^{q} \,,
\end{equation}
where $\theta$ is a flow on the tree $\bt^{(n)}$.
An upper bound is therefore obtained simply by choosing a specific flow $\theta$ on $\bt^{(n)}$:
similarly to what is done in \cite[Lem.~2.2]{PemPer95} (see also \cite[Lem.~3.3]{ChenHuLin}), we use the uniform flow $\hat \theta$ on $\bt^{(n)}$. 
For a vertex $u \in t_k^{(n)}$ in the $k$-th generation of $\bt^{(n)}$,  we let
\[
\hat \theta(u) := \dfrac{\vert t_{k,n}(u) \vert}{\vert t_{0,n}(\rho) \vert},
\]
where we denote by $t_{k,n}(u)$ the set of leaves that are descendants of the vertex~$u$. Thus, $\vert t_{k,n}(u) \vert$ is the number of descendants of $u$ in generation $n$. Note that $t_{0,n}(\rho) = t_n^{(n)}$ is the set of leaves of $\bt^{(n)}$.
We can easily see that the uniform flow $\hat \theta$ is a unitary flow: indeed, we have
\[
\hat \theta (\rho) 
= \dfrac{1}{\vert t_{n}^{(n)}\vert}\sum_{v \colon \rho \to v} \vert t_{1,n}(v) \vert = 1 \,.
\]
Therefore, by the Thomson principle \eqref{eq:ThompsonPple}, we have 
\[
\mathcal{R}_p\big( \rho \leftrightarrow \bt^{(n)} \big)^{s} \leq \sum_{v \in \bt^{(n)}} R_v^{s} \hat \theta(v)^{q}
 =  \frac{1}{|t_n^{(n)}|^{q}}\sum_{k=0}^n \sum_{u\in t_k^{(n)}} R^{-ks}  |t_{k,n}(u)|^q \, .
\]
Before we work on this upper bound, let us rewrite it using some notation.
For $0\leq k \leq n$ and $u\in t_k^{(n)}$, let us define for $\ell \in\{k,\ldots, n\}$
\begin{equation}
\label{eq:defMartingaleW}
W_{k,\ell}(u) = \dfrac{1}{M_{k,\ell}} \vert t_{k,\ell}(u) \vert \,,
\end{equation}
and notice that $(W_{k,\ell}(u))_{k\leq \ell \leq n}$ is a martingale (with mean $1$).
We also denote $W_{0,k}:= |t_k^{(n)}|/M_{0,k}$ the martingale $W_{0,k}(\rho)$. 
Then, we have
\begin{equation}
\label{eq:upperResistance}
\begin{split}
\mathcal{R}_p(\rho \leftrightarrow \bt^{(n)})^{s} &\leq \frac{1}{(W_{0,n})^q} \sum_{k=0}^n R^{-ks} \Big(\frac{M_{0,n}}{M_{k,n}} \Big)^q \sum_{v\in t_k^{(n)}} W_{k,n}(v)^q
\\
& \leq \frac{1}{(W_{0,n})^q} \sum_{k=0}^n (R^kM_{0,k})^{-s} \frac{1}{M_{0,k}} \sum_{v\in t_k^{(n)}} W_{k,n}(v)^q \,,
\end{split}
\end{equation}
where we have used that $M_{0,n}= M_{0,k}M_{k,n}$.

\smallskip
\noindent
{\it Step 2.} We are now ready to conclude the proof.
Using the above inequality and decomposing according to whether $W_{0,n}$ is small or not, we have
\begin{align*}
\bP\Big( \mathcal{R}_p(\rho \leftrightarrow T)^{s} &\geq \gep^{-1}   A_n \Big) \leq \bP\bigg( \sum_{k=0}^n (R^kM_{0,k})^{-s}   \frac{1}{M_{0,k}} \sum_{v\in t_k^{(n)}} W_{k,n}(v)^q \geq \gep^{-1} (W_{0,n})^q A_n\bigg) \\
& \leq \bP\big( W_{0,n} \leq \gep^{\frac{1}{q+1}} \big) + \bP \bigg(\sum_{k=0}^n (R^kM_{0,k})^{-s}  \frac{1}{M_{0,k}} \sum_{v\in t_k^{(n)}} W_{k,n}(v)^q \geq \gep^{-\frac{1}{q+1}} A_n \bigg) \\
& \leq \bP\big( W_{0,n} \leq \gep^{\frac{1}{q+1}} \big) + \gep^{\frac{1}{q+1}}  A_n^{-1} \sum_{k=0}^n (R^kM_{0,k})^{-s}  \bE\big[ (W_{k,n})^q\big] \,,
\end{align*}
where we have used Markov's inequality for the last part, together with the fact that $(W_{k,n}(v))_{v\in T_k}$ are i.i.d.; here $W_{k,n}$ denotes a random variable with the same distribution.

We now show that $\bE\big[ (W_{k,n})^q\big] \leq v_{k,n}$, by making use of the following lemma by Neveu. 
\begin{lemma}[p.~229 in \cite{neveu87}]
\label{lem:Neveu}
Let $q\in (1,2]$ and let $\xi_1,\ldots, \xi_\ell$ be non-negative independent random variables with a finite moment of order $q$. Then we have that
\[
V_q\Big(\sum_{i=1}^{\ell} \xi_i \Big) \leq \sum_{i=1}^\ell V_q(\xi_i) \,,
\]
where $V_q(\xi) = \E[\xi^q]-\E[\xi]^q$ is the $q$-variance of $\xi$.  
\end{lemma}

Since $(|t_{k,\ell}|)_{0\leq \ell \leq n}$ is an inhomogeneous branching process with offspring distribution $(\mu_\ell)_{k\leq \ell <n}$, we can apply Lemma~\ref{lem:Neveu} conditionnally to the first $\ell$ generations: we obtain that
\[
\bE\big[ |t_{k,\ell+1}|^q\big] \leq  \bE\big[(\nu_{\ell} |t_{k,\ell}| )^q\big] + \sigma_{q,\ell} \bE\big[ |t_{k,\ell}|\big] = (\nu_{\ell})^q \bE\big[ |t_{k,\ell}|^q\big] + \sigma_{\ell}^2 M_{k,\ell} \,,
\]
where we recall that $\sigma_{q,\ell}$ is the $q$-variance of the law $\mu_{\ell}$. 
Iterating, we get that
\[
\bE\big[ |t_{k,n}|^q \big] \leq (M_{k,n})^q+ \sum_{i=k}^{n-1} (M_{i,n})^q \sigma_{q,i} M_{k,i} \,,
\]
so using that $M_{k,n}=M_{k,i} M_{i,n}$ we end up with
\[
\bE[ (W_{k,n})^q] \leq 1 + \sum_{i=k}^{n-1} \sigma_{q,i} (M_{k,i})^{-(q-1)}  =: v_{k,n} \,.
\]
This concludes the proof of~\eqref{LowerBoundExpCap}, recalling the definition of $A_n$.
\qed

\subsection{Estimating the $p$-capacity of $\bT_n^*$}
\label{sec:estimatecapa}

We now apply Proposition~\ref{prop:capacity} to obtain an upper and a lower bound on $\capa_{p}(\bT^*_n)$, which is the content of Proposition~\ref{prop:capacitypruned}.

\subsubsection{Upper bound on $\capa_{p}(\bT^*_n)$: proof of~\eqref{upperTn*}}
By Proposition~\ref{prop:capacity} and Theorem~\ref{thm:PemPer}, we have that
\[
\bE\otimes \EE \big[ \capa_p(\bT_n^*) \big] \leq  \bigg(\sum_{k=1}^{n} (R^{k} M_{0,k}^*)^{-s}  \bigg)^{-1/s} \,,
\]
with $s=\frac{1}{p-1}$, $R:=\tanh(\gb) <1$ and $M_{i,j}^*$ defined in~\eqref{defMij}.
We now use Lemma~\ref{lem:boundMk} to get that 
\begin{equation}
\label{eq:Kn}
c_7^{-1} K_n \leq \sum_{k=1}^{n} (R^{k} M_{0,k}^*)^{-s} \leq c_6^{-1} K_n\,,\qquad
\text{ with }  K_n := \sum_{k=1}^{n} R^{-ks} \, \nu^{ - (k\wedge k^*) s}  \,.
\end{equation}
Now, we can easily study $K_n$ depending on the value of $R \nu$; note that we only need a lower bound on $K_n$ since the above shows that $\bE\otimes \EE[\capa_{p}(\bT_n^*)]\leq c (K_n)^{-1/s}$.

\smallskip
\textbullet\ If $R\nu \neq  1$, then keeping only the term $k=n$ in the sum, we have
\[
K_n \geq  \nu^{-k^* s} R^{-ns}  =  \big( p_n (R\nu)^n  \big)^{-s} \,,
\]
where we have used that $n-k^* = \log_{\nu}(\frac{1}{p_n})$.
We end up with
\[
\bE\otimes \EE\big[ \capa_{p}(\bT_n^*) \big] \leq c (K_n)^{-1/s} \leq c\,  p_n (R \nu)^{n} \,,
\]
which gives~\eqref{upperTn*} in the case $R\nu\neq 1$ (in any case, $\capa_{p}(\bT_n^*) \leq $ since $K_n \geq (R \nu)^{-s}$ by considering only the term $k=1$ in the sum).
%Note that it goes to zero if $\lim_{n\to\infty} p_n (R \nu)^{n} =0$.

\smallskip
\textbullet\ If $R\nu=1$, then we have
\[
K_n \geq \sum_{k=1}^{k^*} 1  + \sum_{k=k^*+1}^{n}  R^{-(k-k^*)s} \geq  k^* + R^{- (n-k^*)s} \geq  \max\{k^*, p_n^{-s}\}\,,
\] 
where we have used that $R^{n-k^*} = \nu^{-(n- k^*)}$ with $\nu^{n-k^*} = \nu^{\bar k^*} =\frac{1}{p_n}$, by the definition~\eqref{defk*} of $\bar k^*$.
Now, notice that $k^* = n- \log_{\nu}(\frac{1}{p_n})$ so either $p_n \leq n^{-1/s}$, or $p_n \geq n^{-1/s}$ and then $k^* \geq n/2$.
We therefore end up with $K_n \geq c \max\{ n , p_n^{-s}\}$, so that
\[
\bE\otimes\EE\big[ \capa_{p}(\bT_n^*) \big] \leq  c (K_n)^{-1/s} \leq c' \min\big\{ n^{-1/s}, p_n \big\} \,.
\]
This concludes the proof of~\eqref{upperTn*} in the case $R\nu= 1$, recalling that $s=q-1$.
\qed

\subsubsection{Lower bound on $\capa_{p}(\bT^*_n)$: proof of~\eqref{lowerTn*}}

Applying Proposition~\ref{prop:capacity}, we need to control the two quantities in~\eqref{LowerBoundExpCap}. We treat them in the two following lemmas, whose proofs are postponed to Section~\ref{sec:lastlemmas}.

\begin{lemma}
\label{lem:vkn}
Let $\sigma_{q,k}^*$ be the $q$-variance of the law $\mu_k^*$ and $M_{k,i}= \prod_{j=k}^{i-1} \nu_j^*$ as defined in~\eqref{defMij}.
Then, there exists a constant $C>0$ such that for all $0\leq k \leq n$
\[
v_{k,n}^*:=1+ \sum_{i=k}^{n-1} \sigma_{q,i}^* (M_{k,i}^*)^{-(q-1)} \leq C \,.
\]
\end{lemma}

\begin{lemma}
\label{lem:convmartingale}
Let $W_{0,n}^* = \frac{1}{M_{0,n}^*} |T_n^*|$. Then, if $\liminf_{n\to\infty} p_n \nu^n = +\infty$, for any $\gep>0$ we have 
\begin{equation*}
\lim_{\gep \downarrow 0}\limsup_{n\to\infty} \PP\big(W_{0,n} \leqslant \gep \big) =0\,.
\end{equation*}
\end{lemma}

\noindent
Combining these two lemmas, Proposition~\ref{prop:capacity} and the bound~\eqref{eq:Kn}, we obtain that  
\[ 
\lim_{\gep\downarrow 0}\limsup_{n\to\infty} \bP\otimes \PP\big( \capa_{p}(\bT^*_n) \leq \gep (K_n)^{-1/s} \big) =0 \,.
\]
%Using again Lemma~\ref{lem:boundMk} to bound $M_{0,k}^*$, we get that
%$c  K_n \leq \sum_{k=1}^n (R^k M_{0,k}^*)^{-s} \leq c'  K_n$ with $K_n$ as in
%so that  
%\[
%\lim_{\gep\downarrow 0}\limsup_{n\to\infty} \bP\otimes \PP( \capa_{\frac32}(\bT^*_n) \leq \gep (\tilde K_n)^{-1}) =0 \,.
%\] 
It only remains to get an upper bound on~$K_n$, depending on the value of $R\nu$.
 
\smallskip
\textbullet\ If $R\nu\neq 1$,  since $R=\tanh(\gb)<1$, we have that 
\begin{align*}
 K_n & = \sum_{k=1}^{k^*} (R\nu)^{-ks} +\nu^{-k^* s}  \sum_{k=k^*+1}^n R^{-ks} \\
 & \leq  \frac{1}{(R\nu)^{s}-1}  \big( 1-(R\nu)^{-k^*s} \big) + c \nu^{-k^*s} R^{-ns} \leq c' + c''(R\nu)^{-k^*s}  + c \big(p_n (R\nu)^{n}\big)^{-s} \,,
\end{align*}
where we have used again that $\nu^{n-k^*} = 1/p_n$ by definition~\eqref{defk*} of $k^*$.
In the case $R\nu<1$ (where we already know that $\capa_{p}(\bT^*_n)$ goes to $0$ in probability) we get that the upper bound on $K_n$ diverges.
In the case $R\nu>1$ the upper bound is a constant times $\max \{1, (p_n (R\nu)^n)^{-s}\}$. 
By taking the power $-1/s$, this concludes the proof of~\eqref{lowerTn*} in the case $R\nu \neq 1$.

%Note that in the case where $R\nu>1$ and if $\liminf_{n\to\infty} p_n (R\nu)^{n} >0$, then we have $\limsup_{n\to\infty} (\tilde K_n)^{-1} <+\infty$, so it proves in particular that $\lim_{\gep\downarrow 0}\limsup_{n\to\infty} \bP\otimes \PP\big( \capa_{\frac32}(\bT^*_n) \leq \gep \big) =0$.

\smallskip
\textbullet\ If $R\nu =1$, then similarly to the above, we have
\[
K_n \leq k^* + c \nu^{-k^*} R^{-n} \leq c' \max \{k^* , \nu^{(n-k^*)s}\}  \leq  c' \max \big\{n , p_n^{-s} \big\}
\]
recalling that $\nu^{n-k^*} =1/p_n$.
By taking the power $-1/s$, this concludes the proof of~\eqref{lowerTn*} in the case $R\nu=1$.
\qed

\subsubsection{Last technical lemmas: proof of Lemmas~\ref{lem:vkn} and~\ref{lem:convmartingale}}
\label{sec:lastlemmas}

\begin{proof}[Proof of Lemma~\ref{lem:vkn}]
We start with the case where $k\geq k^*$. Then,  splitting the sum in the definition of $v_{k,n}^*$ (see Proposition~\ref{prop:capacity}) according to whether $i \leq k^*$ or $i \geq k^*$, we obtain
\[ 
v_{k,n}^* =1+ \sum_{i=k}^{k^*} \sigma_{q,i}^* (M_{k,i}^*)^{q-1} + \sum_{i=k^*+1}^{n-1} \sigma_{q,i} (M_{k,i}^*)^{q-1} 
\leq 1+ c \sum_{i=k}^{k^*} \nu^{-(q-1)(i-k)} + \sum_{i=k^*+1}^{n-1} \sigma_{q,i}^*\,,
\]
where we have used that $\sigma_{q,i} \leq \E[X^q]$, see Lemma~\ref{lem:BoundVar}, that $ M_{k,i}^* = \frac{M_{0,i}^*}{M_{0,k}^*} \geq c \nu^{i-k}$ for $k\leq i\leq k^*$ thanks to Lemma~\ref{lem:boundMk}, and finally that $M_{k,i}^* \geq 1$ in the case $i\geq k^*$.
The first sum is finite since $\nu>1$, and the second sum is also finite, using the bound~\eqref{Boundvark-2} in Lemma~\ref{lem:BoundVar}.

In the case where $k<k^*$, then we have similarly
\[
v_{k,n}^* =1+ \sum_{i=k}^{n-1} \sigma_{q,i}^* (M_{k,i}^*)^{q-1} \leq 1+ \sum_{i=k}^{n-1} \sigma_{q,i}^* \,,
\]
that last sum being also bounded by a universal constant.
\end{proof}

\begin{proof}[Proof of Lemma~\ref{lem:convmartingale}]
First of all, let us observe that in the general case of a branching process with inhomogeneous, $n$-dependent, offspring distributions $(\mu_{k}^{(n)})_{1\leq k \leq n}$, one cannot use a martingale convergence as one would for a homogeneous branching process.
% in particular when considering an array of offspring distributions, \textit{i.e.}\ when the offspring distribution $\mu_k$ depends on $n$ (as it is the case for $\mu_k^*$, see~\eqref{defmuk}).
For the pruned tree, we now give some \textit{ad-hoc} proof, which uses the structure of the tree. 
%For completeness, we give in Appendix~\ref{app:alternative} and alternative line of proof that one should follow in the general case. 

In the case of the pruned tree $\bT_n^*$, we use the fact that the distribution of $|T_n^*|$ is explicit: conditionally on $|T_n|$, it is a binomial $\Bin(|T_n|,p_n)$.
In particular, 
\[
M_{0,n}^* := \bE\big[ |T_n^*| \mid |T_n^*| >1 \big] = (1-\gamma_0)^{-1} \bE\big[ |T_n^*| \big] = (1-\gamma_0)^{-1} p_n \nu^n \,,
\]
with $\gamma_0$ close to $0$.
We therefore get that
\[
\begin{split}
\bP\big( |T_n^*| \leq \gep M_{0,n} \big)
&\leq \bP\big( |T_n|  \leq 2\gep \nu^n \big) + \bP\big(  |T_n|  \geq 2\gep \nu^n , \Bin(|T_n|,p_n) \leq \gep (1-\gamma_0)^{-1} p_n \nu^n \big)  \\
& \leq \bP\big( |T_n|  \leq 2\gep \nu^n \big) + \bP\big(\Bin(2 \gep \nu^n,p_n) \leq \gep (1-\gamma_0)^{-1} p_n \nu^n \big) \,.
\end{split} 
\]
The second probability clearly goes to $0$, as long as $\liminf_{n\to\infty} p_n \nu^n = +\infty$ (using also that $\gamma_0$ goes to $0$).
For the first probability, we use that $\bT$ is a homogeneous branching process: we have that $(\frac{1}{\nu^n} |T_n|)_{n\geq 0}$ is a martingale that converges almost surely to a \emph{positive} random variable $\mathcal{W}$ (recall we assumed that $\mu(0)=0$).
Hence, we get that
\[
\limsup_{n\to\infty} \bP\big( |T_n^*| \leq \gep M_{0,n} \big) \leq \bP(\mathcal{W} \leq 2\gep) \,,
\]
which goes to $0$ as $\gep\downarrow 0$ and concludes the proof.
\end{proof}

\appendix
\section{Some technical proofs and comments}

\subsection{About recursions: proofs of Lemmas~\ref{lem:BoundsG} and~\ref{Recursionzero}}
\label{app:proofs}

\begin{proof}[Proof of Lemma~\ref{lem:BoundsG}]
Recall that $G$ is the generating function of a random variable $X\in \NN$ with mean $\nu$ and a finite moment of order $q \in (1,2]$, that we denote $m_q:= \E[X^q]$.

The lower bound in~\eqref{BoundGin1} is obvious since by the Taylor-Lagrange theorem, for all $s \in [0,1]$ there exists $x_1 \in [s,1]$ such that:
\begin{equation*}
G(s)  = G(1) + G'(1)(s-1) + \dfrac{G''(x_1)}{2} (s-1)^2  \geq 1 + \nu(s-1)\,,  \label{BorneInfG}
\end{equation*}
using that $G''(x_1)\geq 0$ by convexity.

For the upper bound in~\eqref{BoundGin1}, since $\sum_{k\geq 0}\mu(k)=1$ and $\sum_{k\geq 0} k\mu(k) =\nu$, we can write, for $s\in (0,1)$ 
\[
G(s) - (1-\nu(1-s)) = \sum_{k=0}^{\infty} \big[s^k -(1-k(1-s))\big] \mu(k)  \leq \sum_{k=0}^{\infty} \big[e^{-k(1-s)} -(1-k(1-s)) \big]\mu(k) \,.
\]
Splitting the sum at $k_s:= \lfloor \frac{1}{1-s} \rfloor$, and using that 
\[
e^{-k(1-s)} -(1-k(1-s)) \leq e^{-1} (k(1-s))^2 \leq e^{-1} (k(1-s))^q \quad \text{ for } k\leq k_s
\]
(since then $k(1-s)\leq 1$ and $e^{-x}\leq 1-x+e^{-1} x^2$ for $x\in [0,1]$), we obtain that
\[
G(s) - (1-\nu(1-s)) \leq e^{-1} (1-s)^q \sum_{k=1}^{k_s} k^q \mu(k) + \sum_{k=k_s+1}^{+\infty} e^{- k(1-s)} \mu(k) \,.
\]
The first term is bounded by $e^{-1} m_q (1-s)^q$, so we focus on the second term.
We write it as
\[
\sum_{j=0}^{\infty} \sum_{k=2^{j-1} (k_s+1)}^{2^j (k_s+1)} e^{- k(1-s)} \mu(k) \leq \sum_{j=0}^{\infty} e^{-2^{j-1}} \P\big( X \geq  2^{j-1} (1-s)^{-1} \big) \,,
\]
using also that $(k_s+1)(1-s)\geq 1$.
Applying Markov's inequality to bound the last probability, we end up with
\[
\sum_{k=k_s+1}^{+\infty} e^{- k(1-s)} \mu(k) \leq \sum_{j=0}^{\infty} e^{-2^{j-1}} 2^{-q(j-1)} (1-s)^q m_q =: c' m_q (1-s)^q \,,
\]
which concludes the proof.
 
%By convexity of the function $G$, we have that $0\leq G''(x_1) \leq G''(1) = m_2-m_1$. 
%Since $G(1)=1$, $G'(1)=m_1$, we therefore obtain from \eqref{BorneInfG} that for all $s \in [0,1]$
%\[
% 1+m_1 (s-1) \leq G(s) \leq  1+m_1 (s-1) + \frac12 (m_2-m_1) (s-1)^2 \,,
%\]
%which is the desired bound.

The bounds of \eqref{BoundGin0} are easily deduced from the definition $G(s)= \sum_{i\geqslant 1} \mu(i) s^{i}$, using that $s^i\leq s^{d_0+1}$ for $i\geq d_0+1$ and $\sum_{i\geq d_0+1} \mu(i) \leq 1$ for the upper bound. 
\end{proof}

\begin{proof}[Proof of Lemma~\ref{Recursionzero}]
Thanks to the bounds \eqref{BoundGin0}, we have for any $j \geqslant 1$
\[
\mu(d_0)\; (u_{j-1})^{d_0} \leqslant u_j = G(u_{j-1}) \leqslant \mu(d_0) \; (u_{j-1})^{d_0} + (u_{j-1})^{d_0+1}.
\]
%= \mu(1) U_{j-1} \left( 1+ \dfrac{U_{j-1}}{\mu(1)} \right) 
Thus, by iteration, we obtain for $j \geqslant 1$, 
\begin{equation}
\label{eq:formulaiterate}
\prod_{i=0}^{j-1} \mu(d_0)^{d_0^i} \times (u_0)^{d_0^{j}} \leqslant u_{j} \leqslant \prod_{i=0}^{j-1} \mu(d_0)^{d_0^i} \times (u_0)^{d_0^{j-1}} \times \prod_{i=0}^{j-1} \left( 1 + \dfrac{u_{i}}{\mu(d_0)} \right)^{d_0^{j-1-i}}.
\end{equation}

\smallskip
(i) In the case $d_0=1$, the lower bound is directly given by~\eqref{eq:formulaiterate}.
For the upper bound, we only need to control $\prod_{i=0}^{j-1} ( 1 + \frac{1}{\mu(1)} u_{i})$. We use that, $G$ being convex, $G(s) \leqslant p_{\alpha} s$ for all $s\in [0,1-\alpha]$, with  $p_{\alpha}= \frac{G(1-\alpha)}{1-\alpha}<1$:
we easily deduce by iteration that $u_i \leqslant (p_{\alpha})^{i} u_0$ for all $i\geq 0$.
This leads to the following:
\[
\prod_{i=0}^{j-1} \Big( 1 + \dfrac{u_{i}}{\mu(1)} \Big) \leqslant \exp\Big( \dfrac{1}{\mu(1)} \sum_{i=0}^{j-1}  u_i \Big) \leqslant \exp\Big( \dfrac{u_0}{\mu(1)} \dfrac{1}{1-p_{\alpha} } \Big) \leq C_{\alpha} ,
\]
with $C_{\alpha}:=e^{\frac{1-\alpha}{\mu(1) (1-p_{\alpha})}}$, which proves \eqref{RecursionIn0}.

\smallskip
(ii) In the case $d_0\geq 2$, the lower bound is also immediate from~\eqref{eq:formulaiterate}, since $\sum_{i=0}^{j-1} d_0^i =  \frac{d_0^{j}-1}{d_0-1}$.
For the upper bound, let us start by noting that if $\mu(d_0)=1$, then we have the exact formula $u_j = (u_0)^{d_0^{j-1}}$, with $u_0\leq 1-\alpha$.
If $\mu(d_0)<1$, then there is some $v_0 \in (0,1)$ such that $G(s) \leq s^{d_0}$ for all $s\leq v_0$.
Since there is some $j_{\alpha}$ such that $u_j\leq v_0$ for all $j\geq j_{\alpha}$, we get as above that for $j\geq j_{\alpha}$
\[
u_{j} \leqslant (u_{j_{\alpha}})^{d_0^{j-j_{\alpha}}} \leq \exp\big( - (d_0^{-j_{\alpha}} \log v_0)  d_0^j \big)\,.
\]
Adjusting the constant to deal with the terms $j< j_{\alpha}$, this gives the desired bound.
\end{proof}

\subsection{About the offspring distribution $\mu_k^*$}
\label{app:phasetransition}

In this section, we prove Proposition~\ref{prop:d_TV} and Lemmas~\ref{lem:boundmk}-\ref{lem:BoundVar} and~\ref{lem:boundMk}, which show the sharp phase transition in the pruned tree;
all these results rely on Proposition~\ref{PhaseTransGammas}

Recall Lemma~\ref{ZeroTBin} which says that $X_k^* \sim \mu_k^* =\widehat{\Bin}(X,1-\gamma_{k+1})$, where $\widehat{\Bin}$ denotes a zero-truncated binomial, \textit{i.e.}\ a binomial conditioned on being strictly positive, see Definition~\ref{def:ZeroTBin}, and $X$ is a random variable with law $\mu$.

\subsubsection{Proof of Proposition~\ref{prop:d_TV}}
Let us start with the case $k\leq k^*$.
Since $X_k^*\sim\mu_k^* =\widehat{\Bin}(X,1-\gamma_{k+1})$ by Lemma~\ref{ZeroTBin}, it provides a natural coupling $(X_k^*, X)$ with $X_{k}^*\sim \mu_k^*$ and $X\sim \mu$, where the conditional law of $X_k^*$ given $X$ is a zero-truncated Binomial $\widehat{\Bin}(X,1-\gamma_{k+1})$.
Then, by the interpretation of the total variation distance in terms of coupling, we obtain
\begin{equation*}
d_{\TV} (\mu_k^*, \mu) \leqslant \P (X_k^* \neq X) = \sum_{d \geqslant 1} \mu(d) \P(B_{k,d} \neq d \mid B_{k,d} >0 ) \,,
\end{equation*}
with $B_{k,d} \sim \Bin(d, 1-\gamma_{k+1} )$.
As $\gamma_{k+1} \in [0,1]$ and $d \geqslant 1$, we have 
\begin{align*}
\P(B_{k,d} = d) & = (1-\gamma_{k+1})^d \geqslant 1-d \gamma_{k+1} \\
\P(B_{k,d} = 0) & = (\gamma_{k+1})^d \leqslant \gamma_{k+1} \leq \gamma_{k^*+1} \quad \text{ for } k\leq k^* \,,
\end{align*}
recalling also that $(\gamma_k)_{k\geq 0}$ is non-decreasing for the last inequality (see Lemma~\ref{MonoGamma}).
Since $\gamma_{k^*+1} = \bar \gamma_{\bar k^*-1} \leq 1-\alpha_1$ for some $\alpha_1>0$ (see Step~3 of the proof of Proposition~\ref{PhaseTransGammas}), we therefore have
\[
\P(B_{k,d} \neq d \mid B_{k,d} >0 ) = \dfrac{1-(1-\gamma_{k+1})^d}{1-(\gamma_{k+1})^d} \leqslant \frac{d}{\alpha_1}  \, \gamma_{k+1}\,,
\]
so $d_{\TV} (\mu_k^*, \mu) \leq \frac{1}{\alpha_1}\nu  \gamma_{k+1}$. 
This yields the desired upper bound thanks to Proposition~\ref{PhaseTransGammas}; see also the general bound~\eqref{eq:boundsparameters}.

\smallskip
For the case $k\geq k^*$, note that for any coupling $(X_k^*, Y)$ such that the marginal distributions are $X_k^* \sim \mu_{k}^*$ and $Y=1$ we have $d_{\TV} (\mu_k^*, \delta_1) \leqslant \P (X_k^* \neq 1)$.
Using again Lemma~\ref{ZeroTBin} that describes the law of $X_k^*$, we have
\[
\P (X_k^* \neq 1) = \sum_{d \geqslant 2} \mu(d) \P\big( B_{k,d} \geqslant 2 \mid B_{k,d} >0 \big),
\]
with $B_{k,d} \sim \Bin(d,1-\gamma_{k+1} )$. 
By sub-additivity and Bonferroni's inequality\footnote{For any events $(A_i)_{i\geq 1}$, we have $\P(\bigcup_{i=1}^n A_i) \geq \sum_{i=1}^n \P(A_i) - \sum_{1\leq i<j \leq n} \P(A_i\cap A_j)$.}, we have
\[
\P(B_{k,d} \geqslant 2) \leqslant \dfrac{d(d-1)}{2} (1-\gamma_{k+1})^2,
\qquad 
\P(B_{k,d} \geqslant 1) \geqslant d (1-\gamma_{k+1}) - \dfrac{d(d-1)}{2} (1-\gamma_{k+1})^2.
\]
Thus, we have the bound
\[
\P(B_{k,d} \geqslant 2 \mid B_{k,d} >0  )  \leqslant \dfrac{(d-1)(1-\gamma_{k+1})^2}{ (1-\gamma_{k+1}) (2-(d-1)(1-\gamma_{k+1}))} \,.
\]
This gives that $\P(B_{k,d} \geqslant 2 \mid B_{k,d} >0  ) \leq  (d-1) (1-\gamma_{k+1}) $ for $d \leqslant (1-\gamma_{k+1})^{-1}$, using also that $(1-\gamma_{k+1})_{k\geq 0}$ is non-increasing.
Bounding $\P(B_{k,d} \geqslant 2 \mid B_{k,d} \geqslant 1 )$ by $1$ when $d > (1-\gamma_{k+1})^{-1}$, we get
\begin{align*}
\sum_{d \geqslant 2} \mu(d) \P(B_{k,d}  \geqslant 2 \mid B_{k,d} >0 ) 
 & \leqslant \sum_{d=2}^{(1-\gamma_{k+1})^{-1}} \mu(d) d (1-\gamma_{k+1}) + \sum_{d > (1-\gamma_{k+1})^{-1}} \mu(d) \\
&  \leqslant \nu (1-\gamma_{k+1})+ \P\big(X > (1-\gamma_{k+1})^{-1} \big) \leq 2 \nu (1-\gamma_{k+1} )\,,
\end{align*}
where we have used Markov's inequality in the last term.
This gives the desired upper bound, thanks to Proposition~\ref{PhaseTransGammas}.
\qed

\subsubsection{Proof of Lemmas~\ref{lem:boundmk} and~\ref{lem:BoundVar}}

%
%(Remove?) For later reference, if $\hat B \sim\widehat{\Bin}(n,p)$, its mean and variance are given by
%\begin{equation}
%\E\big[ \hat{B} \big]  = \dfrac{np}{1-(1-p)^n}, \qquad 
%\Var( \hat{B}) =\dfrac{np}{(1-(1-p)^n)^2}  \big(1-p - (1-p+np)(1-p)^n\big).
%\end{equation}

First of all, let us stress that one can easily obtain an explicit formula for $\nu_k^*$: we have
\begin{equation}
\label{formulanuk}
\nu_k^*  :=\E [X_k^*] = \dfrac{1-\gamma_{k+1}}{1-\gamma_{k}} \nu \,.
\end{equation}

\begin{proof}[Proof of Equation~\eqref{formulanuk}]
Letting $B_k\sim \mathrm{Bin}(X, 1-\gamma_{k+1})$, we have thanks to Lemma~\ref{ZeroTBin} that
\begin{equation*}
\nu_k^* = \sum_{z\geq 1} z \P(X_k^*=z) = \sum_{z\geq 1} z \dfrac{\P(B_k=z)}{\P(B_k >0)} = \dfrac{1}{1-\gamma_{k}} \sum_{z \geq 1} z \P(B_k=z),
\end{equation*}
where we have used that $\P(B_k = 0) =\gamma_k$ so $\P(B_k>0)=1-\gamma_k$, see the proof of Lemma~\ref{ZeroTBin}. 
Working with the definition of $B_k$, we easily see that 
\[
\sum_{z\geq 1} z \P(B_k=z) = \sum_{d \geq 1} \mu(d) \E\big[\Bin(d, 1-\gamma_{k+1})\big] = (1-\gamma_{k+1} )\sum_{d \geq  0} \mu(d) d =  (1-\gamma_{k+1}) \nu \,,
\]
which gives~\eqref{formulanuk}.
%For the expression of the variance, let us compute the second moment of $X_k^*$: for $0 \leqslant k < n$, similarly as above we have
%\[
%\begin{split}
%\E\big[ (X_k^*)^2 \big] =  \dfrac{1}{1-\gamma_{k}} \sum_{z \geqslant 1} z^2 \P(B_k=z)
%& = \dfrac{1}{1-\gamma_{k}} \sum_{d \geqslant 0} \mu (d) \E \big[\Bin(d, 1-\gamma_{k+1})^2 \big] \\
%& = \dfrac{1}{1-\gamma_{k}}  \big( \gamma_{k+1}(1-\gamma_{k+1}) \nu +(1-\gamma_{k+1})^2 m_2 \big).
%\end{split}
%\]
%Thus, by computing  $\E [(X_k^*)^2] - \E [X_k^*]^2 $, we get the following expression for the variance:
%\[
%(\sigma_k^*)^2= \dfrac{1-\gamma_{k+1}}{1-\gamma_{k}} \left(\gamma_{k+1} \nu + (1-\gamma_{k+1}) m_2  - \dfrac{1-\gamma_{k+1}}{1-\gamma_{k}} \nu^2 \right) \,,
%\]
%which gives the correct expression, using the formula for $\nu_k^*$.
\end{proof}

\begin{proof}[Proof of Lemma~\ref{lem:boundmk}]
Our starting point is the formula~\eqref{formulanuk}.

Let us begin with the case $k \leq k^*$.
First of all, using that the parameters $(\gamma_k)_{0 \leqslant k < n}$ are non-decreasing (see Lemma~\ref{MonoGamma}), we get from the expression~\eqref{formulanuk} of $\nu_k^*$ the easy bound $\nu_k^* \leqslant \nu$.
On the other hand, using that $1-\gamma_{k}\leq 1$, we get from~\eqref{formulanuk} that
$\nu_k^* \geqslant \nu (1-\gamma_{k+1})$.
Applying the bound in Proposition \ref{PhaseTransGammas}, we then obtain the upper bound in~\eqref{Boundmk-2}.

For the case $k\geq k^*$, recalling that $1-\bar \gamma_{k+1} = F(1-\bar \gamma_k)$ with $F(t) =1-G(1-t)$, we have thanks to~\eqref{BornesF} 
\[
\nu (1-\gamma_{k+1}) \big(1- C_{\mu} (1-\gamma_{k+1}) \big)  \leqslant (1-\gamma_{k}) \leqslant \nu (1-\gamma_{k+1}).
\]
Therefore, using the expression~\eqref{formulanuk} for $\nu_k^*$, we get that 
\begin{equation}
\label{boundsDelta/delta}
1 \leqslant \nu_k^* = \dfrac{1-\gamma_{k+1}}{1-\gamma_{k}}  \nu \leqslant  \dfrac{1}{1-C_{\mu} (1-\gamma_{k+1})} \leq 1+ 2 C_{\mu} (1-\gamma_{k+1}) \,,
\end{equation}
for all $k$ such that $C_{\mu} (1-\gamma_{k+1}) \leq \frac12$, using that $(1-x)^{-1} \leq 1+2x$ for $x\in [0,\frac12]$.
Hence, using the bound in Proposition~\ref{PhaseTransGammas}, this yields \eqref{Boundmk-1} for $k \geq k_2^* := \min \{k \,,   1-\gamma_{k+1} \leq 1/(2C_{\mu})\}$.
Similarly as in the proof of Step~2 of Proposition~\ref{PhaseTransGammas}, one easily gets that $k_2^* \leq k^* +L_1'$ for some constant $L_1'$ that depends only on the law $\mu$; adapting the constant, this yields~\eqref{Boundmk-1}.
\end{proof}

\begin{proof}[Proof of Lemma~\ref{lem:BoundVar}]
Our starting point is the following computation, analogous to the one in the proof of~\eqref{formulanuk}: letting $B_k\sim \mathrm{Bin}(X, 1-\gamma_{k+1})$, we have thanks to Lemma~\ref{ZeroTBin} that
\[
\E\big[ (X_k^*)^q \big] = \frac{1}{1-\gamma_k} \sum_{z\geq 1} z^q \P(B_k=z) = \frac{1}{1-\gamma_k} \sum_{d\geq 0}  \mu(d)  \E\big[ \mathrm{Bin}(d, 1-\gamma_{k+1})^q \big] \,.
\]
Now, by Lemma~\ref{lem:Neveu}, since $\mathrm{Bin}(d, 1-\gamma_{k+1})$ is a sum of $d$ independent Bernoulli random variables (each of $q$-variance bounded by $1-\gamma_{k+1}$), we obtain that 
\[
\E\big[ \mathrm{Bin}(d, 1-\gamma_{k+1})^q \big] \leq \E\big[ \mathrm{Bin}(d, 1-\gamma_{k+1}) \big]^q + d (1-\gamma_k) = d^q (1-\gamma_{k+1})^q + d(1-\gamma_{k+1})\,.
\]
We therefore end up with
\[
\E\big[ (X_k^*)^q \big]  \leq \frac{1-\gamma_{k+1}}{1-\gamma_k} (1-\gamma_{k+1})^{q-1} m_q + \frac{1-\gamma_{k+1}}{1-\gamma_k} \nu \leq (1-\gamma_{k+1})^{q-1} m_q + \nu_k^* \,,
\]
where we have used $1-\gamma_{k+1}\leq 1-\gamma_k$ (by Lemma~\ref{MonoGamma}) and the formula~\eqref{formulanuk}.
We therefore end up with the bound
\[
\sigma_{q,k}^* := \E\big[ (X_k^*)^q \big] - (\nu_k^*)^q \leq (1-\gamma_{k+1})^{q-1} m_q  \,,
\]
using also that $\nu_k^* - (\nu_k^*)^q \leq 0$ since $\nu_k^* \geq 1$.

This readily proves that $\sigma_{q,k}^* \leq m_q$ for any $k\geq 0$ and gives the bound~\eqref{Boundvark-2} for $k\geq k^*$ by applying Proposition \ref{PhaseTransGammas}.
%
%For $k \geqslant k^*$, using Lemma~\ref{lem:boundmk}, we get that $\nu_k^*\leq m_1$ and $1-\nu_k^*\leq 0$, so
%\[
%(\sigma_k^*)^2 \leq (1-\gamma_{k+1}) (m_2-\nu) \,,
%\]
%and we directly get the upper bound~\eqref{Boundvark-1} thanks to Proposition~\ref{PhaseTransGammas}.
%%adapting the indexation and switching to (blabla)
%
%
%For $k \leqslant k^*$, setting $\gep_k := \nu- \nu_{k}^* \geq 0$, we have (bounding also $\nu_k^*\leq \nu$),
%\[
%(\sigma_{k}^*)^{2} \leq  \nu \Big(1-\nu +\gep_k + \frac{m_2}{\nu} -1  -\gamma_{k+1} \Big( \frac{m_2}{\nu} -1  \Big)  \Big)
%\leq \sigma^2 + \gep_k \nu \,,
%\]
%where we have used that $\frac{m_2}{\nu} -1 \geq 0$.
%Using Lemma~\ref{lem:boundmk}, we get the desired upper bound.
%For the lower bound, we write similarly
%\[
%(\sigma_{k}^*)^{2} \geq  (\nu- \gep_k) \Big(1-\nu + \frac{m_2}{\nu} -1  -\gamma_{k+1} \Big( \frac{m_2}{\nu} -1  \Big)  \Big)
%\geq  \sigma^2 - \frac{\sigma^2}{\nu}\gep_k -\gamma_{k+1} (m_2 -\nu) \,.
%\]
%Using Lemma~\ref{lem:boundmk} and Proposition~\ref{PhaseTransGammas}, this concludes the proof.
\end{proof}

\subsubsection{Proof of Lemma~\ref{lem:boundMk}}
Notice that the expression of $\nu_i^{*}$ in~\eqref{formulanuk} yields the following expression of $M_{0,k}^*$: for all $ 0\leq k\leq n$,
\[
M_{0,k}^* = \prod_{i=0}^{k-1} \dfrac{1-\gamma_{i+1}}{1-\gamma_{i}} \nu = \dfrac{1-\gamma_{k}}{1-\gamma_{0}}\,\nu^k.
\]

For $k \leq  k^*$, as the parameters $(\gamma_k)_{0 \leqslant k < n}$ are non-decreasing (see Lemma~\ref{MonoGamma}) and $1-\gamma_{0}\leq10$,
we have
\[
(1-\gamma_k) \,\nu^k \leq M_{0,k}^* = \dfrac{1-\gamma_{k}}{1-\gamma_{0}} \, \nu^{k} \leqslant \nu^{k}.
\]
Applying the bound in Proposition \ref{PhaseTransGammas} for $1-\gamma_k$ (or the general bound~\eqref{eq:boundsparameters}), we get the desired estimate.

For $k \geq k^*$, the upper bound of $1-\gamma_{k}$ from Proposition~\ref{PhaseTransGammas} yields
\[
M_{0,k}^* \leq \nu^{k} \dfrac{1}{\nu^{k-k^*}} \dfrac{1}{1-\gamma_{0}}.
\]
Since $\gamma_{0}$ is close to $0$ (see e.g.\ Proposition~\ref{PhaseTransGammas}), we have the right-hand side of \eqref{BoundMk-1}. 
For the left-hand side, using $1-\gamma_{0} \leq 1$ and Proposition~\ref{PhaseTransGammas}, we have 
\[
M_{0,k}^* \geqslant  \nu^{k} (1- \gamma_{k}) \geq c_1  \nu^{k^*} \,,
\]
which concludes the proof.
\qed

\paragraph*{Acknowledgements.}
The authors warmly thank Remco van der Hofstad and Arnaud Le Ny for many enlightening discussions.
They also thank the anonymous referee for his or her thorough comments which helped improve the presentation of the paper.
I.A.V. acknowledges the support of CNRS IRP Bezout - Eurandom.
Q.B.\ acknowledges the support of grant ANR-22-CE40-0012.

\bibliographystyle{amsplain}
\bibliography{bibPrunedTree.bib}

\providecommand{\bysame}{\leavevmode\hbox to3em{\hrulefill}\thinspace}
\providecommand{\MR}{\relax\ifhmode\unskip\space\fi MR }
\providecommand{\MRhref}[2]{%
  \href{http://www.ams.org/mathscinet-getitem?mr=#1}{#2}
}
\providecommand{\href}[2]{#2}
\begin{thebibliography}{10}

\bibitem{AVB24}
Irene Ayuso~Ventura and Quentin Berger, \emph{Non-linear conductances of galton--watson trees and application to the (near) critical random cluster model}, preprint arXiv:2404.11564 [math.PR] (2024).

\bibitem{DemboBasak17}
Anirban Basak and Amir Dembo, \emph{{Ferromagnetic Ising measures on large locally tree-like graphs}}, The Annals of Probability \textbf{45} (2017), no.~2, 780 -- 823.

\bibitem{BRZ95}
Pavel~M Bleher, Jean Ruiz, and Valentin~A Zagrebnov, \emph{On the purity of the limiting gibbs state for the ising model on the bethe lattice}, Journal of Statistical Physics \textbf{79} (1995), 473--482.

\bibitem{BlekGani1991}
PM~Blekher and NN~Ganikhodgaev, \emph{On pure phases of the ising model on the bethe lattices}, Theory of Probability \& Its Applications \textbf{35} (1991), no.~2, 216--227.

\bibitem{Bodineau2006}
Thierry Bodineau, \emph{Translation invariant {G}ibbs states for the ising model}, Probability theory and related fields \textbf{135} (2006), 153--168.

\bibitem{Bovier}
Anton Bovier, \emph{Statistical {Mechanics} of {Disordered} {Systems}: {A} {Mathematical} {Perspective}}, Cambridge {Series} in {Statistical} and {Probabilistic} {Mathematics}, Cambridge University Press, Cambridge, 2006.

\bibitem{can2017critical}
Van~Hao Can, \emph{Critical behavior of the annealed {I}sing model on random regular graphs}, Journal of Statistical Physics \textbf{169} (2017), 480--503.

\bibitem{can2022annealed}
Van~Hao Can, Cristian Giardin{\`a}, Claudio Giberti, and Remco van~der Hofstad, \emph{Annealed {I}sing model on configuration models}, Annales de l'Institut Henri Poincar\'e (B) Probabilit\'es et Statistiques \textbf{58} (2022), no.~1, 134--163.

\bibitem{CCST86}
JT~Chayes, L~Chayes, James~P Sethna, and DJ~Thouless, \emph{A mean field spin glass with short-range interactions}, Communications in Mathematical Physics \textbf{106} (1986), no.~1, 41--89.

\bibitem{ChenHuLin}
Dayue Chen, Yueyun Hu, and Shen Lin, \emph{Resistance growth of branching random networks}, Electronic Journal of Probability \textbf{23} (2018), 1--17.

\bibitem{CKLN23}
Loren Coquille, Christof Kuelske, and Arnaud Le~Ny, \emph{Continuity of the extremal decomposition of the free state for finite-spin models on cayley trees}, preprint arXiv:2310.11101 (2023).

\bibitem{CoqKulLeNy2023}
Loren Coquille, Christof K{\"u}lske, and Arnaud Le~Ny, \emph{Extremal inhomogeneous gibbs states for {SOS}-models and finite-spin models on trees}, Journal of Statistical Physics \textbf{190} (2023), no.~4, 71.

\bibitem{CNS21}
Cl{\'e}ment Cosco, Shuta Nakajima, and Florian Schweiger, \emph{Asymptotics of the $p$-capacity in the critical regime}, preprint arXiv:2112.03661 (2021).

\bibitem{DemboMontGibbs10}
Amir Dembo and Andrea Montanari, \emph{{Gibbs measures and phase transitions on sparse random graphs}}, Brazilian Journal of Probability and Statistics \textbf{24} (2010), no.~2, 137 -- 211.

\bibitem{DemboMontanari10}
Amir Dembo and Andrea Montanari, \emph{Ising models on locally tree-like graphs}, The Annals of Applied Probability \textbf{20} (2010), no.~2, 565--592.

\bibitem{DomGiaRvdH}
Sander Dommers, Cristian Giardin{\`a}, and Remco van~der Hofstad, \emph{Ising models on power-law random graphs}, Journal of Statistical Physics \textbf{141} (2010), no.~4, 638--660.

\bibitem{DomGiaRvdH12}
Sander Dommers, Cristian Giardin{\`a}, and Remco van~der Hofstad, \emph{Ising critical exponents on random trees and graphs}, Communications in Mathematical Physics \textbf{328} (2014), 355--395.

\bibitem{Dor08}
S.~N. Dorogovtsev, A.~V. Goltsev, and J.~F.~F. Mendes, \emph{Critical phenomena in complex networks}, Reviews of Modern Physics \textbf{80} (2008), no.~4, 1275--1335.

\bibitem{EndoVanEnterLeNy2021}
Eric~O Endo, Aernout~CD van Enter, and Arnaud Le~Ny, \emph{The roles of random boundary conditions in spin systems}, In and Out of Equilibrium 3: Celebrating Vladas Sidoravicius (2021), 371--381.

\bibitem{FriedliVelenik}
Sacha Friedli and Yvan Velenik, \emph{Statistical {Mechanics} of {Lattice} {Systems}: {A} {Concrete} {Mathematical} {Introduction}}, 1 ed., Cambridge University Press, November 2017.

\bibitem{GMR20}
Daniel Gandolfo, Christian Maes, Jean Ruiz, and Senya Shlosman, \emph{Glassy states: the free ising model on a tree}, Journal of Statistical Physics \textbf{180} (2020), 227--237.

\bibitem{GandRuizShlos2012}
Daniel Gandolfo, Jean Ruiz, and Senya Shlosman, \emph{A manifold of pure gibbs states of the ising model on a cayley tree}, Journal of Statistical Physics \textbf{148} (2012), 999--1005.

\bibitem{GGvdHP2015}
Cristian Giardin{\`a}, Claudio Giberti, Remco van~der Hofstad, and {Maria Luisa} Prioriello, \emph{Quenched central limit theorems for the {I}sing model on random graphs}, Journal of Statistical Physics \textbf{160} (2015), no.~6, 1623--1657.

\bibitem{griffiths67}
Robert~B Griffiths, \emph{Correlations in ising ferromagnets. {I}}, Journal of Mathematical Physics \textbf{8} (1967), no.~3, 478--483.

\bibitem{Higuchi1977}
Yasunari Higuchi, \emph{Remarks on the limiting {Gibbs} states on a $(d+1)$-tree}, Publications of the Research Institute for Mathematical Sciences \textbf{13} (1977), no.~2, 335--348.

\bibitem{Ioffe96}
Dmitry Ioffe, \emph{On the extremality of the disordered state for the ising model on the bethe lattice}, Letters in Mathematical Physics \textbf{37} (1996), 137--143.

\bibitem{KellySherman68}
Douglas~G Kelly and Seymour Sherman, \emph{General {G}riffiths' inequalities on correlations in {I}sing ferromagnets}, Journal of Mathematical Physics \textbf{9} (1968), no.~3, 466--484.

\bibitem{Lyons89}
Russell Lyons, \emph{The {Ising} model and percolation on trees and tree-like graphs}, Communications in Mathematical Physics \textbf{125} (1989), no.~2, 337--353.

\bibitem{LyonsPeres}
Russell Lyons and Yuval Peres, \emph{Probability on {Trees} and {Networks}}, Cambridge {Series} in {Statistical} and {Probabilistic} {Mathematics}, Cambridge University Press, Cambridge, 2017.

\bibitem{Montanari2013}
Andrea Montanari, \emph{Statistical mechanics and algorithms on sparse and random graphs}, Lectures on Probability Theory and Statistics. Saint-Flour (2013).

\bibitem{MontanariMosselSly}
Andrea Montanari, Elchanan Mossel, and Allan Sly, \emph{The weak limit of {Ising} models on locally tree-like graphs}, Probability Theory and Related Fields \textbf{152} (2012), 31--51.

\bibitem{neveu87}
Jacques Neveu, \emph{Multiplicative martingales for spatial branching processes}, Seminar on Stochastic Processes, 1987, Springer, 1987, pp.~223--242.

\bibitem{PemPer95}
Robin Pemantle and Yuval Peres, \emph{{G}alton-{W}atson trees with the same mean have the same polar sets}, The Annals of Probability (1995), 1102--1124.

\bibitem{PemPer10}
\bysame, \emph{The critical {Ising} model on trees, concave recursions and nonlinear capacity}, The Annals of Probability \textbf{38} (2010), no.~1, 184--206.

\bibitem{Rozikov}
Utkir~A. Rozikov, \emph{Gibbs {Measures} {On} {Cayley} {Trees}}, World Scientific, July 2013 (en).

\bibitem{vdHRandomGraphs}
Remco van~der Hofstad, \emph{Random graphs and complex networks, vol.~1}, Cambridge Series in Statistical and Probabilistic Mathematics, Cambridge University Press, 2017.

\bibitem{vdH21}
\bysame, \emph{The giant in random graphs is almost local}, preprint arXiv:2103.11733 (2021).

\bibitem{vdHRandomGraphs2}
\bysame, \emph{Random graphs and complex networks, vol.~2}, accessible on R. van der Hofstad website, 2023.

\bibitem{vanEnterNetocnySchaap2006}
Aernout~CD van Enter, Karel Neto{\v{c}}n{\`y}, and Hendrikjan~G Schaap, \emph{Incoherent boundary conditions and metastates}, Lecture Notes-Monograph Series (2006), 144--153.

\bibitem{Wu2000}
C.~Chris Wu, \emph{Ising models on hyperbolic graphs {II}}, Journal of Statistical Physics \textbf{100} (2000), 893--904.

\end{thebibliography}

\end{document}